\documentclass{article}

\usepackage{graphicx} 
\usepackage[margin=1in]{geometry}
\usepackage[T1]{fontenc}
\usepackage{times}
\usepackage{amsmath,amsfonts,amssymb,amsthm}
\usepackage{stmaryrd,authblk}
\usepackage{hyperref,cleveref}
\hypersetup{colorlinks=true, breaklinks=true, urlcolor=black, linkcolor=black, citecolor=black}
\usepackage{enumerate,enumitem,csquotes,doi}
\usepackage{color,todonotes,comment}
\usepackage{bm,bbm}

\newtheorem{theorem}{Theorem}[section]
\newtheorem{lemma}[theorem]{Lemma}
\newtheorem{proposition}[theorem]{Proposition}
\newtheorem{corollary}[theorem]{Corollary}
\newtheorem{definition}[theorem]{Definition}
\theoremstyle{remark}

\newtheorem{remark}{Remark}

\newcommand{\bN}{\mathbb{N}}
\newcommand{\bZ}{\mathbb{Z}}
\newcommand{\bR}{\mathbb{R}}
\newcommand{\bx}{\textbf{x}}
\newcommand{\bC}{\mathbb{C}}
\newcommand{\bT}{\mathbb{T}}
\newcommand{\cO}{\mathcal{O}}
\newcommand{\cv}[1]{\,\underset{#1}{\longrightarrow}\,}
\newcommand{\abs}[1]{{\left|#1\right|}}

\newcommand{\dist}{\mathrm{dist}}
\newcommand{\bD}{\mathbb{D}}
\newcommand{\pairs}{{\rm B}}
\newcommand{\ba}{\textbf{a}}
\newcommand{\cT}{\mathcal{T}}
\newcommand{\cD}{\mathcal{D}}
\newcommand{\cM}{\mathcal{M}}
\newcommand{\cX}{\mathcal{X}}
\newcommand{\cI}{\mathcal{I}}
\newcommand{\cJ}{\mathcal{J}}
\newcommand{\cP}{\mathcal{P}}
\newcommand{\cE}{\mathcal{E}}
\newcommand{\cH}{\mathcal{H}}
\newcommand{\cC}{\mathcal{C}}
\newcommand{\wR}{\widetilde{R}}
\newcommand{\fS}{\mathfrak{S}}
\newcommand{\fT}{\mathfrak{T}}
\newcommand{\bc}{\mathbf{c}}
\newcommand{\bj}{\mathbf{j}}
\newcommand{\bn}{\mathbf{n}}
\newcommand{\wbc}{\widetilde{\bc}}
\newcommand{\wbj}{\widetilde{\bj}}
\newcommand{\wbn}{\widetilde{\bn}}
\newcommand{\wbF}{\widetilde{\mathbf{F}}}
\newcommand{\be}{\mathbf{e}}
\newcommand{\by}{{\mathbf{y}}}
\newcommand{\br}{{\mathbf{r}}}
\newcommand{\bz}{\textbf{z}}
\newcommand{\boldZ}{\textbf{Z}}
\newcommand{\bE}{\mathbb{E}}
\newcommand{\expec}[1]{\mathbb{E}{\left[#1\right]}}
\newcommand{\expecond}[2]{\mathbb{E}{\left[\left. #1 \,\right|\, #2 \right]}}
\newcommand{\bP}{\mathbf{P}}
\newcommand{\prob}[1]{\mathbf{P}{\left(#1\right)}}
\newcommand{\probcond}[2]{\mathbf{P}{\left(\left. #1 \,\right|\, #2 \right)}}
\newcommand{\bV}{\mathbb{V}}
\newcommand{\bU}{\mathbb{U}}
\newcommand{\var}[1]{\mathrm{Var}{\left[#1\right]}}

\newcommand{\Law}[1]{\mathcal{L}{\left(#1\right)}}
\newcommand{\Lawcond}[2]{\mathcal{L}{\left(\left.#1 \,\right\vert\, #2\right)}}
\newcommand{\One}[1]{\mathbbm{1}_{#1}}
\newcommand{\Unif}[1]{\mathrm{Unif}{\left(#1\right)}}
\newcommand{\Normal}[2]{\mathcal{N}{\left(#1,#2\right)}}
\newcommand{\BinomialDistribution}[2]{\mathrm{Bin}{\left(#1,#2\right)}}
\newcommand{\BetaDistribution}[2]{\mathrm{Beta}{\left(#1,#2\right)}}
\newcommand{\PoissonDirichlet}[2]{\mathrm{PD}{\left(#1,#2\right)}}

\newcommand{\fp}{{\mathfrak{p}}}
\newcommand{\des}[1]{{\mathrm{des}{\left(#1\right)}}}
\newcommand{\red}{\mathrm{red}}
\newcommand{\rec}{\mathrm{rec}}
\DeclareMathOperator{\discrete}{disc}
\DeclareMathOperator{\rk}{rk}
\DeclareMathOperator{\id}{id}
\DeclareMathOperator{\St}{St}
\DeclareMathOperator{\AncestralLine}{AL}

\title{\textsc{Limits of descent-biased trees}}

\author[$1,2$]{Victor Dubach}
\author[$3$]{Paul Th\'evenin}
\author[$1,4$]{Stephan Wagner}
\affil[$1$]{Department of Mathematics, Uppsala University, Sweden}
\affil[$2$]{Institut Camille Jordan, Université Lyon 1, France}
\affil[$3$]{Université d’Angers, CNRS, LAREMA, F-49000 Angers, France}
\affil[$4$]{Institute of Discrete Mathematics, TU Graz, Austria}

\date{}

\begin{document}

\maketitle

\begin{abstract}
    We investigate scaling and local limits of random trees biased according to their number of descents.
    A descent in a rooted labeled tree $t$ is a parent-child pair such that the label of the parent is greater than the label of the child, and the total number of descents is denoted by $\des{t}$.
    For $n\geq 1$ and $q_n\ge0$, we consider the probability measure on trees of size $n$ where each tree $t$ is chosen with a probability proportional to $q_n^\des{t}$.
    We study the resulting random tree $\cT_n^{(q_n)}$ properly rescaled as $n\to\infty$, and focus on two regimes for the bias parameter $q_n$.
    When $q_n = q \in (0,1]$ is fixed, we prove that $\cT_n^{(q)}$ converges in distribution to the Brownian Continuum Random Tree.
    When $q_n = a/n$ for some fixed $a>0$, we prove that $\cT_n^{(a/n)}$ converges in distribution to a random non-trivial dendron constructed from a Poisson--Dirichlet sequence.
    %The base tree of this dendron is a random infinite discrete tree, and its sampling measure is a sum of Dirac masses supported on the vertices of the base tree.
    We complement these results with a description of the Benjamini--Schramm local limit of $\cT_n^{(q_n)}$ in all regimes of parameters.
    Our proofs rely on analytic combinatorics to find the asymptotics of certain statistics of the tree, and on a probabilistic analysis of the structure of descent-biased random trees.
\end{abstract}

\begin{figure}[h]
    \centering
    \includegraphics[width=0.3\linewidth]{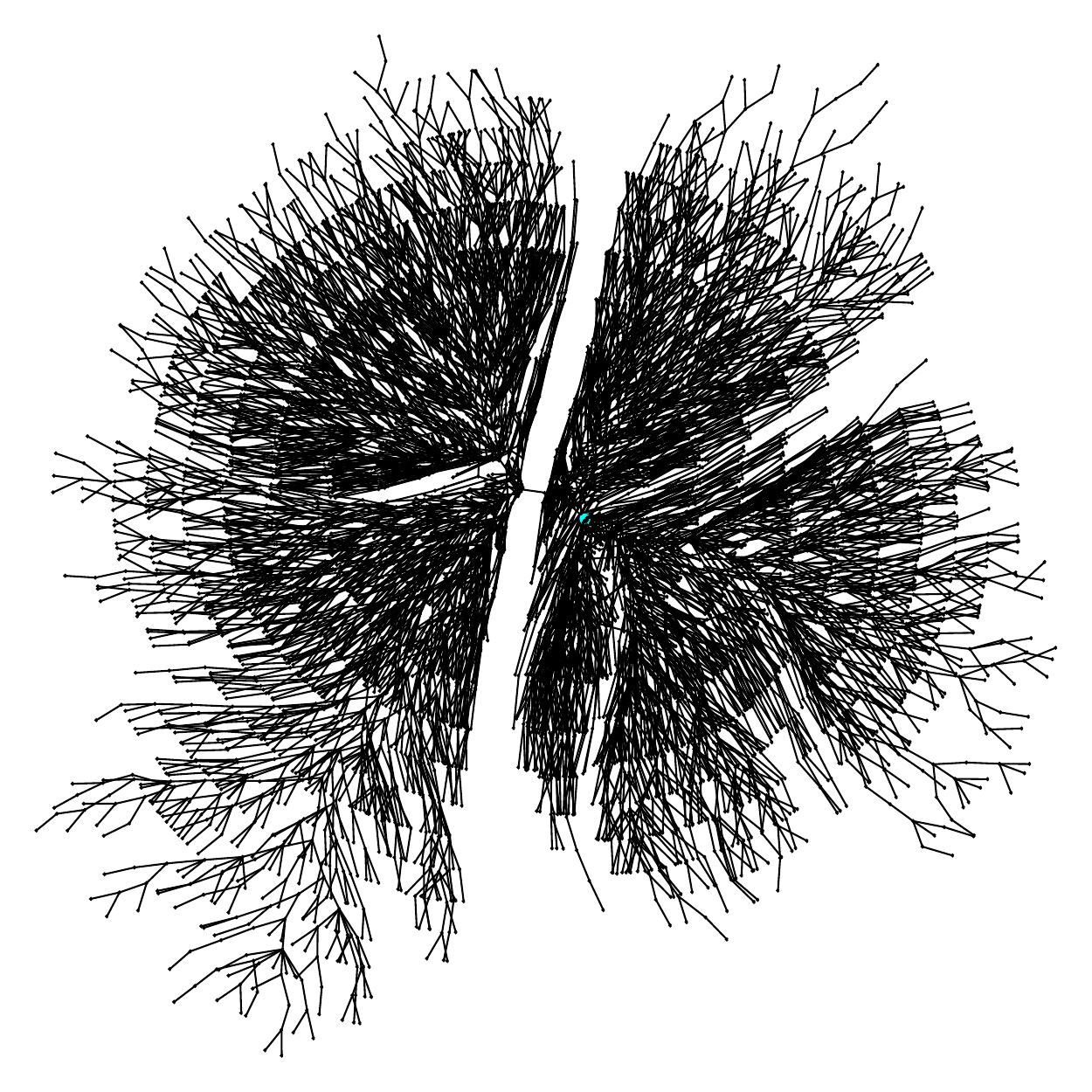}
    \includegraphics[width=0.3\linewidth]{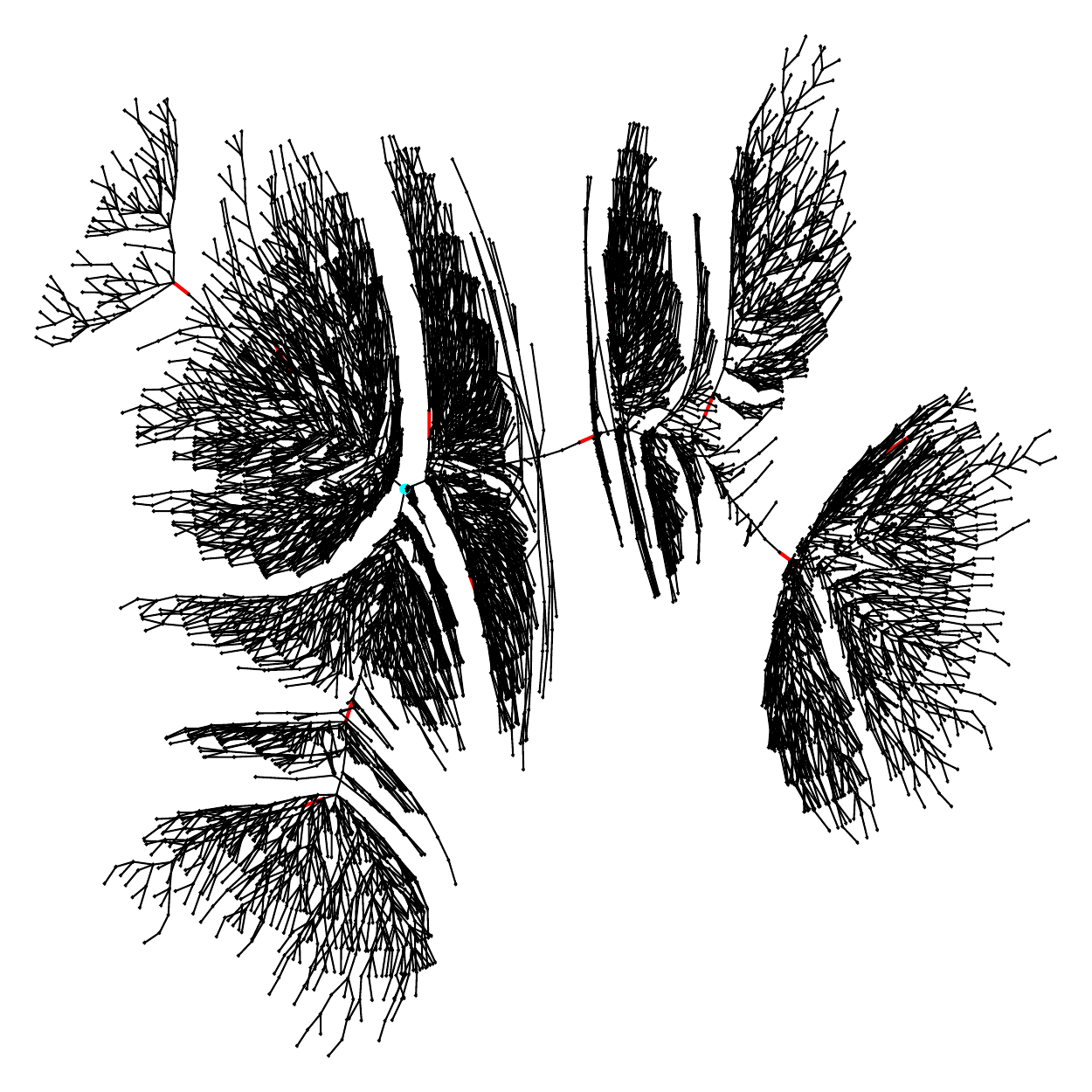}
    \includegraphics[width=0.3\linewidth]{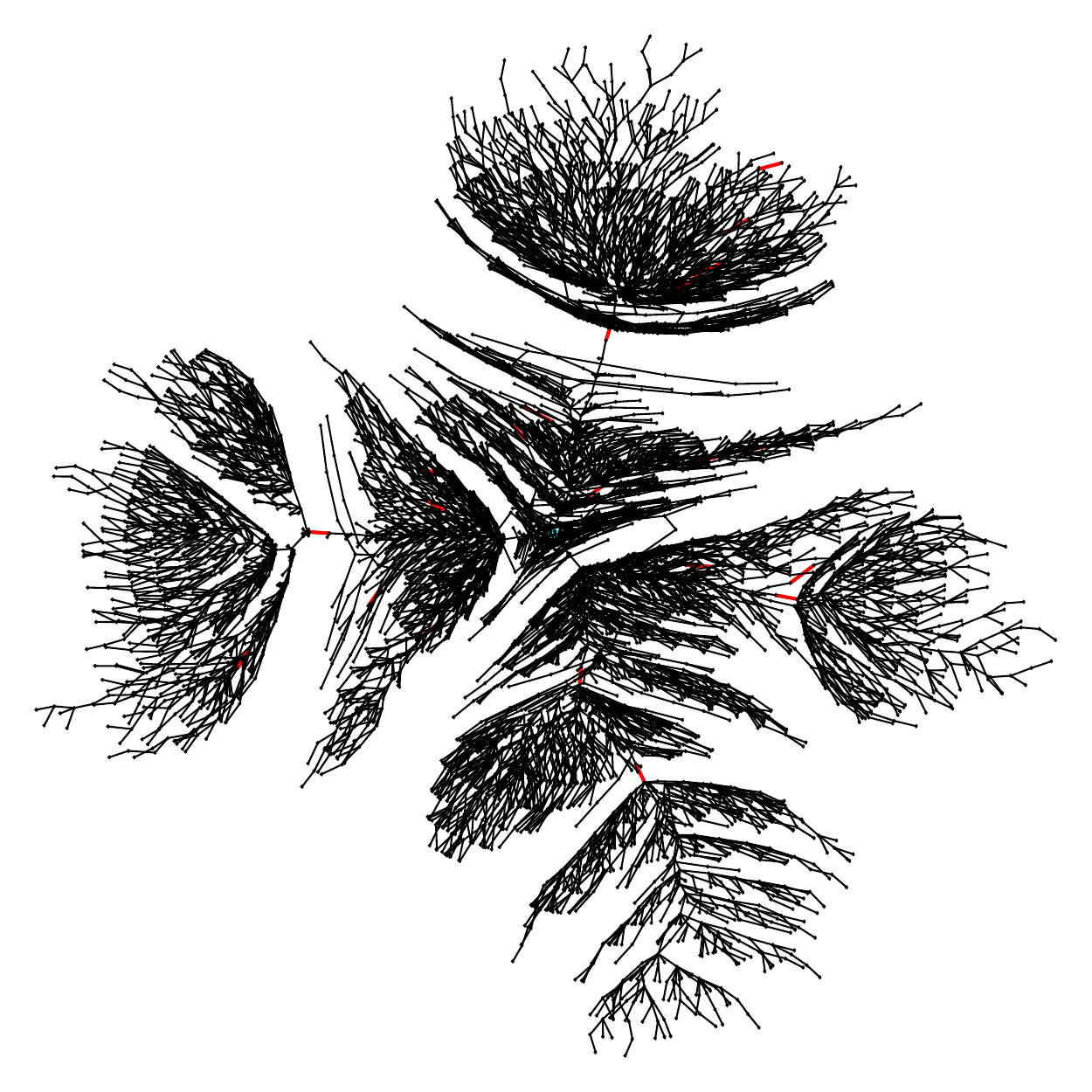}
    \includegraphics[width=0.3\linewidth]{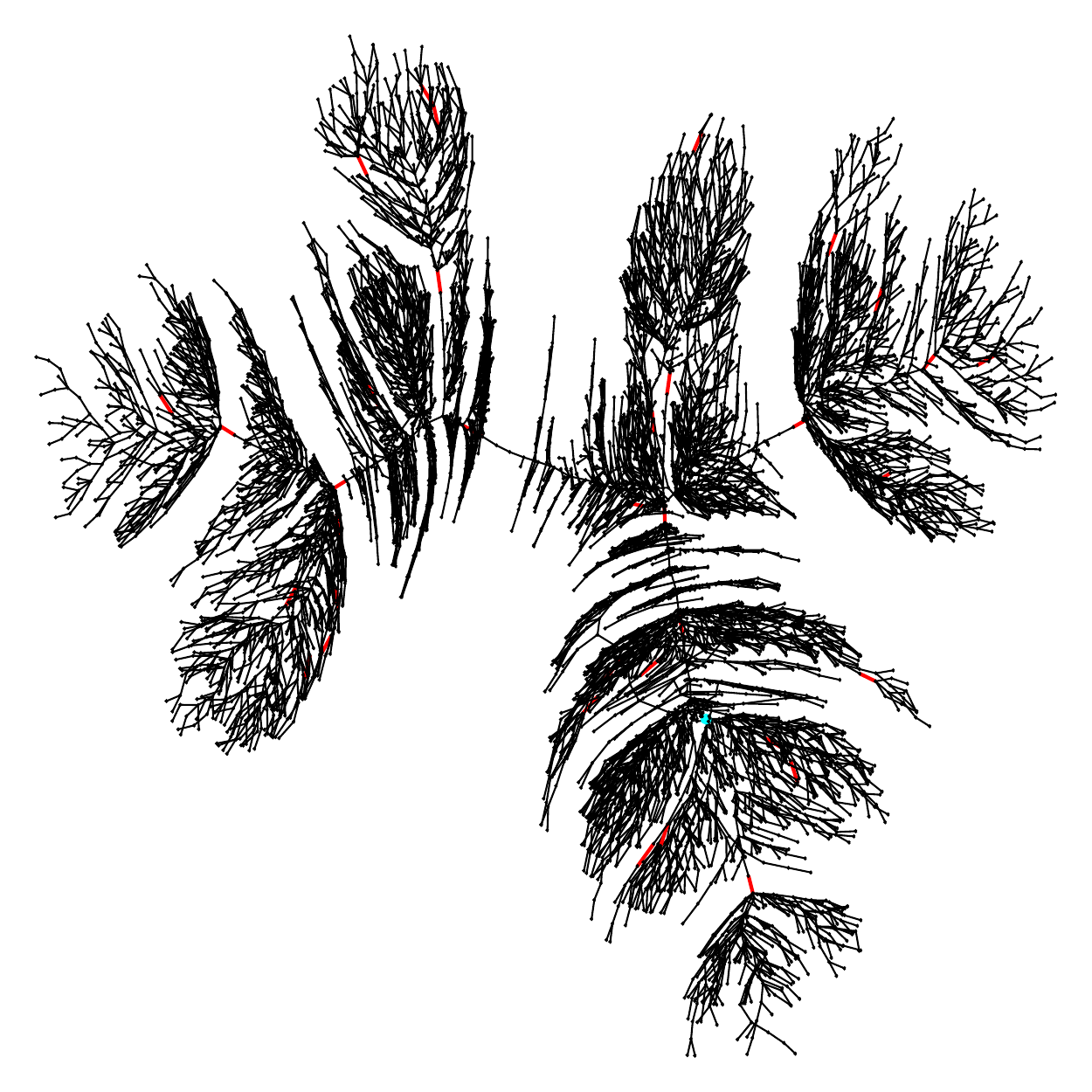}
    \includegraphics[width=0.3\linewidth]{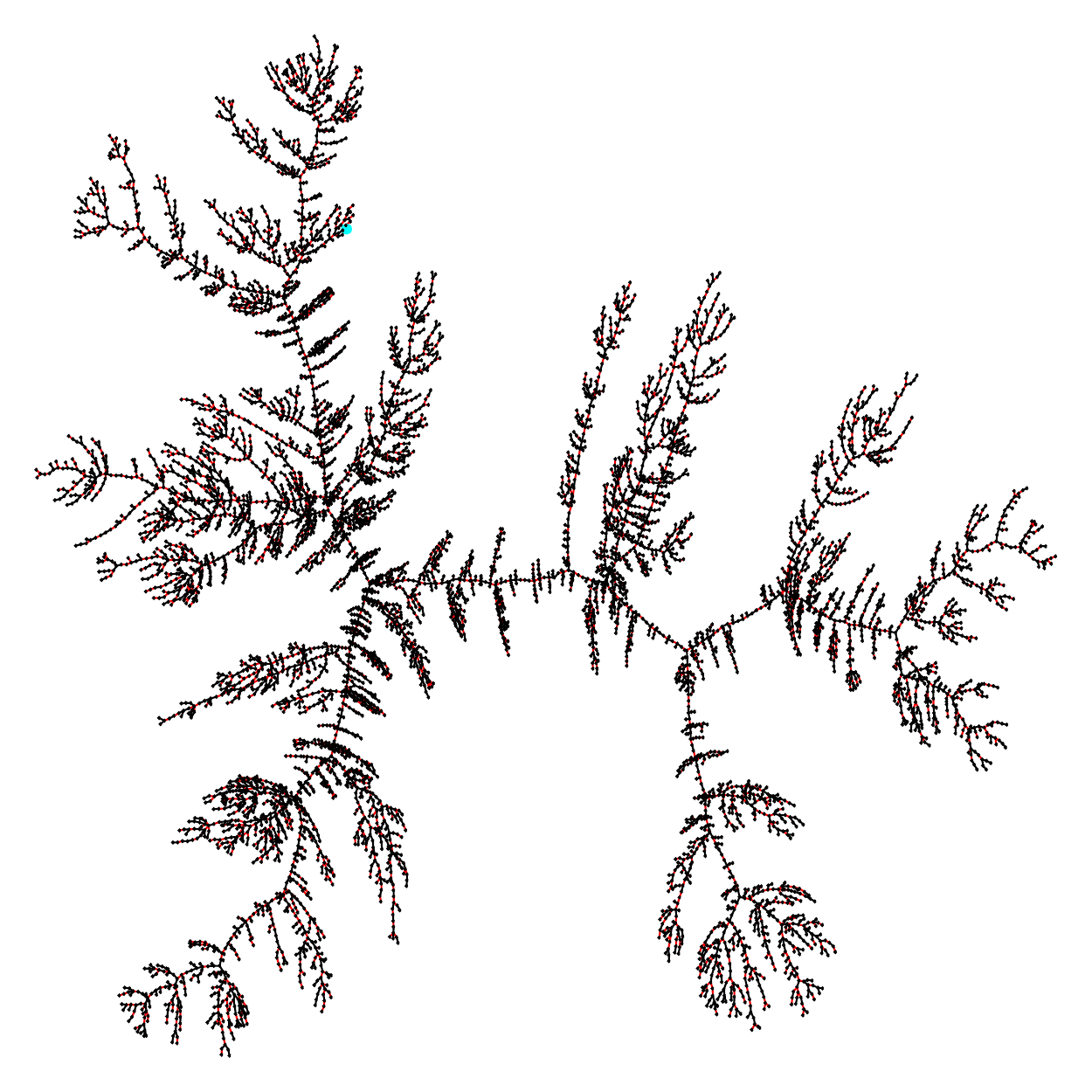}
    \includegraphics[width=0.3\linewidth]{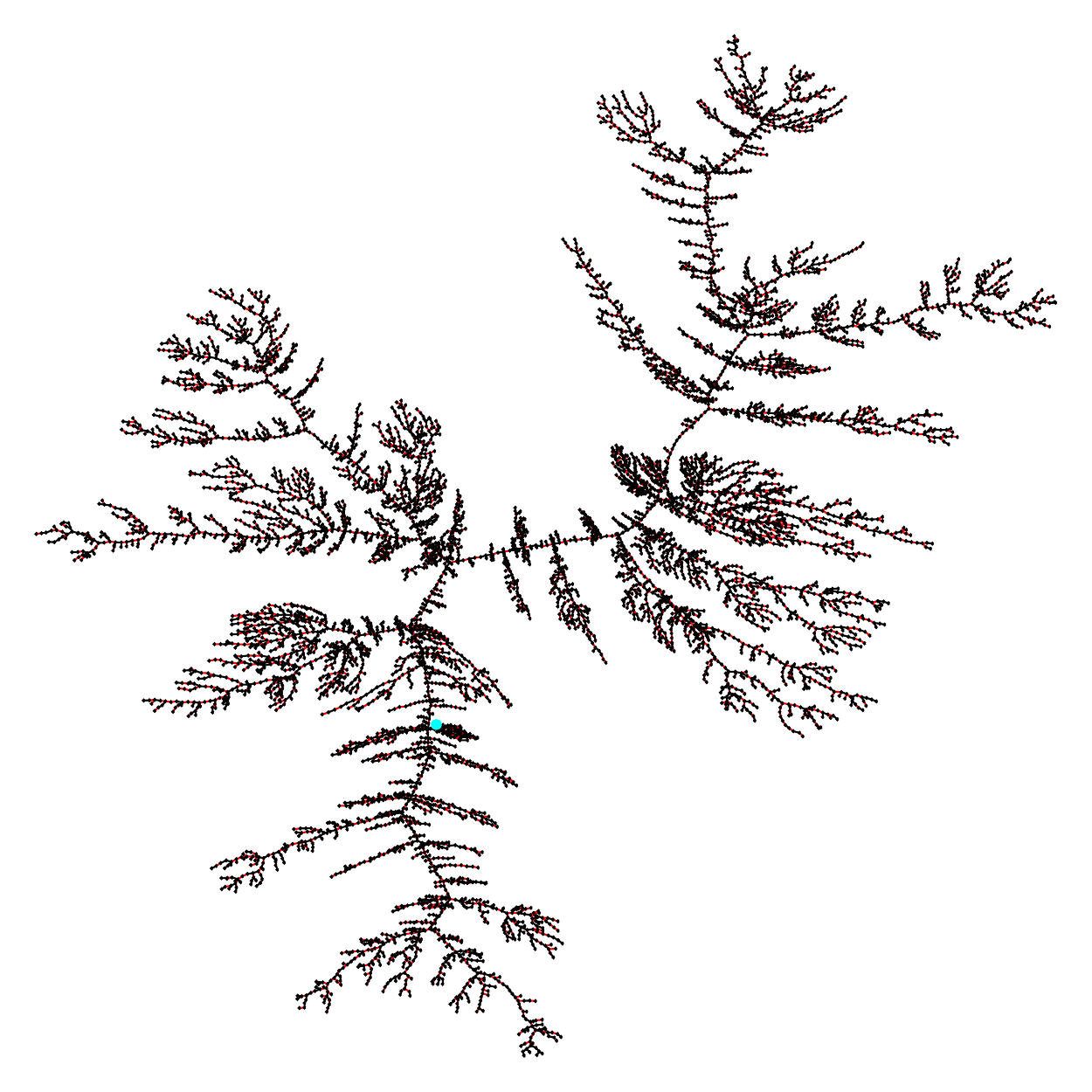}
    \caption{Simulations of descent-biased trees for $n=8000$ and varying values of the parameter $q$.
    From left to right and top to bottom: $q = 0, \frac{2}{n}, \frac{4}{n}, \frac{8}{n}, 0.5, 1$.
    The root is in blue and the descent edges are in red.
    }
    \label{fig: simulations}
\end{figure}

\newpage

\tableofcontents

\newpage

\section{Introduction}

\subsection{Background}

\paragraph{Our model.}
For $n\in\bN := \{1,2,\ldots \}$, we let $\bT_n$ denote the set of rooted unordered (also called non-plane) trees with $n$ vertices labeled from $1$ to $n$.
Recall that $\bT_n$ has cardinality $n^{n-1}$.
We say that a vertex $u$ is the parent of another vertex $v$ if there is an edge between $u$ and $v$, and $v$ is further away from the root than $u$.
A \emph{descent} in a tree $t\in\bT_n$ is a pair of vertices $(u,v)$ where $u$ is the parent of $v$ and the label of $u$ is greater than the label of $v$.
The total number of descents in $t$ is then denoted by $\des{t}$.
For $n\in\bN$ and $q\ge0$, we introduce a random tree $\cT_n^{(q)}$ defined as follows:
\begin{equation*}
    \text{for each } t\in\bT_n,\quad
    \prob{ \cT_n^{(q)} = t } = \frac{q^\des{t}}{Z_{n,q}}
\end{equation*}
where $Z_{n,q} := \sum_{t\in\bT_n} q^\des{t}$.
This model, called the \emph{$q$-descent-biased random tree of size~$n$}, was first introduced in \cite{Thevenin_Wagner_2023}.
From now on, we restrict our attention to $q\in[0,1]$.
This is because for any $q>0$, the tree obtained from $\cT_n^{(q)}$ by replacing each label $i$ with $n-i+1$ is distributed like a $(1/q)$-descent-biased random tree of size $n$.
Observe that the $q$-descent-biased distribution interpolates between random recursive trees (i.e., trees with increasing labels along the branches) for $q=0$ and rooted Cayley trees (i.e., uniformly random trees) for $q=1$.
%For $q<1$, the distribution of $\cT_n^{(q)}$ favors trees with fewer descents than the uniform model.
See \Cref{fig: simulations} for simulations of this model.

The weighted enumeration of these trees goes back to E\u{g}ecio\u{g}lu and Remmel \cite{ER86}, who proved the formula $Z_{n,q} = \prod_{k=1}^{n-1} (n-k+kq)$. 
The resulting random tree $\cT_n^{(q)}$ was introduced and studied in \cite{Thevenin_Wagner_2023} by the second and third author of this paper, who proved that $\cT_n^{(q)}$ admits a \enquote{one-ended} local limit at the root as $n\to\infty$ if $q\in(0,1]$ is fixed, see Remark \ref{rk:root local limit}.
In this paper, we are mainly interested in the dual problem of finding the scaling limit of $\cT_n^{(q)}$ as $n\to\infty$, when $q=q_n$ is allowed to depend on $n$.
In addition, we complete the local understanding by providing the Benjamini--Schramm local limit of $\cT_n^{(q)}$, that is, the local limit at a uniformly random vertex, in all regimes.

\paragraph{Scaling limits of random trees.}
Aldous proved in a seminal work \cite{Aldous_1993} that Cayley trees $(\cT_n^{(1)})_{n \geq 1}$
admit a scaling limit known as the Brownian Continuum Random Tree, or simply the CRT.
The CRT is a random \emph{real tree}, that is, a metric measure space with tree-like properties (see \Cref{sec: Gromov Prokhorov} for details on real trees).
The convergence of Cayley trees towards the CRT is known to hold in distribution for the so-called 
%Gromov--Prokhorov--Hausdorff topology.
%The convergence for the weaker 
\emph{Gromov--Prokhorov} (abbreviated GP; sometimes called Gromov--weak) topology,  
meaning that the pairwise distances of a finite collection of i.i.d.~random vertices in $\cT_n^{(1)}$ converge in distribution, after rescaling by $1/\sqrt n$, to the analog on the CRT.

Since the work of Aldous, the study of real trees and scaling limits of random trees has seen numerous developments.
Notably, the CRT has been found to be the scaling limit of many natural models of random trees, cementing it as a universal limit object for random trees.
Here, we contribute to this picture by proving that $\cT_n^{(q)}$ also admits the CRT as its scaling limit, as $n\to\infty$, when $q\in(0,1]$ is fixed.

On the other hand, the pairwise distances of a finite collection of i.i.d.~random vertices in a random recursive tree $\cT_n^{(0)}$ are known to 
behave in a very different way: 
most vertices are at distance $2\log n + \cO_\bP(\sqrt{\log n})$ from each other.
This entails that the scaling limit of random recursive trees cannot be a separable real tree.
To circumvent this, Elek and Tardos \cite{ET22} introduced the notion of \emph{dendron}, which consists of a base tree $T$ along with a mass measure on $T \times \mathbb{R}_+$. Informally, dendrons are like real trees, except that they allow for \enquote{clusters} at certain points that behave like recursive subtrees.
Any given discrete or real tree can be viewed as a dendron, and the GP topology can be extended to the space of dendrons.
In that broader framework, $\cT_n^{(0)}$ converges in probability, after rescaling the distances by $1/\log n$, to the trivial dendron $\Upsilon_{\delta_1}$ whose base tree is reduced to a unique vertex, and whose mass measure is a Dirac mass at height $1$ (using the notation of \cite{Janson21}).

The theory of dendrons is still quite young, and only a few non-trivial dendrons have so far been found as the scaling limits of random trees.
Here, we prove that, for fixed $a>0$, $\cT_n^{(a/n)}$ converges in distribution to a random non-trivial dendron $\cD_a$.
The base tree of this dendron is a random infinite discrete tree constructed according to a preferential attachment procedure (a $p$-tree with the terminology of \cite{Camarri_Pitman_2000,Blanc-Renaudie_2025}), and the sampling measure consists of a sum of Dirac masses at height $1$ on the vertices of this tree.
The Dirac masses coincide with the weights used in the preferential attachment procedure, which are given by a one-parameter Poisson--Dirichlet sequence.
To our knowledge, this is the first example of a limit dendron where the base tree is neither a point (e.g., for random recursive trees) nor a real tree with its sampling measure (e.g., for Cayley trees).

We refer the reader to \Cref{sec: Gromov Prokhorov} for more on real trees, dendrons, and the Gromov--Prokhorov topology.

\paragraph{Local limits of random trees.}
A sequence $(\cT_n)_{n \geq 1}$ of finite random trees converges locally in the Benjamini--Schramm sense if the finite neighborhoods of a uniform random point in $\cT_n$ converge in distribution. 
The Benjamini--Schramm limit of random recursive trees was characterized by Aldous \cite{Aldous_1991} and can be constructed from stopped Yule tree processes. 
On the other hand, the Benjamini--Schramm limit of uniform Cayley trees is called Kesten's tree, and was characterized by Kesten \cite{Kesten_1986}. 
One goal of this work is to fully characterize the Benjamini--Schramm limit of descent-biased trees in all regimes.
Namely, we show that the local limit of $\cT_n^{(q_n)}$ interpolates between the recursive limit when $q_n \to 0$ and Kesten's tree when $q_n \to 1$.
This complements both our scaling limit results, and the local limit at the root found in \cite{Thevenin_Wagner_2023}.

\subsection{Main results: the scaling limit}

We let $\cT_\be$ denote Aldous' Brownian Continuum Random Tree (CRT).
The most common way to construct the CRT is by using a Brownian excursion as its \enquote{contour function} 
(see, e.g., \cite{Aldous_1993})
but we shall not rely on that here.
The (random) metric, measure, and root of the CRT $\cT_\be$ are respectively denoted by $\dist$, $\mu$, and $\rho$.
If $T_n$ is a finite tree of size $n$, we let $\dist_n$ denote the graph distance on $T_n$, $\mu_n$ the uniform probability measure on its vertices, and $\rho_n$ its root.

Our first main result is that the descent-biased trees $\cT_n^{(q)}$, properly rescaled, converge in distribution towards $\cT_\be$ with respect to the GP topology (see \Cref{sec: Gromov Prokhorov} for precise definitions) if $q\in(0,1]$ is fixed and $n\to\infty$.

\begin{theorem}\label{th: CRT limit if q fixed}
    Fix $q\in(0,1]$.
    For each $n\in\bN$, let $\cT_n^{(q)}$ be a $q$-descent-biased random tree of size $n$.
    Let $W_{n,0} := \rho_n$ be the root of $\cT_n^{(q)}$ and $(W_{n,i})_{i\ge1}$ be a sequence of $\mu_n$-i.i.d.~random vertices of $\cT_n^{(q)}$.
    Also let $W_{0} := \rho$ be the root of $\cT_\be$ and $(W_{i})_{i\ge1}$ be a sequence of $\mu$-i.i.d.~random points on $\cT_\be$.
    Then, for any given $\ell\in\bN$, we have
    \begin{equation*}
        \left(\big. \frac{q-1}{\sqrt q \log q} \,\frac{1}{\sqrt n}
        \,\dist_n{\left( W_{n,i} , W_{n,j} \right)} \right)_{0\le i,j\le \ell}
        \cv{n\to\infty}
        \left(\big. \dist{\left( W_{i} , W_{j} \right)} \right)_{0\le i,j\le \ell}
    \end{equation*}
    in distribution and with joint moments.
    In particular,
    \begin{equation*}
        \frac{q-1}{\sqrt q \log q} \frac{1}{\sqrt n} \cdot \cT_n^{(q)}
        \cv{n\to\infty}
        \cT_\be
    \end{equation*}
    in distribution for the Gromov--Prokhorov topology.
    When $q=1$, we interpret the factor $\sqrt q \frac{\log q}{q-1}$ as $1$.
\end{theorem}

Our second main result concerns the regime $q_n = a/n$ for fixed $a>0$.
In that regime, it can be seen that the depth of a typical vertex in $\cT_n^{(a/n)}$ behaves, asymptotically, like a \emph{discrete} random variable times $\log n$ (\Cref{prop: asymptotic coefficient typical depth in transition regime}).
Going beyond this behavior, we describe the full dendron limit of $\cT_n^{(a/n)}$.

\begin{theorem}\label{th: dendron limit if a/n}
    Fix $a>0$.
    For each $n\in\bN$, let $q_n=a/n$ and let $\cT_n^{(a/n)}$ be a $q_n$-descent-biased random tree of size $n$.
    There exists an explicit random dendron $\cD_a$ such that
    \begin{equation*}
        \frac{1}{\log n} \cdot \cT_n^{(a/n)}
        \cv{n\to\infty}
        \cD_a
    \end{equation*}
    in distribution for the dendron topology.
    The dendron $\cD_a$ is constructed in \Cref{def: size-biased dendron}.
\end{theorem}

The topological space of dendrons is recalled in \Cref{sec: Gromov Prokhorov}.
The dendron $\cD_a$ is described in \Cref{ssec: limit dendron convergence}:
its base tree is a random infinite discrete tree, constructed with a preferential attachment procedure, and the sampling measure of $\cD_a$ is a sum of Poisson--Dirichlet Dirac masses at height $1$ on the vertices of this discrete tree.
Roughly speaking, this means that $\cT_n^{(a/n)}$ consists of random recursive trees with random sizes, grafted one onto another like a random \enquote{$p$-tree} (see, e.g., \cite{Camarri_Pitman_2000,Blanc-Renaudie_2025}).
The base tree of $\cD_a$ describes how these recursive ``components'' are assembled together, and the Poisson--Dirichlet masses describe their sizes.

We also show that the window $q_n=a/n$ for descent-biased trees pinpoints the transition between the recursive limit and the CRT limit, in the following sense.

\begin{theorem}\label{th: our dendron cv to CRT}
    The dendrons $\cD_a$, $a \in \bR_+^*$, interpolate between the trivial dendron $\Upsilon_{\delta_1}$ and the CRT $\cT_\be$. 
    More precisely, in the sense of dendrons, we have
    \begin{equation*}
        (i)\quad 
        \cD_a \cv{a\to0} \Upsilon_{\delta_1} 
        \qquad\text{and}\qquad
        (ii)\quad 
        \frac{1}{\sqrt a} \cdot \cD_a  \cv{a\to\infty} \cT_\be
    \end{equation*}
    in distribution.
\end{theorem}

Note, however, that we do not state a scaling limit result for $\cT_n^{(q_n)}$ in the regime $q_n\to0$, $nq_n\to\infty$.
We conjecture that the scaling limit in that regime, with the same rescaling as in \Cref{th: CRT limit if q fixed}, should be the CRT; see \Cref{sec: Discussion}.

\subsection{Complementary results: the Benjamini--Schramm local limit}

Let us recall the mode of convergence introduced by Benjamini and Schramm in \cite{Benjamini_Schramm_2001} in the case of random trees.
For a locally finite tree $T$ (that is, a tree whose degrees are all finite), a vertex $u$ of $T$ and $r \geq 0$, we let $B_r(T,u)$ denote the ball of radius $r$ around $u$ with respect to the graph distance in $T$ (where all edges have length $1$).

\begin{definition}
\label{def:bs local limit}
Let $(\cT_n)_{n \geq 1}$ be a sequence of random finite trees, and let $u_n$ denote a uniformly random vertex in $\cT_n$. 
Let $(\cT_*,\rho_*)$ be a random rooted --- possibly infinite --- locally finite tree. 
We say that $\cT_n$ converges locally to $(\cT_*,\rho_*)$ if, for every fixed $r \geq 0$, we have the convergence in distribution
\begin{align*}
B_r(\cT_n,u_n) \overset{(d)}{\cv{n\to\infty}} B_r(\cT_*,\rho_*) \,.
\end{align*}
\end{definition}

The Benjamini--Schramm local limits of Cayley trees $\cT_n^{(1)}$ and of random recursive trees $\cT_n^{(0)}$ are well-known.
In the recursive case, the limit $\cT_*^{(0)}$ can be constructed from stopped Yule processes, and was first characterized by Aldous in \cite{Aldous_1991} (see, also, \cite[Example 6.1]{Holmgren_Janson_2017} and \cite{Contat_Laulin_2025}). It was also described in a more combinatorial way in \cite[Appendix A]{Dadedzi_Wagner_2024}.
In the Cayley case, the limit $\cT_*^{(1)}$ is known as Kesten's tree \cite{Kesten_1986} and consists of an infinite spine to which random i.i.d.\ finite trees are grafted, each one distributed as a $\mathrm{Poi}(1)$ Bienaymé--Galton--Watson tree.
We unify these two results by introducing a random rooted tree $(\cT_*^{(q)}, \rho_*^{(q)})$ for $q\in[0,1]$, which naturally interpolates between the previous two cases.

\begin{definition}%[The grafted trees $T_c, c \in {[0,1]}$]
\label{def:grafted trees}
    For any $q\in[0,1]$ and $\chi \in [0,1]$, define a random variable $Y_{q,\chi}$ taking values in $\bN^2$, such that for all $m \geq i \geq 1$,
    \begin{align*}
    \prob{Y_{q,\chi} = (m,i)} = f_{m,i}^{(q)}(\chi)
    \end{align*}
    where the function $f_{m,i}^{(q)}$ is defined later by \eqref{eq: function fmi}.
    For all $M\ge i \ge 1$, let $\cT^{(q)}_{M,i}$ denote a $q$-descent-biased tree conditioned to have size $M$ and root label $i$. 
    Finally, define the (a.s.\ finite) random tree $\cT^{(q)}_\chi$ as the mixture of conditioned $q$-descent-biased trees:
    \begin{align*}
    \Lawcond{\cT^{(q)}_\chi }{ Y_{q,\chi}=(M,i) } = \Law{ \cT^{(q)}_{M,i} } \,.
    \end{align*}
\end{definition}

We can now define $\cT_*^{(q)}$.

\begin{definition}%[The tree $\cT_*^{(q)}$]
\label{def:t etoile q}
    For $q \in [0,1]$, consider the kernel 
    \[
        \kappa_q(\mathrm dy|x)=\frac{1}{x+q(1-x)} \mathbbm{1}_{y \in [0,x]} \mathrm dy + \frac{q}{x+q(1-x)} \mathbbm{1}_{y \in [x,1]} \mathrm dy \,.
    \]
    Let $(X_k)_{k \geq 0}$ be the Markov chain on $[0,1]$ with Markov kernel $\kappa_q$, started at a uniform variable $X_0 \sim \Unif{[0,1]}$. 
    We define the tree $\cT_*^{(q)}$ as the random discrete infinite tree obtained by starting from a infinite spine $(v_0,v_1,\ldots)$, rooted at $\rho_*^{(q)} := v_0$, and grafting an independent copy of $\cT_{X_i}^{(q)}$ onto $v_i$ (identifying the root of $\cT_{X_i}^{(q)}$ with $v_i$, and forgetting the labels). 
\end{definition}

\begin{remark}
\label{rk:kernels}
    The Markov chain $(X_k)_{k\ge0}$ intuitively describes the rescaled labels of the first vertices on the path from a uniform vertex to the root.
    For $q=0$, 
    %it holds that $f_{m,i}^{(0)}(\chi) = \chi (1-\chi)^{m-1} \One{i=1}$, 
    the law of $X_{k+1}$ given $X_k$ is uniform on $[0,X_k]$, and given the Markov chain $(X_k)_{k\ge0}$, the grafted trees $\cT_{X_k}^{(0)}$ are independent random recursive trees with independent $\mathrm{Geom}(X_k)$ sizes.
    For $q=1$, the sequence $(X_k)_{k\ge0}$ is i.i.d.~uniform,
    %it holds that f_{m,i}^{(1)}(\chi) = \frac{e^{-m} m^{m-2}}{(m-1)!} \binom{m-1}{i-1} \chi^{i-1} (1-\chi)^{m-i}
    %$(X_k)_{k\ge0}$ is a sequence of i.i.d.\ $\Unif{[0,1]}$ random variables, and given the Markov chain $(X_k)_{k\ge0}$, 
    and the grafted trees $\cT_{X_k}^{(1)}$ are i.i.d.\ random Cayley trees with i.i.d.\ random sizes given by the probability mass function $m\in\bN \mapsto e^{-m} m^{m-1} / m!$.
\end{remark}

\begin{theorem}\label{th: BS local limit unified}
    Let $(q_n) \in [0,1]^\bN$ be a sequence such that $q_* := \lim_{n \to \infty} q_n \in [0,1]$ exists.
    Then, $\cT_n^{(q_n)}$ converges locally in the Benjamini--Schramm sense to the random rooted tree $(\cT_*^{(q_*)},\rho_*^{(q_*)})$.
\end{theorem}

We emphasize that, in the case $q=0$, this provides an elegant new representation of the local limit $\cT_*^{\rec}$.

\begin{remark}
\label{rk:root local limit}
    The second and third authors used a slightly different notion of local convergence in \cite{Thevenin_Wagner_2023}, and proved the convergence in distribution of $B_r(\cT_n^{(q)}, \rho_n)$ for constant $q$, where $\rho_n$ is the root of the $q$-descent-biased tree $\cT_n^{(q)}$.
\end{remark}

\subsection{Organization of the paper}
\label{sec: Organization of the paper and proof strategy}

We start in Section \ref{sec:asymptotics} by studying the exponential generating function of descent-biased trees, and the asymptotics of its coefficients in different regimes. 
These results are mainly obtained by singularity analysis.
We use these asymptotics in Section \ref{sec:scaling limit supercritical} to prove \Cref{th: CRT limit if q fixed}, which states that $\cT_n^{(q)}$ converges to the CRT whenever $q \in (0,1]$ is fixed.
In Section \ref{sec:dendron limit}, we show the convergence of $\cT_n^{(q_n)}$ towards a limit dendron when $nq_n$ converges to some $a>0$.
This relies again on the asymptotics of \Cref{sec:asymptotics}, together with a thorough study of a decomposition of our trees into ``components''.
Then, our results about the local limit of descent-biased trees in all regimes are gathered in Section \ref{sec:local limit}.
Finally, in Appendix \ref{sec: Gromov Prokhorov}, we recall the framework of real trees and dendrons, as well as the Gromov--Prokhorov topology which we use here. 
%Finally, Appendix \ref{sec: Markov chain simulation} is devoted to the description of a Markov chain with which we draw simulations of large descent-biased trees.

\subsection{Discussion}
\label{sec: Discussion}

\paragraph{Open questions.}
In this work, we only deal with the Gromov--Prokhorov topology.
In the context of real trees, however, it is also common to consider the \emph{Gromov--Hausdorff} (GH) topology.
Together with the GP topology, these combine into the stronger \emph{Gromov--Hausdorff--Prokhorov} (GHP) topology.
Extending the convergence of \Cref{th: CRT limit if q fixed} to the GHP topology was out of reach for us, and this remains an open question.
The main challenge lies in establishing bounds on global statistics, such as height, of descent-biased trees.
Note that this question would not make sense in the context of \Cref{th: dendron limit if a/n}, as there is no suitable extension of the GH topology to the space of dendrons.

Another natural question concerns the intermediate range $q_n\to0, nq_n\to\infty$.
Our belief is that, in this regime, the random trees $\cT_n^{(q_n)}$ rescaled by $\frac{q_n-1}{\sqrt{nq_n} \log q_n}$ converge in distribution towards the CRT.
One potential approach would be to extend our proof in the $q\in(0,1]$ regime to allow for a dependency on $n$;
however, this would prevent direct applications of singularity analysis, on which our current proof relies.
Instead, more extensive computation would be required to better understand the behavior of the generating function $A(x,q_n)$ and its coefficients, and this would raise new difficulties.
Another potential approach would be to rely on a weighted version of the well-known Wilson algorithm, which produces a random weighted spanning tree (see, e.g., \cite[Section~9.7]{Lawler_Limic_2010}).
Putting weights $1$ and $q$ on the edges of the directed complete graph, this algorithm returns a $q$-descent-biased tree.
In addition to providing an efficient algorithm for simulation,
this approach has proved useful for studying the scaling limits of uniform spanning trees; see, e.g., \cite{Archer_Shalev_2024,Archer_Nachmias_Shalev_2024}.

\paragraph{Related work.}
Let us mention some related work.
Our model is an \emph{exponential family}, which is a popular type of model in discrete probability.
Recently, the authors of \cite{Durhuus_Unel_2023,Addario-Berry_Corsini_Maitra_Unel_2025} studied another exponential model of random trees.
Specifically, they investigated random rooted plane trees biased by their \emph{height}.
It was shown in \cite{Durhuus_Unel_2023} that these trees admit a local limit at the root which is either single-spine or multi-spine, depending on the bias parameter.
Then, the authors of \cite{Addario-Berry_Corsini_Maitra_Unel_2025} proved that the scaling limit of this model is, in the right regime, a height-biased variant of the CRT.
\medskip

To our knowledge, not many papers deal with the dendron limits of random tree models.
In \cite{Janson21}, Janson exhibits a wide class of random trees with logarithmic height that converge to the same limit $\Upsilon_{\delta_1}$ as recursive trees, up to rescaling.
Other dendron limits have recently been found in the works \cite{Bellin_Blanc-Renaudie_Kammerer_Kortchemski_2025,Bellin_Blanc-Renaudie_Kammerer_Kortchemski_2025_scaling}, which are concerned with a model of random trees constructed by uniform attachment with freezing. 
These are constructed with the same attachment procedure as recursive trees, except that vertices may randomly freeze:
then, new vertices cannot be attached to a frozen vertex for the rest of the procedure.
It was shown in \cite{Bellin_Blanc-Renaudie_Kammerer_Kortchemski_2025} that, if the number of non-frozen vertices is of order $cn$ for some $c>0$ after $n$ steps, then these trees converge after rescaling by $\log n$ to the dendron $\Upsilon_{\delta_{(c+1)/(2c)}}$ (using the notation of \cite{Janson21}).
Informally, this means that most vertices are at distance $\frac{c+1}{2c} \log n$ from the root and $\frac{c+1}{c} \log n$ from each other.
When $c=1$, this is the classical dendron limit $\Upsilon_{\delta_1}$ of recursive trees.
In \cite{Bellin_Blanc-Renaudie_Kammerer_Kortchemski_2025_scaling}, it was proved that if the number of active vertices is of order $n^\alpha$ for some $\alpha\in(1/2,1)$ after $n$ steps, then these trees converge after rescaling by $n^{1-\alpha}$ to a dendron $\Upsilon_\nu$ with a trivial base tree and an explicit diffuse measure.
\medskip

In a different direction, Duquesne and Winkel \cite{DuquesneWinkel2026} investigate another way to define limits which cannot be captured in the classical Gromov–Prokhorov setting, addressed through the notion of mass erasure of real trees. The mass erasure convergence is strictly weaker than the Gromov-Prokhorov one, as it allows mass going to infinity. Investigating the connection between these two notions of convergence in more detail will be the object of future work.
\medskip

Finally, let us mention the recent preprint \cite{stufler2026poissondirichletgraphonspermutons}.
Therein, Stufler introduced new classes of random graphons and permutons, and coined the terms Poisson--Dirichlet graphons and permutons.
The theory of graphons and permutons is in a sense analogous to the theory of dendrons, since a graphon (resp., permuton) is uniquely determined by the finite subgraphs (resp., subpermutations) that it randomly induces.
In the same way, a dendron is uniquely determined by its random distance matrices, and convergence of finite trees towards a given dendron corresponds to the convergence in distribution of its randomly spanned subtrees.

Let us focus on the graphon case, as the permuton case is similar.
The Poisson--Dirichlet graphons of \cite{stufler2026poissondirichletgraphonspermutons} are defined by a \enquote{head and components} structure, where the head describes how to assemble the components into a single object.
The head is a random graphon, and the components are i.i.d.~random graphons with random \enquote{sizes} given by a Poisson--Dirichlet sequence.
Stufler proved that these graphons are the limits of a certain class of random graphs constructed with a similar \enquote{head} and \enquote{components} structure, with a bias on the number of components.

There is a strong similarity between these objects and ours.
Indeed, the $q$-descent-biased tree may be seen as a combination of recursive trees, grafted one onto another, with a bias on the number of such components.
Then, consider the random dendron $\cD_a$.
Its \enquote{head} is the base random infinite discrete tree (describing how the recursive components are grafted one onto another), its \enquote{components} are trivial dendrons $\Upsilon_{\delta_1}$ (the limits of the recursive components), and the masses are given by a Poisson--Dirichlet distribution.
Although the similarities with \cite{stufler2026poissondirichletgraphonspermutons} are striking, we do not rely exactly on this head-component structure for our proofs (see \Cref{sec:dendron limit} for details).

\subsection{Notation}

In the paper, $\bN := \{1,2,\ldots\}$ denotes the set of positive integers, and $\bN_0 := \bN \cup \{0\}$ the set of nonnegative integers.
For $n\in\bN_0$, we set $[n] := \{1,\dots,n\}$.
For $r\in\bN_0$, the $r$-th falling factorial of $n$ is $(n)_r := n(n-1)\cdots(n-r+1)$, with the convention that $(n)_r := 1$ if $r=0$.
If $0\le i\le j$ are two integers, we set $[i,j] := \{i,i+1,\dots,j\}$.

For $x \in \bR$, we set $\bR_{\ge x}:=[x,\infty)$ and define $\bR_{\le x}, \bR_{> x}, \bR_{< x}$ similarly.
We sometimes write $\bR_+ := \bR_{\ge 0}$ and $\bR_+^* := \bR_{> 0}$.
For $z\in\bC$ and $r\in\bR_{\ge0}$, we write $\bD(z,r)$ for the complex open disk with center $z$ and radius $r$.

The set of all rooted labeled trees is $\bT := \bigcup_{n\ge1} \bT_n$.
The size $\abs{t}$ of a tree $t$ is its number of vertices, i.e., $\abs{t}=n$ if $t\in\bT_n$.
The root of a tree $t$ is denoted by $\rho(t)$, or simply $\rho$ if there is no ambiguity.

The $n$-th coefficient of a power series $A(x) = \sum_{k\ge0} a_k x^k$ is $[x^n] A(x) := a_n$.
The exponential generating function of descent-biased trees is
\begin{equation}\label{eq: def of generating function Axq}
    A(x,q) := \sum_{t\in\bT} q^\des{t} \frac{x^\abs{t}}{\abs{t}!}
    = \sum_{n\ge1} Z_{n,q} \frac{x^n}{n!},
\end{equation}
where we recall that $Z_{n,q} := \sum_{t \in \bT_n} q^\des{t}$.
Note that $Z_{n,0} = (n-1)!$ and $Z_{n,1} = n^{n-1}$ for each $n\in\bN$.
Correspondingly, the exponential generating functions of recursive trees and of Cayley trees are characterized by $A(x,0) = -\log(1-x)$ and $A(x,1) = x e^{A(x,1)}$, respectively.

If $\cX$ is a topological space, we let $\cM_1(\cX)$ denote the space of Borel probability measures on $\cX$, endowed with the weak topology.
If $X$ is a random variable taking values in $\cX$, we let $\Law{X} \in \cM_1(\cX)$ denote its law.

Finally, let 
\[
    \Delta := \Big\{ (x_1,x_2, \dots) \in \bR_+^{\bN} : \sum_{i\ge1} x_i = 1 \Big\}
\]
denote the infinite-dimensional simplex, endowed with the product topology.
In order to define the dendron limit of descent-biased trees, we use a Poisson--Dirichlet distribution $\PoissonDirichlet{0}{a} \in \cM_1(\Delta)$ for $a \in (0,\infty)$. 
Recall that $X=(X_n)_{n\ge1}$ follows the distribution $\PoissonDirichlet{0}{a}$ if $X$ is the decreasing reordering of the sequence $\big( Y_n \prod_{k=1}^{n-1}(1-Y_k) \big)_{n \geq 1}$, where $(Y_k)_{k \geq 1}$ are i.i.d.~$\BetaDistribution{1}{a}$ random variables.

\section*{Acknowledgments}

The authors would like to thank Arthur Blanc-Renaudie for pointing to references \cite{Camarri_Pitman_2000,Blanc-Renaudie_2025} on $p$-trees, which helped prove our \Cref{th: our dendron cv to CRT}. They are also grateful to Eleanor Archer for suggesting the use of Wilson's algorithm, in particular for simulations.

VD was supported partly by the French PIA project ``Lorraine Université d’Excellence'', reference ANR-15-IDEX-04-LUE;
by the %ANR 
project LOUCCOUM 
%(Large Objects Under Combinatorial Constraints and Outside Uniform Models, 
(reference ANR-24-CE40-7809);
by the Knut and Alice Wallenberg Foundation;
and by Ragnar Söderberg's Foundation.

PT was supported by the project ANR-24-CPJ1-0032-01, and the project "Rawabranch" number ANR-23-CE40-0008.

SW was supported by the Swedish research council (Vetenskapsr{\aa}det), grant 2022-04030.

\section{Asymptotics of generating functions and their coefficients}
\label{sec:asymptotics}

This section contains the main tools from analytic combinatorics that we shall use.
The principal object of study is the generating function $A(x,q)$ of descent-biased trees; see \eqref{eq: def of generating function Axq}.
In \Cref{sec: Generating functions related to descent-biased trees}, we recall some analytical properties of $A(x,q)$ that were proven in \cite{Thevenin_Wagner_2023}.
We show in Lemmas~\ref{lem:trees with smallest root} and~\ref{lem: generating function typical depth} how to compute two other useful generating functions:
that of descent-biased trees with root label $1$, and that for the depth of a typical vertex in the tree.
In Sections \ref{sec: Asymptotics for fixed q} and \ref{sec: Asymptotics for q=a/n}, we use standard tools from analytic combinatorics (see, e.g., \cite{Flajolet_Sedgewick_2009} and \cite[Corollary~2.16]{Drmota_2009} for references) to derive the asymptotic behavior of our generating functions in the two regimes $q_n \equiv q\in(0,1]$ and $q_n = a/n$, $a>0$.
Specifically, Equation \eqref{eq: asymptotic coefficient A for fixed q} and Lemma~\ref{lem: asymptotic coefficient partition function in transition regime} treat the function $A(x,q_n)$, whereas Propositions~\ref{prop: Rayleigh limit of typical depth when q fixed} and~\ref{prop: asymptotic coefficient typical depth in transition regime} concern the depth of a typical vertex.
These propositions can be seen as first steps towards \Cref{th: CRT limit if q fixed,th: dendron limit if a/n}, which investigate the pairwise distances of finitely many typical vertices.
As an easy byproduct of the asymptotics of $A(x,q_n)$, we provide central limit theorems for the number of descents in $\cT_n^{(q_n)}$ in Corollaries~\ref{cor: CLT descents q fixed} and~\ref{cor: CLT descents q=a/n}.

\begin{remark}
    We believe that similar results could be obtained across the entire $q_n \to0$, $nq_n\to\infty$ range, although this might require heavier computation.
\end{remark}

\subsection{Generating functions related to descent-biased trees}
\label{sec: Generating functions related to descent-biased trees}

When dealing with generating functions related to descent-biased trees, an approach that we extensively use is to decompose a given tree $t\in \bT_n$ according to the vertex $v_1$ with label $1$.
If $v_1$ is the root of $t$, then the tree is decomposed into a family of subtrees $t_1, \dots, t_p$, rooted at the children of $v_1$.
Otherwise, $t$ is decomposed into a tree $t_0$ together with a vertex $v_0\in t_0$ to which $v_1$ is attached, and a family of subtrees $t_1, \dots, t_p$ rooted at the children of $v_1$.
In the first case, $\des{t} = \des{t_1} + \cdots + \des{t_p}$, and in the second case, $\des{t} = 1 + \des{t_0} + \des{t_1} + \cdots + \des{t_p}$.
Expressing this decomposition in terms of the generating function $A(x,q)$, we get the differential equation
\begin{equation}\label{eq: differential equation for A}
    \partial_x A(x,q) = e^{A(x,q)} + q x e^{A(x,q)} \partial_x A(x,q) ,
    \quad\text{or equivalently,}\quad
    (\partial_x A(x,q)) ( e^{-A(x,q)} - qx ) = 1 \,,
\end{equation}
leading to the implicit equation 
\begin{equation}\label{eq: implicit equation for A}
    x = \frac{e^{-q A(x,q)} - e^{-A(x,q)}}{1-q} \,,
\end{equation}
see \cite{Thevenin_Wagner_2023}.
As was observed in \cite[Lemma~21]{Thevenin_Wagner_2023}, this entails that for any fixed $q>0$, $A(x,q)$ is amenable to singularity analysis with a square root singularity at 
\begin{equation}
    x_0(q) := q^{q/(1-q)} \,.
\end{equation}
In particular, $A(x,q)$ is analytic on a complex domain of the form $\bD_A := \bD(0,x_0(q)+\delta) \setminus \bR_{\ge x_0(q)}$ for some $\delta>0$. In fact, it can be shown that $A(x,q)$ is analytic in the slit plane $\bC \setminus \bR_{\ge x_0(q)}$, but we will not need this.

The same decomposition method can be used to find expressions for other generating functions.
For example, this was done in \cite[Section~6.3]{Thevenin_Wagner_2023} for the mean path length of the tree.
The following two lemmas concern the generating functions of descent-biased trees with root labeled $1$, and of the typical depth of a vertex in the tree.

\begin{lemma}
\label{lem:trees with smallest root}
    Let $\cE_{1}$ denote the set of trees $t \in \bT$ with root labeled $1$, and
    \begin{align*}
        A^{(1)}(x,q) := \sum_{t\in \bT} q^\des{t} \frac{x^{|t|}}{|t|!}  \One{t \in \cE_1}
    \end{align*}
    denote the (exponential) generating function of descent-biased trees in $\cE_1$.
    Then, we have
    \begin{align*}
        \partial_x A^{(1)}(x,q) = e^{A(x,q)}. 
    \end{align*}
\end{lemma}

\begin{proof}
    This is a straightforward application of our decomposition.
\end{proof}

\begin{lemma}
\label{lem: generating function typical depth}
    For a tree $t\in \bT$, recall that $\rho(t)$ denotes its root and $\dist$ denotes its graph metric. Let
    \begin{equation*}
        F_1(x,y,q) := \sum_{t\in\bT} q^\des{t} \frac{x^\abs{t}}{\abs{t}!} \sum_{w\in t} y^{\dist(\rho(t), w)}
    \end{equation*}
    be the generating function of the depth of a typical vertex in a descent-biased tree.
    Then, we have the differential equation
    \begin{align}\label{eq: equadiff on F1}
        \partial_x F_1(x,y,q) 
        = \partial_x A(x,q) \left(
        1 + (1+q)y F_1(x,y,q) + qy^2 F_1(x,y,q)^2 
        \right) \,,
    \end{align}
    with the solution
    \begin{equation}\label{eq: explicit formula for generating function of typical depth}
        F_1(x,y,q) = \frac{ e^{(1-q)yA(x,q)} - 1 }{ y - qye^{(1-q)yA(x,q)} } \,.
    \end{equation}
\end{lemma}

\begin{proof}
    For $t\in\bT$, write 
    \[
        G(t,y) := \sum_{w\in t} y^{\dist(\rho(t), w)} \,.
    \]
    Consider a tree $t\in\bT_{n+1}$, let $v_1$ denote its vertex labeled $1$, let $p\ge0$ be the number of children of $v_1$, and let $t_1, \dots, t_p$ be the subtrees of $t$ rooted at the children of $v_1$ (ordered in an arbitrary manner).
    %, e.g., using their respective root labels).
    If $v_1 = \rho(t)$, then
    \[
        G(t,y) = 1 + y \sum_{i=1}^p G(t_i,y) \,.
    \]
    If $v_1 \ne \rho(t)$, then we let $t_0$ denote the tree $t$ minus $v_1$ and its descendants, let $t_1, \dots, t_p$ be the subtrees rooted at the children of $v_1$, and we let $v_0\in t_0$ denote the parent of $v_1$ in $t$.
    In that case,
    \[
        G(t,y) = G(t_0,y) + y^{1+\dist(\rho(t_0), v_0)} \left( 1 + y \sum_{i=1}^p G(t_i,y) \right) \,.
    \]
    From this analysis, we deduce that
    \begin{align*}
        \partial_x F_1(x,y,q)
        &= \sum_{n\ge0} \sum_{t\in \bT_{n+1}} q^\des{t} \frac{x^n}{n!} G(t,y)
        \\&= \sum_{p\ge0} \sum_{t_1, \dots, t_p\in \bT} 
        \frac{1}{p!} \binom{\abs{t_1}+\dots+\abs{t_p}}{\abs{t_1}, \dots, \abs{t_p}}
        q^{\des{t_1}+\dots+\des{t_p}} \frac{x^{\abs{t_1}+\dots+\abs{t_p}}}{(\abs{t_1}+\dots+\abs{t_p})!}
        \left( 1 + y \sum_{i=1}^p G(t_i,y) \right)
        \\&\hspace{1em} + \sum_{p\ge0} \sum_{t_0,t_1, \dots, t_p\in \bT} \sum_{v_0\in t_0}
        \frac{1}{p!} \binom{\abs{t_0}+\dots+\abs{t_p}}{\abs{t_0}, \dots, \abs{t_p}}
        q^{1+\des{t_0}+\dots+\des{t_p}} \frac{x^{\abs{t_0}+\dots+\abs{t_p}}}{(\abs{t_0}+\dots+\abs{t_p})!}
        \\&\hspace{11em} \cdot \left( G(t_0,y) + y^{1+\dist(\rho(t_0), v_0)} \left( 1 + y \sum_{i=1}^p G(t_i,y) \right) \right)
        \\&= \sum_{p\ge0} \frac{1}{p!} \left( 
        A(x,q)^p + y p F_1(x,y,q) A(x,q)^{p-1}
        \right)
        \\&\hspace{1em} + q \sum_{p\ge0} \frac{1}{p!} \left(
        x \partial_x F_1(x,y,q) A(x,q)^p + y F_1(x,y,q) A(x,q)^p + y^2 p F_1(x,y,q)^2 A(x,q)^{p-1}
        \right) \,,
    \end{align*}
    and thus
    \begin{equation*}
        e^{-A(x,q)} \partial_x F_1(x,y,q) = 
        1
        + y F_1(x,y,q) 
        + q x \partial_x F_1(x,y,q) 
        + q y F_1(x,y,q) 
        + q y^2 F_1(x,y,q)^2 \,.
    \end{equation*}
    Regrouping the terms containing $\partial_x F_1(x,y,q)$, multiplying by $\partial_x A(x,q)$, and recalling \eqref{eq: differential equation for A}, we get \eqref{eq: equadiff on F1}.
    Observing that this is a Riccati equation, \eqref{eq: explicit formula for generating function of typical depth} follows by standard methods.
\end{proof}

\subsection{Asymptotics for fixed $q\in(0,1]$}
\label{sec: Asymptotics for fixed q}

In this section, we consider a fixed $q_n \equiv q \in (0,1]$.
Then, by \cite[Lemma~21]{Thevenin_Wagner_2023}, we have
\begin{equation}\label{eq: asymptotic A close to singularity for fixed q}
    A(x,q) = \frac{\log q}{q-1} - \sqrt{\frac{2}{q}} \sqrt{1 - \frac{x}{x_0(q)}} + \cO(\abs{x-x_0(q)})
\end{equation}
as $x\to x_0(q)$ in $\bD_A$.
This holds uniformly for $q$ in any compact subinterval of $(0,1]$.
By singularity analysis (see \cite[Theorem VI.4]{Flajolet_Sedgewick_2009}), this entails that
\begin{equation}\label{eq: asymptotic coefficient A for fixed q}
    \frac{1}{n!} Z_{n,q} := [x^n] A(x,q) \sim \frac{1}{\sqrt{2\pi q}} n^{-3/2} x_0(q)^{-n}
\end{equation}
as $n\to\infty$, uniformly for $q$ in any compact subinterval of $(0,1]$.
As a more involved application of singularity analysis, we show that the depth of a typical vertex in $\cT_n^{(q)}$ follows a Rayleigh distribution after rescaling, indicative of the CRT scaling limit.
This is the first step towards \Cref{th: CRT limit if q fixed}, whose full proof is conducted in \Cref{sec:scaling limit supercritical}.

\begin{proposition}\label{prop: Rayleigh limit of typical depth when q fixed}
    Recall the function $F_1$ of Lemma~\ref{lem: generating function typical depth}.
    For $r\in\bN_0$, we let $f_{1,r}(x,q)$ denote the $r$-th derivative of $F_1(x,y,q)$ with respect to $y$, evaluated at $y=1$.
    Then, for fixed $q\in(0,1]$, we have
    \begin{equation}\label{eq: claim singularity f1r}
        f_{1,r}(x,q) \sim \frac{r!}{\sqrt{2q}} \left( \sqrt{\frac{q}{2}} \frac{\log q}{q-1} \right)^r \left( 1 - \frac{x}{x_0(q)} \right)^{-1/2-r/2}
    \end{equation}
    as $x\to x_0(q)$.
    As a consequence, if $u_n$ denotes a uniformly random vertex of $\cT_n^{(q)}$ and $\dist(\rho,u_n)$ its depth, then
    \begin{equation*}
        \frac{q-1}{\sqrt{q} \log q} \frac{1}{\sqrt n} \dist(\rho,u_n) 
        \cv{} \mathrm{Rayleigh}(1)
    \end{equation*}
    in distribution and with convergence of all moments as $n\to\infty$.
\end{proposition}

\begin{proof}
    Let us skip the dependence on $q$ in notation.
    We first prove \eqref{eq: claim singularity f1r} by induction on $r$.
    For $r=0$, we simply have $f_{1,0}(x) = x\partial_x A(x)$.
    By known closure properties of functions amenable to singularity analysis \cite[Section VI.10]{Flajolet_Sedgewick_2009}, we can differentiate~\eqref{eq: asymptotic A close to singularity for fixed q} to get
    \begin{equation}\label{eq: singularity for single derivative of A and for f10}
    \partial_x A(x) = \frac{1}{x_0\sqrt{2q}} (1-x/x_0)^{-1/2} + \cO(1) ,\quad
    \text{ and thus,}\quad
    f_{1,0}(x) = \frac{1}{\sqrt{2q}} (1-x/x_0)^{-1/2} + \cO(1)    
    \end{equation}
    as $x\to x_0$.
    For $r=1$, observe that $f_{1,1}(x)$ is the generating function for the mean path length of $\cT_n^{(q)}$.
    It is thus the function $B$ of \cite{Thevenin_Wagner_2023}, for which \eqref{eq: claim singularity f1r} was proved in \cite[above Proposition~25]{Thevenin_Wagner_2023}.
    Now consider $r\ge2$, and assume that \eqref{eq: claim singularity f1r} holds for all orders up to $r-1$.
    Let us differentiate equation \eqref{eq: equadiff on F1} of Lemma~\ref{lem: generating function typical depth} a total of $r$ times with respect to $y$:
    \begin{align*}
        \partial_x \partial_y^r F_1(x,y) = \partial_x A(x) \Bigg(
        &(1+q)y \partial_y^r F_1(x,y) 
        + r(1+q) \partial_y^{r-1} F_1(x,y) 
        + qy^2\sum_{k=0}^r \binom{r}{k} \partial_y^k F_1(x,y) \partial_y^{r-k} F_1(x,y) 
        \\&+ 2rqy\sum_{k=0}^{r-1} \binom{r-1}{k} \partial_y^k F_1(x,y) \partial_y^{r-1-k} F_1(x,y) 
        \\&+ r(r-1)q\sum_{k=0}^{r-2} \binom{r-2}{k} \partial_y^k F_1(x,y) \partial_y^{r-2-k} F_1(x,y)
        \Bigg) \,.
    \end{align*}
    At $y=1$, this becomes
    \begin{align*}
        \partial_x f_{1,r}(x) = \partial_x A(x) \Bigg(
        &(1+q) f_{1,r}(x) 
        + r(1+q) f_{1,r-1}(x) 
        + q\sum_{k=0}^r \binom{r}{k} f_{1,k}(x) f_{1,r-k}(x) 
        \\&+ 2rq\sum_{k=0}^{r-1} \binom{r-1}{k} f_{1,k}(x) f_{1,r-1-k}(x) 
        + r(r-1)q\sum_{k=0}^{r-2} \binom{r-2}{k} f_{1,k}(x) f_{1,r-2-k}(x)
        \Bigg) \,.
    \end{align*}
    Isolating the terms containing $f_{1,r}$ on the right-hand side, we can write 
    \begin{equation*}
        \partial_x f_{1,r}(x) = f_{1,r}(x) G_1'(x) + R_{1,r}(x)
    \end{equation*}
    where $G_1'(x) := \partial_x A(x) \left(\big. (1+q) + 2qf_{1,0}(x) \right)$, and $R_{1,r}$ contains only $f_{1,\alpha}$'s with $\alpha<r$.
    The function $G_1$ is defined as an appropriate primitive of $G_1'$, to be specified later.
    We solve this linear ODE using variation of parameters to obtain
    \begin{equation}\label{eq: solution f1r with variation of constant}
        f_{1,r}(x) = e^{G_1(x)} \int_0^x R_{1,r}(u) e^{-G_1(u)} \mathrm{d}u \,.
    \end{equation}
    From \eqref{eq: singularity for single derivative of A and for f10}, we get
    \begin{equation*}
        G_1'(x) = \frac{1}{x_0} (1-x/x_0)^{-1} + \cO\left( \abs{x-x_0}^{-1/2} \right) ,
        \quad\text{and thus}\quad
        G_1(x) = -\log(1-x/x_0) + \cO\left( \abs{x-x_0}^{1/2} \right)
    \end{equation*}
    as $x\to x_0$, for a unique adequate choice of $G_1$.
    This implies
    \begin{equation}\label{eq: singularity for expG1}
        e^{G_1(x)} \sim (1-x/x_0)^{-1} 
    \end{equation}
    as $x\to x_0$.
    Next, using the induction hypothesis and the fact that $r\ge2$, it is straightforward to find the dominant terms inside $R_{1,r}$:
    \begin{equation*}
        R_{1,r}(x) 
        \;\sim\; q \partial_x A(x)
        \sum_{k=1}^{r-1} \binom{r}{k} f_{1,k}(x) f_{1,r-k}(x)
        \;\sim\; \frac{r!}{\sqrt{2q}} \left( \sqrt{\frac{q}{2}} \frac{\log q}{q-1} \right)^r \frac{r-1}{2x_0} (1-x/x_0)^{-3/2 - r/2}
    \end{equation*}
    as $x\to x_0$.
    Therefore, with \eqref{eq: singularity for expG1},
    \begin{equation*}
        \int_0^x R_{1,r}(u) e^{-G_1(u)} \mathrm{d}u 
        \;\sim\; \frac{r!}{\sqrt{2q}} \left( \sqrt{\frac{q}{2}} \frac{\log q}{q-1} \right)^r (1-x/x_0)^{1/2 - r/2} \,,
    \end{equation*}
    and thus, using \eqref{eq: solution f1r with variation of constant},
    \begin{equation*}
        f_{1,r}(x) \sim \frac{r!}{\sqrt{2q}} \left( \sqrt{\frac{q}{2}} \frac{\log q}{q-1} \right)^r (1-x/x_0)^{-1/2 - r/2}
    \end{equation*}
    as $x\to x_0$.
    This completes the proof of \eqref{eq: claim singularity f1r}.
    Now, for the second claim, observe that
    \begin{align*}
        \expec{ \left( \dist(\rho,u_n) \right)_r } = \frac{[x^n] f_{1,r}(x)}{n [x^n] A(x)} \,,
    \end{align*}
    where $(.)_r$ denotes the $r$-th falling factorial.
    Using \eqref{eq: asymptotic coefficient A for fixed q}, \eqref{eq: claim singularity f1r} and singularity analysis, we get
    \begin{align*}
        \expec{ \left( \dist(\rho,u_n) \right)_r } 
        \;\sim\; \frac{r! \sqrt\pi}{\Gamma(1/2+r/2)} \left( \sqrt{\frac{q}{2}} \frac{\log q}{q-1} \right)^r n^{r/2}
    \end{align*}
    as $n\to\infty$.
    Therefore, the same holds for the expectation of the $r$-th power of $\dist(\rho,u_n)$, and thus
    \begin{align*}
        \expec{ \left( \frac{q-1}{\sqrt q \log q} \frac{1}{\sqrt n} \dist(\rho,u_n) \right)^r } 
        \cv{} \frac{r! \sqrt\pi}{2^{r/2}\Gamma(1/2+r/2)} 
        = 2^{r/2} \Gamma(1+r/2)
    \end{align*}
    as $n\to\infty$, using Legendre's duplication formula.
    The right-hand side represents the moments of a standard Rayleigh distribution.
    Since this distribution is characterized by its moments, this concludes the proof of the proposition.
\end{proof}
    
Before moving on to the next regime, let us derive a central limit theorem (CLT) for the number of descents from~\eqref{eq: asymptotic coefficient A for fixed q}.
Although we will not need it in the rest of the paper, we choose to include it since it follows fairly easily from what we have already proved.

\begin{corollary}\label{cor: CLT descents q fixed}
    Let $q\in(0,1]$ be fixed.
    Then,
    \begin{equation*}
        \frac{\des{\cT_n^{(q)}} - \mu(q) n}{\sqrt{\nu(q) n}}
        \cv{}
        \Normal{0}{1}
    \end{equation*}
    in distribution as $n\to\infty$, where $\mu(q) = \frac{q}{q-1}\left( 1 - \frac{\log q}{q-1} \right)$ and $\nu(q) = \frac{q(1+q)\log q}{(q-1)^3} -\frac{2q}{(q-1)^2}$ (for $q=1$, these are to be interpreted as their respective limits $\frac12$ and $\frac16$).
\end{corollary}

\begin{proof}
    Uniformly for $u$ in a (real) neighborhood of $1$, we have (by \eqref{eq: asymptotic coefficient A for fixed q})
    \begin{align*}
        \expec{u^{\des{\cT_n^{(q)}}}}
        = \frac{[x^n] A(x,qu)}{[x^n] A(x,q)}
        \sim \frac{1}{\sqrt u} \left( \frac{x_0(qu)}{x_0(q)} \right)^{-n}
        = e^{n\cdot \alpha(u) + \beta(u)} ( 1 + o(1) )
    \end{align*}
    as $n\to\infty$, where $\alpha(u) = q\frac{\log q}{1-q} - qu\frac{\log qu}{1-qu}$ and $\beta(u) = -\frac{1}{2}\log u$.
    The result follows from a standard application of Curtiss's theorem \cite{Curtiss_1942}.%, or of the Quasi-Power Theorem (see, e.g., \cite[Theorem~2.22]{Drmota_2009}) with $\mu = \alpha'(1)$ and $\nu = \alpha'(1) + \alpha''(1)$.
    % \begin{align*}
    %     &\alpha'(u) &= 
    %     &\frac{-q}{1-qu} \left( 1 + \log(qu) + qu\log(qu)\frac{1}{1-qu} \right)
    %     \\& \alpha''(u) &= 
    %     &-2\left( \frac{q}{1-qu} \right)^2 \left( 1 + \log(qu) + qu\log(qu)\frac{1}{1-qu} \right) - \frac{q/u}{1-qu}
    %     = 2\frac{q}{1-qu} \alpha'(u) - \frac{q/u}{1-qu}
    % \end{align*}
\end{proof}

\subsection{Asymptotics for $q_n=a/n$}
\label{sec: Asymptotics for q=a/n}

Now, we consider $q_n = a/n$ for some fixed $a\ge0$.
In this regime where $q_n$ depends on $n$, singularity analysis is no longer directly applicable.
Nonetheless, we can adapt the approach to get the coefficient asymptotics of several generating functions by means of contour integration.

\begin{lemma}\label{lem: asymptotic coefficient partition function in transition regime}
    Fix $M>0$.
    Then, uniformly for $a\in[0,M]$, we have
    \begin{equation*}
        [x^n] A(x,a/n) \sim \frac{n^{a-1}}{e^a \Gamma(a+1)}
    \end{equation*}
    as $n\to\infty$.
\end{lemma}

\begin{lemma}\label{lem: asymptotic coefficient exponentials in transition regime}
    Fix $M>0$ and $j\ge0$.
    Then, uniformly for $a\in[0,M]$, we have
    \begin{equation*}
        [x^n] e^{j A(x,a/n)} \sim  \frac{j n^{a+j-1}}{e^{a} \Gamma(a+j+1)}
    \end{equation*}
    as $n\to\infty$.
    As a consequence, the probability that the root of $\cT_n^{(a/n)}$ has label $1$ is asymptotically equal to $\frac{1}{a+1}$.
\end{lemma}

\begin{proposition}\label{prop: asymptotic coefficient typical depth in transition regime}
    Recall the function $F_1$ of Lemma~\ref{lem: generating function typical depth}.
    For any fixed $a\ge0$ and $y\in(0,1]$, we have
    \begin{equation*}
        [x^n] F_1{\left( x, y^{1/\log n}, a/n \right)} \sim n^{a} e^{-a} \sum_{k\ge1} k a^{k-1} y^{k} \frac{1}{\Gamma(a+k+1)}
    \end{equation*}
    as $n\to\infty$.
    As a consequence, if $u_n$ denotes a uniformly random vertex of $\cT_n^{(a/n)}$, then $\frac{1}{\log n} \dist(\rho,u_n)$ converges in distribution to a discrete random variable with probability mass function $k\mapsto \frac{k a^{k-1}}{(a+k)\cdots(a+1)}$ on $\bN$.
\end{proposition}

We start with a detailed proof of Lemma~\ref{lem: asymptotic coefficient partition function in transition regime}.
The proofs of Lemma~\ref{lem: asymptotic coefficient exponentials in transition regime} and Proposition~\ref{prop: asymptotic coefficient typical depth in transition regime} use the same technique, so we will skip over some of the computational details.

\begin{proof}[Proof of Lemma~\ref{lem: asymptotic coefficient partition function in transition regime}]
    For $q \in (0,1)$, recall that $x\mapsto A(x,q)$ is analytic on $\bD_A := \bD(0,x_0(q)+\delta) \setminus \bR_{\ge x_0(q)}$ for some $\delta>0$, where $x_0 = x_0(q):= q^{q/(1-q)}$.
    On this domain, it is one-to-one with inverse $z\mapsto \frac{e^{-qz}-e^{-z}}{1-q} = A^{-1}(z,q)$.
    Also, $A(0,q)=0$, and from \eqref{eq: asymptotic A close to singularity for fixed q}, $A(x_0,q) = \frac{\log q}{q-1} \in [0,\infty]$.
    
    Let $\cC$ be a simple closed contour in $\bD_A \setminus \{0\}$ that winds around $0$ in counterclockwise direction (for example, a sufficiently small circle centered at $0$).
    By Cauchy's integral formula, we have
    \begin{equation*}
        [x^n] A(x,q) = \frac{1}{n}[x^{n-1}] \partial_x A(x,q) 
        = \frac{1}{n} \frac{1}{2i\pi} \oint_\cC \frac{1}{x^n} \partial_x A(x,q) \mathrm{d}x \,.
    \end{equation*}
    The change of variables $z = A(x,q)$, $\mathrm{d}z = \partial_x A \cdot \mathrm{d}x$, $x = \frac{e^{-qz}-e^{-z}}{1-q}$, yields
    \begin{align*}
        [x^n] A(x,q)
        = \frac{(1-q)^n}{n} \frac{1}{2i\pi} \oint_{\cC'} \left( e^{-qz} - e^{-z} \right)^{-n} \mathrm{d}z\,,
    \end{align*}
    where $\cC' := A(\cC,q)$ is a simple contour in $\bC \setminus \{0\}$ that winds around $0$ in counterclockwise direction.
    Using the standard determination of $\log$, the map $z\mapsto ne^{-z}$ is one-to-one from the strip $\left\{ z \in \bC, \, \abs{\Im(z)} < \pi \right\}$ onto $\bC \setminus \bR_{\le0}$. If $\cC$ is taken to be sufficiently small, then $\cC'$ lies entirely inside this strip and in particular avoids all singularities of the integrand (which are located at $z = \frac{2ki\pi}{1-q}$, $k \in \bZ$) except for $0$.    
    Now set $q=a/n$ and use the further change of variables $u = ne^{-z}$, $z = -\log(u/n)$, $\mathrm{d}z = -\frac{\mathrm{d}u}{u}$ to obtain
    \begin{align}\label{eq: Cauchy contour formula for A after two changes of variables in transition regime}
        [x^n] A(x,a/n)
        = \frac{-(1-a/n)^n}{n} \frac{1}{2i\pi} \oint_{\cC''} \left( \left(\frac{u}{n}\right)^{a/n} - \frac{u}{n} \right)^{-n} \frac{\mathrm{d}u}{u},
    \end{align}
    where $\cC'' := n\exp(-\cC')$ is a simple contour in $\bC \setminus \bR_{\le0}$, that winds around $n$ in clockwise direction.
    By analyticity, we can use any such contour without changing the value of the above integral.
    Specifically, we take $\cC'':= \gamma_\text{small} \cup \gamma_\text{long,1} \cup \gamma_\text{big} \cup \gamma_\text{long,2}$, where
    \begin{align}\label{eq: contour left-pacman}
        \begin{array}{llll}
        &\gamma_\text{small} 
        &:= &\left\{ u = r e^{i\varphi} : -\frac\pi2 \le \varphi \le \frac\pi2 \right\} ;\vspace{.5em}
        \\&\gamma_\text{long,1} 
        &:= &\left\{ u = -t + r i : 0\le t\le n^2 \right\} ;\vspace{.5em}
        \\&\gamma_\text{big} 
        &:= &\left\{ u = \sqrt{n^4 + r^2} e^{i\varphi} : \abs{\varphi} \le \arccos\left(-\sqrt{n^4/(n^4+r^2)}\right) \right\} ;\vspace{.5em}
        \\&\gamma_\text{long,2} 
        &:= &\left\{ u = -t - r i : n^2\ge t\ge 0 \right\} ;
        \end{array}
    \end{align}
    and $r > 0$ is fixed and arbitrary, see \Cref{fig: contour left-pacman}.
    With this, we decompose \eqref{eq: Cauchy contour formula for A after two changes of variables in transition regime} into
    
    \begin{figure}
        \centering
        \includegraphics[width=0.4\linewidth]{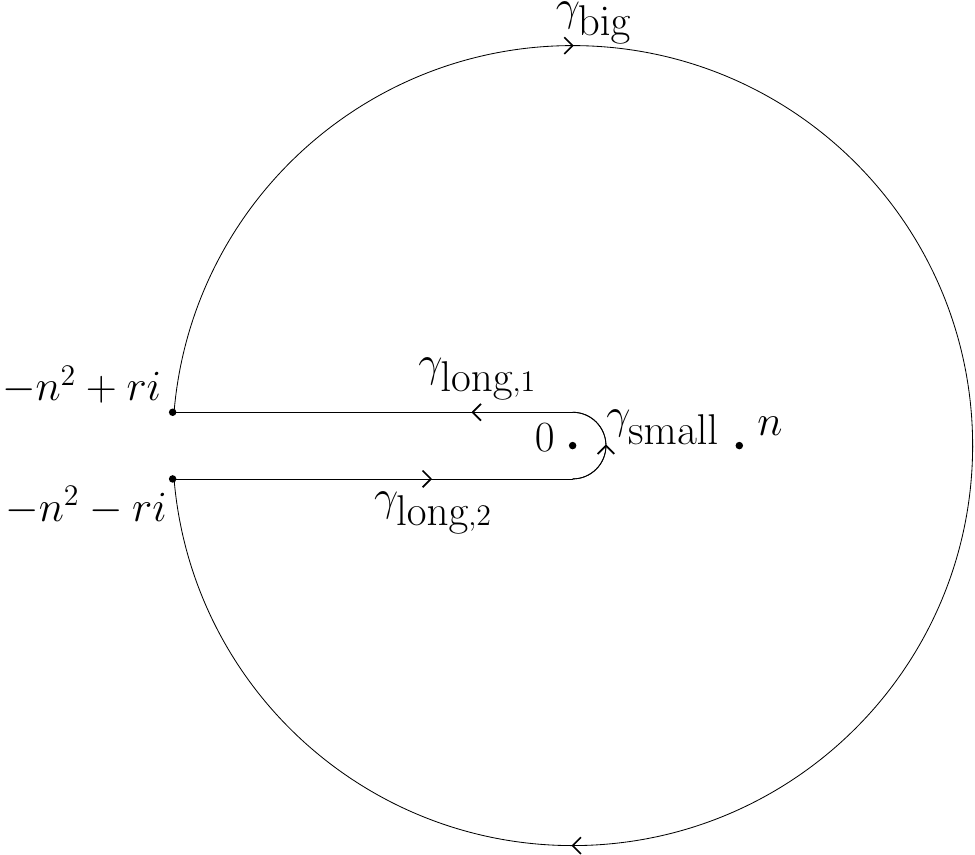}
        \caption{The contour described by \eqref{eq: contour left-pacman}.}
        \label{fig: contour left-pacman}
    \end{figure}
    
    \begin{equation}\label{eq: Cauchy formula for A into 4 parts}
        [x^n] A(x,a/n) = \frac{-(1-a/n)^n}{n} \left( I_\text{small} + I_\text{long,1} + I_\text{big} + I_\text{long,2} \right)
    \end{equation}
    according to the part of the contour we integrate on.
    Let us analyze the term appearing inside the integral of \eqref{eq: Cauchy contour formula for A after two changes of variables in transition regime}.
    All points $u$ on $\cC''$ satisfy $r^2\le \abs{u}^2 \le n^4+r^2$, so we can write 
    \begin{equation}\label{eq: expansion non-inverse integrand before power n}
        \left(\frac{u}{n}\right)^{a/n} - \frac{u}{n} 
        = 1 + \frac{a}{n}\log\left(\frac{u}{n}\right) - \frac{u}{n} + \cO\left(\frac{\log^2n}{n^2}\right)\,,
    \end{equation}
    where the $\cO$ is uniform in $u\in\cC''$ and in $a\in[0,M]$.
    From now on, the analysis differs for each part of the contour.
    To this aim, choose a sequence $m_n$ such that $m_n = o(\sqrt n)$ and $\log n = o(m_n)$.
    
    \paragraph{Small $u$'s.}
    Assume that $r\le |u|\le m_n$.
    From \eqref{eq: expansion non-inverse integrand before power n}, we derive that
    \begin{equation*}
        \left( \left(\frac{u}{n}\right)^{a/n} - \frac{u}{n} \right)^{-1}
        = 1 - \frac{a}{n}\log\left(\frac{u}{n}\right) + \frac{u}{n} + \cO\left(\frac{m_n^2}{n^2}\right)
        =: 1 + \frac{w}{n},
    \end{equation*}
    where $w = w(u)$ depends on $u$ and the error is uniform for $r\le |u|\le m_n$ and $a\in[0,M]$.
    With the same uniformity, we have 
    \begin{equation*}
        e^w = \left(\frac{u}{n}\right)^{-a} e^u \left( 1 + \cO\left(\frac{m_n^2}{n}\right) \right)
    \end{equation*}
    and
    \begin{align*}
        \left( 1 + \frac{w}{n} \right)^n
        = e^{n\log\left(1+\frac{w}{n}\right)}
        = e^{n\left(\frac{w}{n} + \cO\left(\frac{|w|^2}{n^2}\right)\right)}
        = e^w \left( 1 + \cO\left(\frac{m_n^2}{n}\right) \right) \,.
    \end{align*}
    Combining those three, we get
    \begin{align}\label{eq: Hankel computation of A - uniform estimate small u}
        \left( \left(\frac{u}{n}\right)^{a/n} - \frac{u}{n} \right)^{-n}
        = n^a u^{-a} e^u  \left( 1 + \cO\left(\frac{m_n^2}{n}\right) \right)\,,
    \end{align}
    uniformly for $r\le |u|\le m_n$ and $a\in[0,M]$.
    
    \paragraph{Medium $u$'s.}
    Assume that $u = -t \pm ri$ with $\sqrt{m_n^2-r^2}\le t\le n^2$. 
    Then, we have
    \begin{equation*}
        \abs{ 1 + \frac{a}{n}\log\left(\frac{u}{n}\right) - \frac{u}{n} + \cO\left(\frac{\log^2n}{n^2}\right) }
        = \abs{ 1 + \frac{t}{n} + \cO\left(\frac{\log n}{n}\right) }
        \ge 1 + \frac{m_n}{n} + \cO\left(\frac{\log n}{n}\right)\,,
    \end{equation*}
    uniformly for $t \in [\sqrt{m_n^2-r^2},n^2]$ and $a\in[0,M]$.
    Therefore, with the same uniformity,
    \begin{equation}\label{eq: Hankel computation of A - uniform estimate medium u}
        \abs{ \left(\frac{u}{n}\right)^{a/n} - \frac{u}{n} }^{-n}
        \le \left( 1 + \frac{m_n}{n} + \cO\left(\frac{\log n}{n}\right) \right)^{-n}
        = e^{-n\log\left( 1 + \frac{m_n}{n} + \cO\left(\frac{\log n}{n}\right) \right)}
        = \cO\left( e^{-m_n(1+o(1))} \right) \,.
    \end{equation}
    
    \paragraph{Large $u$'s.}
    If $\abs{u}^2 = n^4+r^2$, then
    \begin{equation}\label{eq: Hankel computation of A - uniform estimate large u}
       \abs{ \left( \left(\frac{u}{n}\right)^{a/n} - \frac{u}{n} \right)^{-n} \frac1u }
        \le n^{-n+\cO(1)}\,,
    \end{equation}
    uniformly for $u\in\gamma_\text{big}$ and $a\in[0,M]$.

    \paragraph{Transfer of the estimates.}
    Now, we can transfer these estimates to the integrals in \eqref{eq: Cauchy formula for A into 4 parts}.
    All $\cO$'s below hold uniformly for $a\in[0,M]$.
    First, using \eqref{eq: Hankel computation of A - uniform estimate small u} on the small half-circle, we get
    \begin{align*}
        I_\text{small} 
        = \frac{1}{2i\pi} \oint_{\gamma_\text{small}} \left( \left(\frac{u}{n}\right)^{a/n} - \frac{u}{n} \right)^{-n} \frac{\mathrm{d}u}{u}
        = \frac{1}{2i\pi} \oint_{\gamma_\text{small}} n^{a} u^{-a-1} e^u \mathrm{d}u\,\cdot \left( 1 + \cO\left(\frac{m_n^2}{n}\right) \right) \,.
    \end{align*}
    On $\gamma_\text{long,1}$, we can cut the path at $t_n:=\sqrt{m_n^2-r^2}$ and apply \Cref{eq: Hankel computation of A - uniform estimate small u,eq: Hankel computation of A - uniform estimate medium u} to the two different parts to get
    \begin{equation*}
        I_\text{long,1} 
        = - \frac{1}{2i\pi} \int_0^{t_n} n^{a} (-t+ri)^{-a-1} e^{-t+ri} \mathrm{d}t \,\cdot \left( 1 + \cO\left(\frac{m_n^2}{n}\right) \right)
        + \cO\left( e^{-m_n(1+o(1))} \right) \,,
    \end{equation*}
    and further
    \begin{equation*}
        I_\text{long,1} 
        = - \frac{1}{2i\pi} \int_0^{\infty} n^{a} (-t+ri)^{-a-1} e^{-t+ri} \mathrm{d}t \,\cdot \left( 1 + o(1) \right) \,.
    \end{equation*}
    Likewise, for $\gamma_\text{long,2}$,
    \begin{equation*}
        I_\text{long,2} 
        = \frac{1}{2i\pi} \int_0^{\infty} n^{a} (-t-ri)^{-a-1} e^{-t-ri} \mathrm{d}t \,\cdot \left( 1 + o(1) \right) \,.
    \end{equation*}
    Finally, on $\gamma_\text{big}$ we can use \eqref{eq: Hankel computation of A - uniform estimate large u} to obtain the bound
    \begin{equation*}
        \abs{I_\text{big}} \le n^{-n+\cO(1)} \,.
    \end{equation*}
    Coming back to \eqref{eq: Cauchy formula for A into 4 parts}, we deduce that
    \begin{equation*}
        [x^n] A(x,a/n) = n^{a-1} \frac{e^{-a}}{2i\pi} \oint_\cH u^{-a-1} e^u \mathrm{d}u\,\cdot \left( 1 + o(1) \right)
    \end{equation*}
    where
    \begin{equation}\label{eq: Hankel contour}
        \cH = \left\{ u = -t+ r i : t\ge0 \right\}
        \cup \left\{ u = r e^{i\varphi} : -\tfrac\pi2\le \varphi\le \tfrac{\pi}2 \right\}
        \cup \left\{ u = -t- r i : 0\le t \right\}
    \end{equation}
    is a positive Hankel contour (i.e., traversed clockwise).
    We can finally use the Hankel form of the Gamma function, i.e.,
    \begin{equation*}
        \frac{1}{\Gamma(\alpha)} = \frac{1}{2i\pi} \oint_{\cH} u^{-\alpha} e^u \mathrm{d}u\,,
    \end{equation*}
    which is valid for all $\alpha\in\bC$, to conclude the proof.
\end{proof}

\begin{proof}[Proof of Lemma~\ref{lem: asymptotic coefficient exponentials in transition regime}]
    The approach is largely the same as in the proof of Lemma~\ref{lem: asymptotic coefficient partition function in transition regime}.
    Let $\cC$ be the same contour around $0$ as in that proof.
    By Cauchy's integral formula, we have
    \begin{equation*}
        [x^n] e^{jA(x,q)} = \frac{j}{n}[x^{n-1}] e^{jA(x,q)} \partial_x A(x,q) 
        = \frac{j}{n} \frac{1}{2i\pi} \oint_\cC \frac{e^{jA(x,q)}}{x^n} \partial_x A(x,q) \mathrm{d}x \,.
    \end{equation*}
    Applying the same change of variables (namely, $A(x,q) = z = - \log(u,n)$) as in the proof of Lemma~\ref{lem: asymptotic coefficient partition function in transition regime}, we reach 
    \begin{align*}
        [x^n] e^{jA(x,a/n)}
        = \frac{-j(1-a/n)^n}{n} \frac{1}{2i\pi} \oint_{\cC''} \left(\frac{n}{u}\right)^j \left( \left(\frac{u}{n}\right)^{a/n} - \frac{u}{n} \right)^{-n} \frac{\mathrm{d}u}{u}\,,
    \end{align*}
    where $\cC''$ is again the contour defined in~\eqref{eq: contour left-pacman}.
    The same bounds lead to
    \begin{equation*}
        [x^n]e^{jA(x,a/n)} = n^{a+j-1} \frac{je^{-a}}{2i\pi} \oint_\cH u^{-a-j-1} e^u \mathrm{d}u\,\cdot \left( 1 + o(1) \right) \,,
    \end{equation*}
    which yields the stated formula.
    For the second claim, simply note that
    \begin{equation*}
        \prob{ \cT_n^{(a/n)} \in \mathcal{E}_1 }
        = \frac{[x^n] A^{(1)}(x,a/n)}{[x^n] A(x,a/n)}
        = \frac{1}{n} \frac{[x^{n-1}] e^{A(x,a/n)}}{[x^n] A(x,a/n)}
        \cv{n\to\infty} \frac{1}{a+1}
    \end{equation*}
    by Lemma~\ref{lem:trees with smallest root}.
\end{proof}

\begin{proof}[Proof of Proposition~\ref{prop: asymptotic coefficient typical depth in transition regime}]
    Fix $y>0$ and set $\nu_n := y^{1/\log n}$.
    By \eqref{eq: explicit formula for generating function of typical depth}, we can write $F_1\left(x, \nu_n, a/n\right) = H_n(A(x,a/n))$, where
    \begin{equation*}
        H_n(z) = \frac{1}{\nu_n}\frac{e^{(1-a/n)z\nu_n} - 1}{1 - (a/n)e^{(1-a/n)z\nu_n}} \,.
    \end{equation*}
    This function is analytic on $\bC \setminus \{\frac{\log(n/a) + 2ki\pi}{\nu_n(1-a/n)}\,:\, k \in \bZ\}$, and
    \begin{equation*}
        H_n'(z) = \left(1-\frac{a}{n}\right)^2 \frac{e^{(1-a/n)z\nu_n}}{\left( 1-(a/n)e^{(1-a/n)z\nu_n} \right)^2} \,.
    \end{equation*}
    As in the proof of Lemma~\ref{lem: asymptotic coefficient partition function in transition regime}, we can write
    \begin{align*}
        [x^n] F_1\left(x, \nu_n, a/n\right)
        = \frac{(1-a/n)^n}{n} \frac{1}{2i\pi} \oint_{\cC'} \frac{H_n'(z)}{\left( e^{-az/n} - e^{-z} \right)^n} \mathrm{d}z\,,
    \end{align*}
    where $\cC'$ is a simple contour in $\bC \setminus \{0\}$ that winds around $0$ in counterclockwise direction. It also needs to be small enough to avoid $\frac{\log(n/a)}{\nu_n(1-a/n)}$, the smallest singularity of $H_n$.
    We want to perform the change of variables $u = ne^{-z}$ again. Now note that
    \begin{equation*}
        \frac{\log(n/a)}{\nu_n(1-a/n)} = \frac{\log n - \log a}{1 + \frac{\log y}{\log n} + \cO((\log n)^{-2})} = \log n - \log (ay) + \cO\big((\log n)^{-1} \big),
    \end{equation*}
    so the image of the singularity $\frac{\log(n/a)}{\nu_n(1-a/n)}$ becomes $ne^{-\log n + \log(ay)} = ay$ in the limit as $n\to\infty$. Hence, for large enough $n$,
    \begin{align*}
        [x^n] F_1\left(x, \nu_n, a/n\right)
        \sim \frac{-e^{-a}}{n} \frac{1}{2i\pi} \oint_{\cC''} H_n'\left(-\log(u/n)\right) \left( \left(\frac{u}{n}\right)^{a/n} - \frac{u}{n} \right)^{-n} \frac{\mathrm{d}u}{u}\,,
    \end{align*}
    where $\cC''$ can be chosen as the contour $\gamma_\text{small}  \cup \gamma_\text{long,1} \cup \gamma_\text{big} \cup \gamma_\text{long,2}$ of \eqref{eq: contour left-pacman}, with $r$ taken to be larger than $ay$.
    Using the same bounds as in the proof of Lemma~\ref{lem: asymptotic coefficient partition function in transition regime}, together with the asymptotic formula
    \begin{equation*}
        H_n'(-\log(u/n)) \sim n \frac{yu}{(u-ay)^2}
    \end{equation*}
    for fixed $u$, we get
    \begin{equation*}
        [x^n]F_1\left(x, \nu_n, a/n\right) \sim n^a \frac{e^{-a}}{2i\pi} \oint_\cH \frac{y}{(u-ay)^2} u^{-a} e^u \mathrm{d}u\,,
    \end{equation*}
    where $\cH$ is the Hankel contour of \eqref{eq: Hankel contour} with $r>ay$.
    On this contour, we can expand
    \begin{equation*}
        (u-ay)^{-2} = u^{-2} \sum_{k\ge0} (k+1) (ay/u)^k \,.
    \end{equation*}
    The same expansion holds when $u$ is replaced by $|u|$, so Fubini's theorem applies and we deduce
    \begin{equation*}
        [x^n]F_1\left(x, \nu_n, a/n\right) \sim n^a e^{-a} \sum_{k\ge0} (k+1) a^k y^{k+1} \frac{1}{2i\pi} \oint_{\cH} u^{-a-k-2} e^u \mathrm{d}u
        = n^a e^{-a} \sum_{k\ge1} \frac{ k a^{k-1} y^{k} }{\Gamma(a+k+1)}\,,
    \end{equation*}
    as claimed.
    Finally, using Lemma~\ref{lem: asymptotic coefficient partition function in transition regime}, we obtain
    \begin{equation*}
        \expec{\frac1n \sum_{v\in \cT_n^{(a/n)}} y^{\dist(\rho,v)/\log n}}
        = \frac{[x^n] F_1\left(x, \nu_n, a/n\right)}{n [x^n] A(x,a/n)}
        \cv{} \sum_{k\ge1} \frac{\Gamma(a+1)}{\Gamma(a+k+1)} k a^{k-1} y^{k} 
    \end{equation*}
    as $n\to\infty$.
    This yields the moment generating function of $\dist(\rho,u_n) / \log n$, where $u_n$ is a uniformly random vertex of $\cT_n^{(a/n)}$.
    Since
    \begin{equation*}
       \sum_{k\ge1} \frac{\Gamma(a+1)}{\Gamma(a+k+1)} k a^{k-1} 
       = \sum_{k\ge1} \frac{ka^{k-1}}{(a+k)\cdots(a+1)}
       = \sum_{k\ge1} \left( \frac{a^{k-1}}{(a+k-1)\cdots(a+1)} - \frac{a^k}{(a+k)\cdots(a+1)} \right) = 1 \,,
    \end{equation*}
    we recognize that $\dist(\rho,u_n) / \log n$ converges in distribution to a discrete random variable with probability mass function $k \mapsto \frac{ka^{k-1}}{(a+k)\cdots(a+1)}$ on $\bN$.
\end{proof}

Finally, as in the previous section, we provide a CLT for the number of descents as a straightforward consequence of Lemma~\ref{lem: asymptotic coefficient partition function in transition regime}.

\begin{corollary}\label{cor: CLT descents q=a/n}
    Fix $a>0$.
    Then, we have
    \begin{equation*}
        \frac{\des{\cT_n^{(a/n)}} - a\log n}{\sqrt{a \log n}}
        \cv{} \Normal{0}{1}
    \end{equation*}
    in distribution as $n\to\infty$.
\end{corollary}

\begin{proof}
    Fix $t>0$.
    Then, by Lemma~\ref{lem: asymptotic coefficient partition function in transition regime},
    \begin{align*}
        \expec{ t^{\des{\cT_n^{(a/n)}} / \sqrt{\log n}} }
        = \frac{[x^n] A(x, t^{1/\sqrt{\log n}} a/n)}{[x^n] A(x,a/n)}
        \sim n^{a\left( t^{1/\sqrt{\log n}} - 1 \right)}
        \sim e^{a\log t \sqrt{\log n} + \frac{a}{2}\log^2 t} \,.
    \end{align*}
    Taking out $t^{a\sqrt{\log n}}$ on both sides, this yields the result by convergence of the moment generating function.
\end{proof}

\section{The CRT scaling limit for a fixed bias parameter}
\label{sec:scaling limit supercritical}

This section is solely devoted to the proof of \Cref{th: CRT limit if q fixed}.
Throughout, $q\in(0,1]$ is fixed, and we may drop the notational dependence on $q$.
We will actually focus on proving the following:

\begin{proposition}\label{prop: joint moments cv to CRT}
    Let $\ell\ge1$, and for each pair $(i,j)$ with $0\le i<j\le \ell$, let $r_{i,j} \in \bN_0$.
    Then, with the same notation as in \Cref{th: CRT limit if q fixed},
    \begin{equation*}
        \expec{ \prod_{0\le i<j\le \ell}
        \left( \frac{q-1}{\sqrt q \log q} \,\frac{1}{\sqrt n}
        \,\dist_n{\left( W_{n,i} , W_{n,j} \right)} \right)^{r_{i,j}}
        }
        \cv{n\to\infty}
        \expec{ \prod_{0\le i<j\le \ell}
        \left(\big. \dist{\left( W_{i} , W_{j} \right)} \right)^{r_{i,j}}
        } \,.
    \end{equation*}
\end{proposition}

The derivation of \Cref{th: CRT limit if q fixed} from Proposition~\ref{prop: joint moments cv to CRT} is straightforward.
Indeed, the distributions of distance matrices in the CRT are characterized by their joint moments (this follows, e.g., from the fact that the depth of each point $W_i$ is Rayleigh-distributed).
Therefore, Proposition~\ref{prop: joint moments cv to CRT} implies the convergence in distribution of distance matrices for all $\ell\ge1$.
Then, by Proposition~\ref{prop: convergence to a random dendron}, this is equivalent to the convergence in distribution towards the CRT in the GP topology.

In the upcoming sections, we prove Proposition~\ref{prop: joint moments cv to CRT} by induction on $\ell$ and $(r_{i,j})_{0\le i<j\le \ell}$.
Note that the case $\ell=1$ has already been established in Proposition~\ref{prop: Rayleigh limit of typical depth when q fixed}.
Our approach is to recursively pump the desired moments from their generating functions, relying heavily on singularity analysis.
This is analogous to our proof of Proposition~\ref{prop: Rayleigh limit of typical depth when q fixed}, but with heavier computations.
To simplify the argument, we will also use the fact that Proposition~\ref{prop: joint moments cv to CRT} is already known to hold for $q=1$ (see, e.g., \cite{Aldous_1993}).

\subsection{The generating functions of distance matrices}

We start with some notation.
Let $\ell\in\bN$.
We let $\pairs_\ell := \binom{[0,\ell]}{2}$ denote the set of unordered pairs $\{i,j\}$ with $0\le i < j\le \ell$.
Throughout, $\by = (y_{i,j})_{\{i,j\}\in \pairs_\ell}$ is a family of complex numbers with modulus $\abs{y_{i,j}}\le1$,  
and $\br = (r_{i,j})_{\{i,j\}\in \pairs_\ell}$ is a family of nonnegative integers. %with the same indexation.
We let $\sum \br := \sum_{\{i,j\}\in \pairs_\ell} r_{i,j}$.
We write $\bm{0}$ and $\bm{1}$ for the families of all $0$'s and all $1$'s with the same indices.

For $I\subseteq [0,\ell]$, we let $\pairs_\ell^{(I)}$ denote the set of pairs in $\pairs_\ell$ that contain at least one element of $I$.
If $I=\{i\}$ for some $i\in[0,\ell]$, we simply write $\pairs_\ell^{(i)}$.
We say that a family of nonnegative integers $\bm{\alpha} = (\alpha_{j,j'})_{\{j,j'\}\in \pairs_\ell} \in \bN_0^{\pairs_\ell}$ is supported on $\pairs_\ell^{(i)}$ if $\alpha_{j,j'}=0$ for all $\{j,j'\} \in \binom{[0,\ell] \setminus \{i\}}{2}$.

For a (possibly labeled) tree $t$ with root $w_0 := \rho(t)$, let
\[
G_\ell (t,\by) := \sum_{w_1,\dots,w_\ell \in t} \prod_{\{i,j\}\in \pairs_\ell} y_{i,j}^{\dist(w_i,w_j)}\,.
\]
We apply the convention $G_0(t) := 1$.
Now define
\[
F_\ell(x,\by) := \sum_{t\in\bT} q^\des{t} \frac{x^{\abs{t}}}{\abs{t}!} G_\ell(t,\by) \,.
\]
Note that $F_0(x) = A(x)$ and that $[x^n]F_\ell(x,\bm{1}) = n^\ell [x^n] A(x)$.
Let $\partial_\by^\br$ denote the operator that differentiates $r_{i,j}$ times with respect to $y_{i,j}$ for all $0\le i < j\le \ell$ (when Schwarz's theorem applies).
Define
\begin{equation*}
    f_{\ell,\br}(x) := \partial_\by^\br F_\ell(x,\by) \vert_{\by=\bm{1}} 
    = \sum_{t\in \bT} q^\des{t} \frac{x^{\abs{t}}}{\abs{t}!} \sum_{w_1,\dots,w_\ell\in t} \prod_{\{i,j\}\in \pairs_\ell} \left(\big. \dist(w_i,w_j) \right)_{r_{i,j}}\,,
\end{equation*}
where $(m)_k = m(m-1)\cdots(m-k+1)$ is the falling factorial, and note that
\begin{equation}\label{eq: joint moment of distances from extracted coefficients}
    \frac{1}{n^\ell} \frac{[x^n]f_{\ell,\br}(x)}{[x^n]A(x)} 
    = \expec{ \prod_{\{i,j\}\in \pairs_\ell} \left(\big. \dist(W_{n,i},W_{n,j}) \right)_{r_{i,j}} }\,,
\end{equation}
where $W_{n,1}, \dots, W_{n,\ell}$ are uniformly random i.i.d.~vertices in $\cT_n^{(q)}$ and $W_{n,0}$ denotes the root of $\cT_n^{(q)}$. 
Recall that $A(x)$ has the dominant singularity $x_0 := q^{q/(1-q)}$ and is amenable to singularity analysis.
We will show that %, for any $\by \in \bC^{\pairs_\ell}$ such that $\abs{y_{i,j}}\le1$ for all $0\le i,j\le \ell$, and 
for any $\br \in \bN_0^{\pairs_\ell}$, the functions $f_{\ell,\br}(x)$ are amenable to singularity analysis with the same dominant singularity $x_0$.

We will proceed as follows.
First, using the usual decomposition method for descent-biased trees, we prove in \Cref{sec: Recursive differential equation} that $F_{\ell}$ and $f_{\ell,\br}$ satisfy intricate yet exploitable recursive differential equations.
Then in \Cref{sec: Pumping the moments}, we use singularity analysis to find the dominant terms in our equations, and derive recursive formulas for \eqref{eq: joint moment of distances from extracted coefficients}.
Rather than explicitly solving these, we extract the dependence on $q$ in our coefficients. 
Thanks to Proposition~\ref{prop: joint moments cv to CRT} for $q=1$, this will complete the proof of Proposition~\ref{prop: joint moments cv to CRT} for general $q\in(0,1]$ in \Cref{sec: Identifying the CRT}.

\subsection{Recursive differential equations}
\label{sec: Recursive differential equation}

Recall from \Cref{sec: Generating functions related to descent-biased trees} that a tree $t\in\bT$ of size $\abs{t} = n+1$ can be decomposed with respect to its vertex $v_1$ labeled $1$.

\paragraph{First case:}
$v_1=\rho(t)$.
Write $t_1, \dots, t_p$ for the subtrees rooted at the children of $v_1$, ordered in an arbitrary manner.
For each $k\in[p]$, let $\rho_k$ and $\dist_{t_k}$ denote the root and graph distance of $t_k$, respectively.

Let $w_1,\dots,w_\ell \in t$, and define $I_k := \left\{ i\in[0,\ell] : w_i\in t_k \right\}$ for $k\in[1,p]$, and $I_{-1} := \left\{ i\in[0,\ell] : w_i = v_1 \right\}$.
Note that $0\in I_{-1}$, and that $\cI := \left( I_{-1}, I_1, \dots, I_p \right)$ is an ordered partition of $[0,\ell]$.
Now, for $0\le i,j\le \ell$,
\[
\dist_t(w_i,w_j) = 
\left\{
\begin{array}{llll}
    0 & \text{if } i,j\in I_{-1} ;\\
    1 + \dist_{t_k}(\rho_k,w_i) & \text{if } i\in I_k, j\in I_{-1} \text{ for some } k\ge1 ;\\
    \dist_{t_k}(w_i,w_j) & \text{if } i,j\in I_k \text{ for some } k\ge1 ;\\
    2 + \dist_{t_k}(\rho_k,w_i) + \dist_{t_m}(\rho_m,w_j) & \text{if } i\in I_k, j\in I_m \text{ for some } k,m\ge1, k \ne m \,.\\
\end{array}
\right.
\]
This allows us to write
\begin{equation}\label{eq: decomposition G first case}
    G_\ell(t,\by) = \sum_{\cI \in \cP_{1,p}} 
    \left( \prod_{i\in I_{-1}, j\notin I_{-1}} y_{i,j} \right)
    \left( \prod_{\{k,m\} \in \binom{[1,\ell]}{2}} \prod_{i\in I_k, j\in I_m} y_{i,j}^2 \right)
    \prod_{1\le k\le p} G_{\abs{I_k}}\left(\big. t_k , \varphi_{I_k}(\by) \right) \,.
\end{equation}
Here, the set $\cP_{1,p}$ contains ordered partitions $\cI = (I_{-1}, I_1, \dots, I_p)$ of $[0,\ell]$ where $0\in I_{-1}$ and all other blocks may be empty.
The family of variables $\varphi_{I_k}(\by)$ is indexed by unordered pairs in $\{0\} \cup I_k$, where the latter set is seen as isomorphic to $[0,\abs{I_k}]$.
It is defined by $\varphi_{I_k}(\by)_{i,j} := y_{i,j}$ and $\varphi_{I_k}(\by)_{0,j} := \prod_{h\in [0,\ell] \setminus I_k} y_{h,j}$ for $i,j\in I_k$.

\paragraph{Second case:}
$v_1\ne \rho(t)$.
Write $t_0$ for the subtree of $t$ obtained by deleting $v_1$ and its descendants, and let $v_0 \in t_0$ be the parent of $v_1$.
As before, write $t_1, \dots, t_p$ for the subtrees rooted at the children of $v_1$, ordered arbitrarily.

Let $w_1,\dots,w_\ell \in t$, and define $I_k := \left\{ i\in[0,\ell] : w_i\in t_k \right\}$ for $k\in[0,p]$, and $I_{-1} := \left\{ i\in[0,\ell] : w_i = v_1 \right\}$.
Note that $0\in I_0$ now, since $w_0$ is the root of $t$ and thus in $t_0$.
Now, for $0\le i,j\le \ell$,
\[
\dist_t(w_i,w_j) = 
\left\{
\begin{array}{llll}
    0 & \text{if } i,j\in I_{-1} ;\\
    1 + \dist_{t_0}(v_0,w_i) & \text{if } i\in I_0, j\in I_{-1} ;\\
    1 + \dist_{t_k}(\rho_k,w_i) & \text{if } i\in I_k, j\in I_{-1} \text{ for some } k\ge1 ;\\
    \dist_{t_k}(w_i,w_j) & \text{if } i,j\in I_k \text{ for some } k\ge0 ;\\
    2 + \dist_{t_0}(v_0,w_i) + \dist_{t_k}(\rho_k,w_j) & \text{if } i\in I_0, j\in I_k \text{ for some } k\ge1 \,;\\
    2 + \dist_{t_k}(\rho_k,w_i) + \dist_{t_m}(\rho_m,w_j) & \text{if } i\in I_k, j\in I_m \text{ for some } k,m\ge1, k \ne m \,.\\
\end{array}
\right.
\]
This allows us to write
\begin{align}\label{eq: decomposition G second case v0 fixed}
    G_\ell(t,\by) = \sum_{(I_{-1}, I_0, \dots, I_p)} &
    \left( \prod_{i\in I_{-1}, j\notin I_{-1}} y_{i,j} \right)
    \left( \prod_{\{k,m\} \in \pairs_\ell} \prod_{i\in I_k, j\in I_m} y_{i,j}^2 \right)
    \left( \prod_{1\le k\le p} G_{\abs{I_k}}\left(\big. t_k , \varphi_{I_k}(\by) \right) \right) \nonumber\\
    & \cdot \sum_{w_i\in t_0 \text{ for } i\in I_0\setminus\{0\}}
    \left( \prod_{ \{i,j\} \in \binom{I_0}{2}} y_{i,j}^{\dist(w_i,w_j)} \right)
    \left( \prod_{i\in I_0, j\notin I_0} y_{i,j}^{\dist(w_i,v_0)} \right) \,.
\end{align}
The sum runs over ordered partitions $(I_{-1}, I_0, I_1, \dots, I_p)$ of $[0,\ell]$ where $0\in I_0$ and all other blocks may be empty.
The family of variables $\varphi_{I_k}(\by)$ is defined as before.
Note that the last product appearing in \eqref{eq: decomposition G second case v0 fixed} equals $1$ if $I_0=[0,\ell]$, and is non-empty otherwise.

Now, we will sum over the choice of $v_0\in t_0$.
Instead of fixing $t$, fix a non-empty tree $t_0$, an integer $p\ge0$, and a family of possibly empty trees $t_1, \dots, t_p$.
For each choice of $v_0\in t_0$, define the tree $t$ by attaching a vertex $v_1$ to $v_0$, and subtrees $t_1, \dots, t_p$ to $v_1$.
Then,
\begin{align}\label{eq: decomposition G second case v0 summed}
    \sum_{v_0\in t_0} G_\ell (t,\by) &= \sum_{(I_{-1}, I_0, \dots, I_p)} 
    \left( \prod_{i\in I_{-1}, j\notin I_{-1}} y_{i,j} \right)
    \left( \prod_{\substack{i\in I_k, j\in I_m \\ k,m\ge0,\,k \ne m}} y_{i,j}^2 \right) \nonumber \\ &\qquad \qquad \qquad
    \left( \prod_{1\le k\le p} G_{\abs{I_k}}\left(\big. t_k , \varphi_{I_k}(\by) \right) \right) G_{\abs{I_0}}\left(\big. t_0 , \psi_{I_0}(\by) \right)
\end{align}
where the sum runs over the same partitions as for \eqref{eq: decomposition G second case v0 fixed}, and the family of variables $\psi_{I_0}(\by)$ is indexed by pairs in $I_0 \cup \{\ell+1\}$, where the latter set is seen as isomorphic to $[0,\abs{I_0}]$.
It is defined by $\psi_{I_0}(\by)_{i,j} = y_{i,j}$ and $\psi_{I_0}(\by)_{i,\ell+1} = \prod_{h\in[1,\ell] \setminus I_0} y_{i,h}$ for $i,j\in I_0$.\footnote{
The definition of $\psi$ differs slightly from that of $\varphi$. 
Indeed, the new index added to $I_k$ for $\varphi_{I_k}(\by)$ is related to the root $w_0$, which is fixed, while the new index added to $I_0$ for $\psi_{I_0}(\by)$ is related to a typical vertex in the tree.
}

It will be useful to isolate one specific term in \eqref{eq: decomposition G second case v0 summed}.
When $\abs{I_0}=\ell+1$, that is when $I_0=[0,\ell]$, we have $\psi_{I_0}(\by)_{i,\ell+1} = 1$ for all $i\in I_0$.
Therefore, we can write $G_{\ell+1}\left(\big. t_0 , \psi_{I_0}(\by) \right) = \abs{t_0} G_\ell(t_0,\by)$, and thus
\begin{equation}\label{eq: decomposition G second case v0 summed, isolated term}
    \sum_{v_0\in t_0} G_\ell (t,\by) = \abs{t_0} G_\ell(t_0,\by)
    + \sum_{\cI \in \cP_{2,p}} 
    M_\cI(\by)
    \left( \prod_{1\le k\le p} G_{\abs{I_k}}\left(\big. t_k , \varphi_{I_k}(\by) \right) \right) G_{\abs{I_0}}\left(\big. t_0 , \psi_{I_0}(\by) \right)
\end{equation}
where the set $\cP_{2,p}$ contains ordered partitions $\cI = (I_{-1}, I_0, I_1, \dots, I_p)$ of $[0,\ell]$ such that $0\in I_0$, $I_0\ne [0,\ell]$, and all other blocks may be empty.
The term $M_\cI(\by)$ is a monomial that only depends on the partition $\cI$:
\begin{equation*}
    M_\cI(\by) := \left( \prod_{i\in I_{-1}, j\notin I_{-1}} y_{i,j} \right)
    \left( \prod_{\substack{i\in I_k, j\in I_m \\ k,m\ge0,\,k \ne m}} y_{i,j}^2 \right) \,.
\end{equation*}
If $\cI$ belongs to $\cP_{1,p}$ instead, $\cI = (I_{-1}, I_1, \dots, I_p)$, we interpret $I_0$ as being empty and $M_\cI(\by)$ appears inside \eqref{eq: decomposition G first case}.

\medskip

We are now equipped to deduce a differential equation for $F_\ell$.
Below, if $t_0,\dots,t_p$ are (possibly labeled) trees and $v_0\in t_0$, we write $(t_1, \dots, t_p)$ for the unlabeled tree with a root to which the subtrees $t_1, \dots, t_p$ are attached, and we write $(t_0, v_0, t_1, \dots, t_p)$ for the unlabeled tree obtained by attaching $(t_1, \dots, t_p)$ to $v_0$ in $t_0$.
Then, we have
\begin{align*}
    \partial_x F_\ell(x,\by)
    &= \sum_{n\ge0} \sum_{t\in\bT_{n+1}} q^\des{t} \frac{x^n}{n!} G_\ell(t,\by) \\
    &= \sum_{n\ge0} \sum_{p\ge0} \sum_{\substack{t_1, \dots, t_p\in \bT\\ \abs{t_1}+\dots+\abs{t_p}=n}}
    \frac{1}{p!} \binom{n}{\abs{t_1},\dots,\abs{t_p}} q^{\des{t_1}+\dots+\des{t_p}} \frac{x^n}{n!} G_\ell\left(\big.(t_1,\dots,t_p), \by\right) \\
    &\hspace{3em} + \sum_{\substack{t_0, \dots, t_p\in\bT\\ \abs{t_0}+\dots+\abs{t_p}=n}}
    \sum_{v_0\in t_0} \frac{1}{p!} \binom{n}{\abs{t_0},\dots,\abs{t_p}} q^{1+\des{t_0}+\dots+\des{t_p}} \frac{x^n}{n!} G_\ell\left(\big.(t_0,v_0,t_1,\dots,t_p), \by\right) \\
    &= \sum_{p\ge0} \sum_{t_1, \dots, t_p\in \bT} \frac{1}{p!} q^{\des{t_1}+\dots+\des{t_p}} \frac{x^{\abs{t_1}}}{\abs{t_1}!} \cdots \frac{x^{\abs{t_p}}}{\abs{t_p}!} G_\ell\left(\big.(t_1,\dots,t_p), \by\right) \\
    &\hspace{3em} + \sum_{t_0, \dots, t_p\in\bT} \frac{1}{p!} q^{1+\des{t_0}+\dots+\des{t_p}} \frac{x^{\abs{t_0}}}{\abs{t_0}!} \cdots \frac{x^{\abs{t_p}}}{\abs{t_p}!} \sum_{v_0\in t_0} G_\ell\left(\big.(t_0,v_0,t_1,\dots,t_p), \by\right) \,.
\end{align*}
Thus, by \eqref{eq: decomposition G first case} and \eqref{eq: decomposition G second case v0 summed, isolated term},
\begin{align}\label{eq: equadiff on Fell before reduced}
    \partial_x F_\ell(x,\by)
    = \sum_{p\ge0} &\sum_{\cI\in \cP_{1,p}} \frac{1}{p!} 
    M_\cI(\by)
    \prod_{1\le k\le p} F_{\abs{I_k}}\left(\big. x , \varphi_{I_k}(\by) \right) \nonumber\\
    &+ \frac{1}{p!} q A(x)^p x \partial_x F_\ell(x,\by) \nonumber\\
    &+ q \sum_{\cI\in \cP_{2,p}} 
    \frac{1}{p!}
    M_\cI(\by)
    F_{\abs{I_0}}\left(\big. x , \psi_{I_0}(\by) \right)
    \prod_{1\le k\le p} F_{\abs{I_k}}\left(\big. x , \varphi_{I_k}(\by) \right) \,.
\end{align}
If $\cI$ is a partition in $\cP_{1,p}$ or $\cP_{2,p}$, let $\cJ$ be its \enquote{reduced} form obtained by deleting empty blocks, except $I_{-1}$. 
If $\cI\in\cP_{1,p}$, then we can write $\cJ = (J_{-1}, J_1, \dots, J_s)$ for some $s\in[0,\ell]$.
This is an ordered partition of $[0,\ell]$ with non-empty blocks and $0\in J_{-1}$.
Write $\cP_{1,s}^\red$ for the set of such partitions.
If $\cI\in\cP_{2,p}$, then we can write $\cJ = (J_{-1}, J_0, J_1, \dots, J_s)$ for some $s\in[0,\ell]$.
This is an ordered partition of $[0,\ell]$ with non-empty blocks, except possibly $J_{-1}$, and $0\in J_0$ as well as $J_0\ne [0,\ell]$.
Write $\cP_{2,s}^\red$ for the set of such partitions.

Let us collect a few observations.
First, the sets $\cup_{s\ge0} \cP_{1,s}^\red$ and $\cup_{s\ge0} \cP_{2,s}^\red$ are now finite (indeed, $s\le \ell$).
Secondly, given a reduced partition $\cJ\in\cP_{i,s}^\red$ for arbitrary $i\in\{1,2\}$, there are $\binom{p}{s}$ partitions $\cI\in\cP_{i,p}$ with this reduced form.
Thirdly, the monomial $M_\cI(\by)$ depends only on the reduced partition $\cJ$.
And finally, $\prod_{1\le k\le p} F_{\abs{I_k}}\left(\big. x , \varphi_{I_k}(\by) \right)$ can be rewritten as $A(x)^{p-s} \prod_{1\le k\le s} F_{\abs{J_k}}\left(\big. x , \varphi_{J_k}(\by) \right)$.
Therefore, by \eqref{eq: equadiff on Fell before reduced},
\begin{align*}
    \partial_x F_\ell(x,\by)
    = q e^{A(x)} x \partial_x F_\ell(x,\by)
    &+ \sum_{p\ge s\ge 0} \sum_{\cJ\in \cP_{1,s}^\red} \frac{1}{p!} \binom{p}{s} 
    M_\cJ(\by)
    A(x)^{p-s} \prod_{1\le k\le s} F_{\abs{J_k}}\left(\big. x , \varphi_{J_k}(\by) \right) \\
    &+ q \sum_{p\ge s\ge 0} \sum_{\cJ\in \cP_{2,s}^\red} \frac{1}{p!} \binom{p}{s} 
    M_\cJ(\by)
    F_{\abs{J_0}}{\left(\big. x , \psi_{J_0}(\by) \right)}
    A(x)^{p-s} \!\!\! \prod_{1\le k\le s} F_{\abs{J_k}}{\left(\big. x , \varphi_{J_k}(\by) \right)} 
\end{align*}
and thus
\begin{align}\label{eq: equadiff on Fell with everything}
    \left( e^{-A(x)} - qx \right)
    \partial_x F_\ell(x,\by)
    = &\sum_{s=0}^\ell \sum_{\cJ\in \cP_{1,s}^\red} \frac{1}{s!} 
    M_\cJ(\by)
    \prod_{1\le k\le s} F_{\abs{J_k}}\left(\big. x , \varphi_{J_k}(\by) \right) \nonumber\\
    &+ q \sum_{s=0}^\ell \sum_{\cJ\in \cP_{2,s}^\red} \frac{1}{s!} 
    M_\cJ(\by)
    F_{\abs{J_0}}\left(\big. x , \psi_{J_0}(\by) \right)
    \prod_{1\le k\le s} F_{\abs{J_k}}\left(\big. x , \varphi_{J_k}(\by) \right) \,.
\end{align}
Recall from \eqref{eq: differential equation for A} that $\left(\partial_x A(x) \right) \left( e^{-A(x)} - qx \right) = 1$.
In order to work with \eqref{eq: equadiff on Fell with everything}, we first need to identify in which terms the function $F_\ell$ appears.
In the first double sum of \eqref{eq: equadiff on Fell with everything}, $F_\ell$ only appears for the reduced partition $\cJ\in\cP_{1,1}^\red$ where $J_{-1}=\{0\}$ and $J_1=[1,\ell]$.
In the second double sum of \eqref{eq: equadiff on Fell with everything}, $F_\ell$ appears for three families of reduced partitions.
The first family consists of partitions $\cJ\in\cP_{2,0}^\red$ with $J_{-1}=\{i\}$ and $J_0=[0,\ell]\setminus\{i\}$ for some $i\in[1,\ell]$.
The second family consists of partitions $\cJ\in\cP_{2,1}^\red$ with $J_0=[0,\ell]\setminus\{i\}$ and $J_{1}=\{i\}$ for some $i\in[1,\ell]$.
The third family contains the single partition $\cJ\in\cP_{2,1}^\red$ with $J_0=\{0\}$ and $J_{1}=[1,\ell]$.
Therefore, we can write:
\begin{align}\label{eq: equadiff on Fell with Q and R}
    \partial_x F_\ell(x,\by) = \left(\partial_x A(x) \right)
    \left(\big. Q_\ell(x,\by) F_\ell(x,\by) + R_\ell(x,\by) \right) \,,
\end{align}
where
\begin{align}\label{eq: def Q}
    Q_\ell(x,\by) := \prod_{i\ne0} y_{0,i}
    + q \sum_{i=1}^\ell \prod_{j\ne i} y_{i,j}
    + q \sum_{i=0}^\ell F_1{\left(x, \prod_{j\ne i} y_{i,j} \right)} \prod_{j\ne i} y_{i,j}^2 \,,
\end{align}
and $R_\ell$ is defined as the appropriate remainder, containing only $F_m$'s with $m<\ell$ (see \eqref{eq: def Rell first line} and \eqref{eq: def Rell second line}).

\subsection{Pumping the moments}
\label{sec: Pumping the moments}

Our goal here is to prove that for any $\ell\in\bN$ and for any family of integers $\br = (r_{i,j})_{\{i,j\}\in \pairs_\ell} \in \bN_0^{\pairs_\ell}$, there exists a positive constant $C_{\ell,\br} > 0$ such that
\begin{equation}\label{eq: claim singularity fellr}
    f_{\ell,\br}(x) \sim C_{\ell,\br} (1-x/x_0)^{\frac12 - \ell - \frac12 \sum \br}
\end{equation}
as $x\to x_0$.
We proceed by induction, using the lexicographic order on $(\ell,\br)$ and the well-founded partial order on $\bN_0^{\pairs_\ell}$ defined by
\begin{equation*}
    \bm{\alpha} \prec \br
    \qquad\text{if and only if}\qquad
    \bm{\alpha}\ne\br 
    \text{ and for all }\{i,j\}\in \pairs_\ell,\:
    \alpha_{i,j} \le r_{i,j} \,.
\end{equation*}
We write $\bm{\alpha} \preceq \br$ if $\bm{\alpha} \prec \br$ or $\bm{\alpha} = \br$.
We also write $\binom{\br}{\bm{\alpha}} := \prod_{\{i,j\}\in \pairs_\ell} \binom{r_{i,j}}{\alpha_{i,j}}$.

\paragraph{Initialization.}
First, we recall that \eqref{eq: claim singularity fellr} has been proved for $\ell=1$ and all $\br = r \ge0$ in Proposition~\ref{prop: Rayleigh limit of typical depth when q fixed}, \eqref{eq: claim singularity f1r}.

Now, we prove that it holds for all $\ell\ge2$ and $\br = \bm{0}$.
We have $[x^n] F_\ell(x,\bm{1}) = n^\ell [x^n] A(x)$.
Therefore, we can write $f_{\ell,\bm{0}}(x) = F_\ell(x,\bm{1}) = x^\ell \partial_x^\ell A(x) + B(x)$, where $B$ is a linear combination of lower-order derivatives of $A$.
Differentiating the asymptotic formula~\eqref{eq: asymptotic A close to singularity for fixed q}, we get
\begin{equation}\label{eq: singularity for multiple derivatives of A}
\partial_x^\ell A(x) 
%\sim \sqrt{\frac{2}{q}} \frac{1}{x_0^\ell} \frac12 \cdot \frac12 \cdot \frac32 \dots \cdot \left( \ell - \frac32 \right) (1-x/x_0)^{1/2 - \ell}
\sim \frac{1}{x_0^\ell \sqrt{2q}} \frac{(2\ell-3)!!}{2^{\ell-1}} (1-x/x_0)^{1/2 - \ell}
\end{equation}
as $x\to x_0$, with the convention $(2\ell-3)!! := 1$ when $\ell=1$.
Hence, \eqref{eq: claim singularity fellr} holds with $C_{\ell,\bm{0}} = \frac{1}{\sqrt{2q}} \frac{(2\ell-3)!!}{2^{\ell-1}}$.

\paragraph{Induction step.}
Now, fix $(\ell,\br)$ with $\ell\ge2$ and $\sum\br > 0$.
Assume that $\eqref{eq: claim singularity fellr}$ holds for all the functions $f_{m,\bm{\alpha}}$ with $m<\ell$ and all the functions $f_{\ell,\bm{\alpha}}$ with $\bm{\alpha}\prec\br$.

From \eqref{eq: equadiff on Fell with Q and R}, we obtain
\begin{align*}
    \partial_x \partial_\by^\br F_\ell(x,\by) &= 
    (\partial_x A(x)) Q_\ell(x,\by) (\partial_\by^\br F_\ell(x,\by))
    \\&\quad + (\partial_x A(x)) \left( 
    \left(\Big. \partial_\by^\br \left[\big. Q_\ell(x,\by) F_\ell(x,\by) \right]
    - Q_\ell(x,\by) \big(\partial_\by^\br F_\ell(x,\by)\big) \right)
    + \partial_\by^\br R_\ell(x,\by)
    \right) \,,
\end{align*}
and at $\by=\bm{1}$, we write it as
\begin{equation}\label{eq: equadiff on fellr with Q and Rtilde}
    \partial_x f_{\ell,\br}(x)
    = G_\ell'(x) f_{\ell,\br}(x)
    + \wR_{\ell,\br}(x)\,,
\end{equation}
where 
\begin{equation}\label{eq: def Gell prime and Rellr tilde}
\begin{cases}
    G_\ell'(x) &= \partial_x A(x) Q_\ell(x,\bm{1}) = \partial_x A(x) \left(\big. 1 + q\ell + q(\ell+1)f_{1,0}(x) \right) \,;\\
    \wR_{\ell,\br}(x) &= \partial_x A(x) \left( 
    \left(\Big. \partial_\by^\br {\left[\big. Q_\ell(x,\by) F_\ell(x,\by) \right]}\vert_{\by=\bm{1}}
    - Q_\ell(x,\bm{1}) f_{\ell,\br}(x) \right)
    + \partial_\by^\br R_\ell(x,\by)\vert_{\by=\bm{1}}
    \right) \,.
\end{cases}
\end{equation}
The function $G_\ell$ is defined as an appropriate primitive of $G_\ell'$, to be specified later.
As $\ell\ge2$, \eqref{eq: equadiff on fellr with Q and Rtilde} is a linear ODE for $f_{\ell,\br}$ which can be solved  using variation of parameters:
\begin{equation}\label{eq: solution fellr with variation of constant}
    f_{\ell,\br}(x) = e^{G_\ell(x)} \int_0^x \wR_{\ell,\br}(u) e^{-G_\ell(u)} \mathrm{d}u \,.
\end{equation}
Thus, we need to study the asymptotics of the functions $G_\ell$ and $\wR_{\ell,\br}$ close to $x_0$.
First, plugging \eqref{eq: singularity for single derivative of A and for f10} into \eqref{eq: def Gell prime and Rellr tilde} yields
\begin{equation*}
    G_\ell'(x) = \frac{\ell+1}{2x_0} (1-x/x_0)^{-1}
    + \cO\left( \abs{x-x_0}^{-1/2} \right)
\end{equation*}
as $x\to x_0$.
Subsequently, for the appropriate choice of $G_\ell$,
\begin{equation}\label{ eq: singularity for expGell}
    G_\ell(x) = -\frac12(\ell+1) \log(1-x/x_0) + \cO\left( \abs{x-x_0}^{1/2} \right),
    \quad\text{and thus}\quad
    e^{G_\ell(x)} \sim (1-x/x_0)^{-\frac12(\ell+1)}
\end{equation}
as $x\to x_0$.
Let us now turn to $\wR_{\ell,\br}$.
Its analysis is contained in the following two lemmas, whose proofs are postponed.

\begin{lemma}\label{lem: asymptotics first term Rellr tilde}
    For $\ell\ge2$ and $\sum\br > 0$, we have
    \begin{equation*}
        \partial_\by^\br {\left[\big. Q_\ell(x,\by) F_\ell(x,\by) \right]}\vert_{\by=\bm{1}}
        - Q_\ell(x,\bm{1}) f_{\ell,\br}(x)
        \;\sim\;
        q\, C_{\ell,\br}^{(1)}\,
        (1-x/x_0)^{-\ell - \frac12\sum\br}
    \end{equation*}
    as $x\to x_0$, where
    \begin{equation*}
        C_{\ell,\br}^{(1)} := \sum_{i=0}^\ell \sum_{\substack{0\prec \bm{\alpha}\preceq \br \\ \bm{\alpha} \text{ supported on } \pairs_\ell^{(i)}}} 
        \binom{\br}{\bm{\alpha}}
        C_{\ell, \br-\bm{\alpha}} C_{1,\sum\bm{\alpha}} 
    \end{equation*}
    is a positive constant.
\end{lemma}

\begin{lemma}\label{lem: asymptotics second term Rellr tilde}
    If $\ell=2$, then as $x\to x_0$,
    \begin{equation*}
        \partial_\by^\br R_\ell(x,\by) \vert_{\by=\bm{1}}
        = \cO\left( (1-x/x_0)^{\frac12-\ell-\frac12\sum\br} \right).
    \end{equation*}
    If $\ell\ge3$, then as $x\to x_0$,
    \begin{equation*}
        \partial_\by^\br R_\ell(x,\by) \vert_{\by=\bm{1}}
        \;\sim\;
        q\, C_{\ell,\br}^{(2)}\,
        (1-x/x_0)^{-\ell-\frac12\sum\br},
    \end{equation*}
    where
    \begin{equation*}
        C_{\ell,\br}^{(2)} 
        = \sum_{\substack{ J_0 \sqcup J_1 = [0,\ell] \\ 0\in J_0,\, \abs{J_0}<\ell,\, \abs{J_1}<\ell }}
        \sum_{\substack{ \bm{0}\preceq \bm{\alpha}\preceq \br \\ \alpha_{i,j}=r_{i,j} \text{ for all } i,j\in J_0 \\  \alpha_{i,j}=0 \text{ for all } i,j\in J_1 }}
        \binom{\br}{\bm{\alpha}}
        C_{\abs{J_0}, \bm\alpha^{J_0,\ell+1}}
        C_{\abs{J_1}, (\br-\bm\alpha)^{0,J_1}} \,.
    \end{equation*}
    Here, for $(J_0,J_1)$ as above and $\bm{\alpha} \in \bN_0^{\pairs_\ell}$, we let $\bm{\alpha}^{J_0,\ell+1}$ denote the element of $\bN_0^{\pairs_{\abs{J_0}}}$ defined by identifying $[0,\abs{J_0}]$ with $J_0 \sqcup \{\ell+1\}$ (by the unique order-preserving bijection) and setting $\alpha^{J_0,\ell+1}_{i,j} := \alpha_{i,j}$ and $\alpha^{J_0,\ell+1}_{i,\ell+1} := \sum_{h\in J_1} \alpha_{i,h}$ for all $i,j\in J_0$.
    Moreover, we let $\bm{\alpha}^{0,J_1}$ denote the element of $\bN_0^{\pairs_{\abs{J_1}}}$ defined by identifying $[0,\abs{J_1}]$ with $\{0\} \sqcup J_1$ and setting $\alpha^{0,J_1}_{i,j} = \alpha_{i,j}$ and $\alpha^{0,J_1}_{0,j} = \sum_{h\in J_0} \alpha_{h,j}$ for all $i,j\in J_1$.
    We set $C_{\ell,\br}^{(2)} := 0$ if $\ell=2$.
\end{lemma}

Before proving these two lemmas, let us see how to conclude the proof of \eqref{eq: claim singularity fellr}.
Plugging \eqref{eq: singularity for single derivative of A and for f10} and Lemmas~\ref{lem: asymptotics first term Rellr tilde} and~\ref{lem: asymptotics second term Rellr tilde} into \eqref{eq: def Gell prime and Rellr tilde}, we get
\begin{equation*}
    \wR_{\ell,\br}(x) \sim
    \frac{1}{x_0} \sqrt{\frac{q}{2}} \left( C_{\ell,\br}^{(1)} + C_{\ell,\br}^{(2)} \right) (1-x/x_0)^{-\frac12 - \ell - \frac12 \sum \br}
\end{equation*}
as $x\to x_0$.
Subsequently, using \eqref{ eq: singularity for expGell},
\begin{equation*}
    \int_0^x \wR_{\ell,\br}(u) e^{-G_\ell(u)} \mathrm{d}u \sim
    \frac{1}{\frac12\ell - 1 + \frac12\sum \br} \sqrt{\frac{q}{2}} \left( C_{\ell,\br}^{(1)} + C_{\ell,\br}^{(2)} \right) (1-x/x_0)^{1 - \frac12\ell - \frac12 \sum \br} \,.
\end{equation*}
Together with \eqref{eq: solution fellr with variation of constant}, this yields
\begin{equation}\label{eq: recursive singularity for fellr}
    f_{\ell,\br}(x) \sim
    \frac{\sqrt{2q}}{\ell - 2 + \sum \br} \left( C_{\ell,\br}^{(1)} + C_{\ell,\br}^{(2)} \right) (1-x/x_0)^{\frac12 - \ell - \frac12 \sum \br} 
\end{equation}
as $x\to x_0$.
This finishes the proof of \eqref{eq: claim singularity fellr} by induction, and yields an explicit recursive formula for $C_{\ell,\br}$.
% \begin{equation}\label{eq: recursion for Cellr}
%     C_{\ell,r} = \frac{\sqrt{2q}}{\ell - 2 + \sum r} \left(\big. \sigma(\cC_{\ell,r}) + \tau(\cC_{\ell,r}) \right) 
% \end{equation}
% which holds for all $\ell\ge2$ and $r\succeq\bm{0}$ such that $r\ne \bm{0}$.
All that remains is to prove Lemmas~\ref{lem: asymptotics first term Rellr tilde} and~\ref{lem: asymptotics second term Rellr tilde}, still within the induction step.

\begin{proof}[Proof of Lemma~\ref{lem: asymptotics first term Rellr tilde}]
Observe that
\begin{equation}\label{eq: Leibniz formula for product Qell Fell}
    \partial_\by^\br \left[\big. Q_\ell(x,\by) F_\ell(x,\by) \right]
    - Q_\ell(x,\by) \big( \partial_\by^\br F_\ell(x,\by) \big)
    = \sum_{\bm{0} \prec \bm{\alpha} \preceq \br} 
    \left( \prod_{\{i,j\}\in \pairs_\ell} \binom{r_{i,j}}{\alpha_{i,j}} \right) 
    %\binom{\br}{\bm{\alpha}}
     \big(\partial_\by^{\bm{\alpha}} Q_\ell(x,\by) \big) \big( \partial_\by^{\br-\bm{\alpha}} F_\ell(x,\by) \big)
\end{equation}
where the sum runs overs families of nonnegative integers $\bm{\alpha} = (\alpha_{i,j})_{\{i,j\}\in \pairs_\ell}$ such that $\bm{\alpha}\ne \bm{0}$ and $0\le \alpha_{i,j}\le r_{i,j}$ for all $\{i,j\}$, and $\br-\bm{\alpha} := (r_{i,j} - \alpha_{i,j})_{\{i,j\}\in \pairs_\ell}$.
Using \eqref{eq: def Q}, one could write a similar Leibniz formula for $\partial_\by^{\bm{\alpha}} Q_\ell(x,\by)$.
This allows us to directly identify the dominant term of $\partial_\by^{\bm{\alpha}} Q_\ell(x,\by) \vert_{\by=\bm{1}}$ as $x\to x_0$, depending on $\bm{\alpha}$:
\begin{itemize}
    \item Assume that $\bm{\alpha}$ is not supported on any of the subsets $\pairs_\ell^{(i)} := \{ \{i,j\} : j\in[0,\ell], j\ne i\}$ for $i\in [0,\ell]$.
    Then, it is easily seen that $\partial_\by^{\bm{\alpha}} Q_\ell(x,\by) = 0$.
    \item Assume that $\bm{\alpha}$ is supported on $\pairs_\ell^{(i)}$ for a unique $i\in[0,\ell]$.
    Then,
    \[
        \partial_\by^{\bm{\alpha}} Q_\ell(x,\by)
        = (\One{i=0} + q\One{i\ne0}) \partial_\by^{\bm{\alpha}} \left[ \prod_{\{i,j\}\in \pairs_\ell^{(i)}} y_{i,j} \right]
        + q \partial_\by^{\bm{\alpha}} \left[ F_1{\left(x, \prod_{\{i,j\}\in \pairs_\ell^{(i)}} y_{i,j}\right)} \prod_{\{i,j\}\in \pairs_\ell^{(i)}} y_{i,j}^{2
        %1+\One{i=0}
        } \right] \,.
    \]
    Using \eqref{eq: claim singularity f1r}, the dominant term is obtained when differentiating solely the $F_1$ term in the product, and not the monomial in $\by$:
    \begin{equation*}
        \partial_\by^{\bm{\alpha}} Q_\ell(x,\by) \vert_{\by=\bm{1}}
        %\sim \partial_\by^{\bm{\alpha}} F_1{\left(x, \prod_{\{i,j\}\in \pairs_\ell^{(i)}} y_{i,j}\right)} \vert_{\by=\bm{1}}
        \sim q f_{1,\sum\bm{\alpha}}(x)
        \sim q \frac{(\sum\bm{\alpha})!}{\sqrt{2q}} \left( \sqrt{\frac{q}{2}} \frac{\log q}{q-1} \right)^{\sum\bm{\alpha}} \left( 1 - {x}/{x_0} \right)^{-\frac12-\frac12\sum\bm{\alpha}}
    \end{equation*}
    as $x\to x_0$.
    \item Assume that $\bm{\alpha}$ is supported on both $\pairs_\ell^{(i)}$ and $\pairs_\ell^{(j)}$ for some $\{i,j\}\in \pairs_\ell$, i.e., that the only non-zero term of $\bm{\alpha}$ is $\alpha_{i,j}$.
    Then, with the same analysis as the previous case,
    \begin{equation*}
        \partial_\by^{\bm{\alpha}} Q_\ell(x,\by) \vert_{\by=\bm{1}}
        \sim 2qf_{1,\alpha_{i,j}}(x)
        \sim 2q \frac{\alpha_{i,j}!}{\sqrt{2q}} \left( \sqrt{\frac{q}{2}} \frac{\log q}{q-1} \right)^{\alpha_{i,j}} \left( 1 - {x}/{x_0} \right)^{-\frac12-\frac12\alpha_{i,j}}
    \end{equation*}
    as $x\to x_0$.
\end{itemize}
Combining these cases and plugging them into \eqref{eq: Leibniz formula for product Qell Fell}, we get that
\begin{equation*}
    \partial_\by^\br {\left. \left[\big. Q_\ell(x,\by) F_\ell(x,\by) \right] \right\vert_{\by=\bm{1}} }
    - Q_\ell(x,\bm{1}) f_{\ell,\br}(x)
    \sim q\sum_{i=0}^\ell \sum_{\substack{0\prec \bm{\alpha}\preceq \br \\ \bm{\alpha} \text{ supported on } \pairs_\ell^{(i)}}} 
    %\left( \prod_{\{i,j\}\in \pairs_\ell^{(i)}} \binom{r_{i,j}}{\alpha_{i,j}} \right)
    \binom{\br}{\bm{\alpha}}
    f_{\ell,\br-\bm{\alpha}}(x) f_{1,\sum\bm{\alpha}}(x)
\end{equation*}
as $x\to x_0$.
Finally, using the induction hypothesis, this concludes the proof.
The fact that the constant $C_{\ell,\br}^{(1)}$ is positive is straightforward.
\end{proof}

\begin{proof}[Proof of Lemma~\ref{lem: asymptotics second term Rellr tilde}]
Recall that $R_{\ell}$ is defined by \eqref{eq: equadiff on Fell with everything} after taking out the terms containing the function $F_\ell$:
\begin{align}
    R_\ell(x,\by)
    = &\sum_{s=0}^\ell \sum_{\substack{\cJ\in \cP_{1,s}^\red \\ \abs{J_k}<\ell \text{ for all } k\in[1,s]}} \frac{1}{s!} 
    M_\cJ(\by)
    \prod_{1\le k\le s} F_{\abs{J_k}}\left(\big. x , \varphi_{J_k}(\by) \right) 
    \label{eq: def Rell first line}\\
    &+ q \sum_{s=0}^\ell \sum_{\substack{\cJ\in \cP_{2,s}^\red \\ \abs{J_k}<\ell \text{ for all } k\in[0,s]}} \frac{1}{s!} 
    M_\cJ(\by)
    F_{\abs{J_0}}\left(\big. x , \psi_{J_0}(\by) \right)
    \prod_{1\le k\le s} F_{\abs{J_k}}\left(\big. x , \varphi_{J_k}(\by) \right) \,. 
    \label{eq: def Rell second line}
\end{align}
Let us analyze the asymptotics, as $x\to x_0$, of each term appearing in $\partial_\by^\br R_\ell(x,\by) \vert_{\by=\bm{1}}$, starting with the first line \eqref{eq: def Rell first line} of the formula.
\begin{itemize}
    \item Let $s=0$ and $\cJ$ be the unique partition in $\cP_{1,0}^\red$.
    The corresponding term is $\partial_\by^\br M_\cJ(\by) \vert_{\by=\bm{1}} = \cO(1)$.
    \item Let $s=1$ and $\cJ\in\cP_{1,1}^\red$ be such that $\abs{J_1}<\ell$.
    Then $M_\cJ(\bm{1}) F_{\abs{J_1}}(x,\bm{1}) = \cO\left( (1-x/x_0)^{\frac12-\abs{J_1}} \right)$ by the induction hypothesis \eqref{eq: claim singularity fellr}.
    Furthermore, each differentiation step changes this exponent by $0$ or $-\frac12$, so that
    \[
        \partial_\by^\br \left. \left[ M_\cJ(\by) F_{\abs{J_1}}\left( x,\varphi_{J_1}(\by) \right) \right] \right\vert_{\by=\bm{1}} 
        = \cO\left( (1-x/x_0)^{\frac12-(\ell-1)-\frac12\sum\br} \right) \,.
    \]
    \item Let $s\ge2$ and $\cJ\in\cP_{1,s}^\red$.
    Then $M_\cJ(\bm{1}) \prod_{1\le k\le s} F_{\abs{J_k}}(x,\bm{1}) = \cO\left( (1-x/x_0)^{\frac{s}{2}-\sum_{1\le k\le s}\abs{J_k}} \right)$ by the induction hypothesis \eqref{eq: claim singularity fellr}.
    Furthermore, each differentiation step changes this exponent by $0$ or $-\frac12$, so that
    \[
        \partial_\by^\br \left. \left[ M_\cJ(\by) \prod_{1\le k\le s} F_{\abs{J_k}}\left( x,\varphi_{J_k}(\by) \right) \right] \right\vert_{\by=\bm{1}} 
        = \cO\left( (1-x/x_0)^{\frac{s}{2}-\ell-\frac12\sum\br} \right) \,.
    \]
\end{itemize}
Each case above yields a $o\left( (1-x/x_0)^{\frac12-\ell-\frac12\sum\br} \right)$.
Now we turn to the second line \eqref{eq: def Rell second line}, which is handled similarly.
\begin{itemize}
    \item Let $s=0$ and $\cJ\in\cP_{2,0}^\red$ be such that $\abs{J_0}<\ell$.
    Then $M_\cJ(\bm{1}) F_{\abs{J_0}}(x,\bm{1}) = \cO\left( (1-x/x_0)^{\frac12-\abs{J_0}} \right)$ by the induction hypothesis \eqref{eq: claim singularity fellr}.
    Furthermore, each differentiation step changes this exponent by $0$ or $-\frac12$, so that
    \[
        \partial_\by^\br \left. \left[ M_\cJ(\by) F_{\abs{J_0}}\left( x,\psi_{J_0}(\by) \right) \right] \right\vert_{\by=\bm{1}} 
        = \cO\left( (1-x/x_0)^{\frac12-(\ell-1)-\frac12\sum\br} \right) \,.
    \]
    \item Let $s=1$ and $\cJ\in\cP_{2,1}^\red$ be such that $\abs{J_0}<\ell$ and $\abs{J_1}<\ell$.
    Then, by the induction hypothesis \eqref{eq: claim singularity fellr}, $M_\cJ(\bm{1}) F_{\abs{J_0}}(x,\bm{1}) F_{\abs{J_1}}(x,\bm{1}) = \cO\left( (1-x/x_0)^{1-\abs{J_0}-\abs{J_1}} \right)$.
    Furthermore, each differentiation step changes this exponent by $0$ or $-\frac12$, so that
    \[
        \partial_\by^\br \left. \left[ M_\cJ(\by) F_{\abs{J_0}}\left( x,\psi_{J_0}(\by) \right) F_{\abs{J_1}}\left( x,\varphi_{J_1}(\by) \right) \right] \right\vert_{\by=\bm{1}} 
        = \cO\left( (1-x/x_0)^{1-(\ell+1)-\frac12\sum\br} \right) \,.
    \]
    Let us analyze this case with more care.
    Specifically, we wish to identify which $\cJ\in\cP_{2,1}^\red$ yield an optimal exponent, that is, $-\ell$, before differentiating.
    First, assume $\ell=2$.
    Then, $\abs{J_0}+\abs{J_1} = \ell+1 = 3$ would imply either $\abs{J_0}=2$ or $\abs{J_1}=2$, which is forbidden in \eqref{eq: def Rell second line}.
    So, if $\ell=2$, we can improve to
    \[
        \partial_\by^\br \left. \left[ M_\cJ(\by) F_{\abs{J_0}}\left( x,\psi_{J_0}(\by) \right) F_{\abs{J_1}}\left( x,\varphi_{J_1}(\by) \right) \right] \right\vert_{\by=\bm{1}} 
        = \cO\left( (1-x/x_0)^{1-\ell-\frac12\sum\br} \right) \,.
    \]
    Next, assume $\ell\ge3$.
    Then, there exist partitions $(J_0,J_1)$ of $[0,\ell]$ with $0\in J_0$, $\abs{J_0}<\ell$ and $\abs{J_1}<\ell$.
    For such a partition, we have $M_\cJ(\bm{1}) F_{\abs{J_0}}(x,\bm{1}) F_{\abs{J_1}}(x,\bm{1}) = \Theta\left( (1-x/x_0)^{-\ell} \right)$ by the induction hypothesis \eqref{eq: claim singularity fellr}.
    Furthermore, to reach the optimal exponent $-\ell-\frac12\sum\br$ after differentiating, one needs to differentiate one of the two $F$ terms at each differentiation step.
    Recall that the family $\psi_{J_0}(\by)$ is defined in terms of $(y_{i,j})_{\{i,j\}\in \pairs_\ell^{(J_0)}}$ and the family $\varphi_{J_1}(\by)$ is defined in terms of $(y_{i,j})_{\{i,j\}\in \pairs_\ell^{(J_1)}}$.
    Recall also that $J_0 \sqcup J_1 = [0,\ell]$.
    Therefore, when differentiating $r_{i,j}$ times with respect to $y_{i,j}$, there are the following cases:
    \begin{itemize}
        \item If $\{i,j\}\in \pairs_\ell^{(J_0)} \cap \pairs_\ell^{(J_1)}$, that is, if $i\in J_0$ and $j\in J_1$ (or vice versa), then we get to choose which $F$ term to differentiate.
        Differentiating $F_{\abs{J_0}}( x,\psi_{J_0}(\by) )$, resp.\ $F_{\abs{J_1}}( x,\varphi_{J_1}(\by) )$, with respect to $y_{i,j}$ yields
        \[
            \left( \prod_{h\in J_1 \setminus \{j\}} y_{i,h} \right) \left[ \partial_{i,\ell+1} F_{\abs{J_0}} \right]( x,\psi_{J_0}(\by) ) \,,
            \quad\text{resp.}\quad
            \left( \prod_{h\in J_0 \setminus \{i\}} y_{h,j} \right) \left[ \partial_{0,j} F_{\abs{J_1}} \right]( x,\psi_{J_1}(\by) ) \,.
        \]
        Here, the first derivative is with respect to the argument of $F_{\abs{J_0}}$ indexed by $\{i,\ell+1\}$, where we identify $[0,\abs{J_0}]$ with $J_0 \sqcup \{\ell+1\}$.
        The second derivative is with respect to the argument of $F_{\abs{J_1}}$ indexed by $\{0,j\}$, where we identify $[0,\abs{J_1}]$ with $\{0\} \sqcup J_1$.
        \item Otherwise, $\{i,j\} \in \pairs_\ell^{(J_0)}$ if and only if both $i$ and $j$ are in $J_0$.
        In that case, we have to differentiate the $F_\abs{J_0}$ term.
        Differentiating $F_{\abs{J_0}}( x,\psi_{J_0}(\by) )$ with respect to $y_{i,j}$ yields
        \[
            \left[ \partial_{i,j} F_{\abs{J_0}} \right]( x,\psi_{J_0}(\by) ) \,,
        \]
        where the first derivative is with respect to the argument of $F_{\abs{J_0}}$ indexed by $\{i,j\}$, where we identify $[0,\abs{J_0}]$ with $J_0 \sqcup \{\ell+1\}$.
        \item Likewise, if both $i$ and $j$ are in $J_1$, we have to differentiate the $F_\abs{J_1}$ term.
        The function $F_{\abs{J_1}}( x,\varphi_{J_1}(\by) )$ has derivative with respect to $y_{i,j}$ given by
        \[
            \left[ \partial_{i,j} F_{\abs{J_1}} \right]( x,\varphi_{J_1}(\by) ) \,.
        \]
        Here, the derivative is with respect to the argument of $F_{\abs{J_1}}$ indexed by $\{i,j\}$, where we identify $[0,\abs{J_1}]$ with $\{0\} \sqcup J_1$.
    \end{itemize}
    In the end we get, with the notation of the lemma,
    \begin{multline*}
        \partial_\by^\br {\left. \left[ M_\cJ(\by) F_{\abs{J_0}}\left( x,\psi_{J_0}(\by) \right) F_{\abs{J_1}}\left( x,\varphi_{J_1}(\by) \right) \right] \right\vert_{\by=\bm{1}} } 
        \\\sim
        \sum_{\substack{ \bm{0}\preceq \bm{\alpha}\preceq \br \\ \alpha_{i,j}=r_{i,j} \text{ for all } i,j\in J_0 \\ \alpha_{i,j}=0 \text{ for all } i,j\in J_1}}
        %\left( \prod_{\{i,j\}\in \pairs_\ell} \binom{r_{i,j}}{\alpha_{i,j}} \right)
        \binom{\br}{\bm{\alpha}}
        f_{\abs{J_0}, \bm\alpha^{J_0,\ell+1}}(x)
        f_{\abs{J_1}, (\br-\bm\alpha)^{0,J_1}}(x) 
    \end{multline*}
    \item Let $s\ge2$ and $\cJ\in\cP_{2,s}^\red$.
    Then $M_\cJ(\bm{1}) F_{\abs{J_0}}(x,\bm{1}) \prod_{1\le k\le s} F_{\abs{J_k}}(x,\bm{1}) = \cO\left( (1-x/x_0)^{\frac{s+1}{2}-\sum_{1\le k\le s}\abs{J_k}} \right)$ by the induction hypothesis \eqref{eq: claim singularity fellr}.
    Furthermore, each differentiation step changes this exponent by $0$ or $-\frac12$, so that
    \[
        \partial_\by^\br \left. \left[ M_\cJ(\by) F_{\abs{J_0}}\left( x,\psi_{J_0}(\by) \right) \prod_{1\le k\le s} F_{\abs{J_k}}\left( x,\varphi_{J_k}(\by) \right) \right] \right\vert_{\by=\bm{1}} 
        = \cO\left( (1-x/x_0)^{\frac{s+1}{2}-(\ell+1)-\frac12\sum\br} \right) \,.
    \]
\end{itemize}
If $\ell=2$, then each case above yields a $\cO\left( (1-x/x_0)^{\frac12-\ell-\frac12\sum\br} \right)$.
If $\ell\ge3$, then the penultimate case yields a $\Theta\left( (1-x/x_0)^{-\ell-\frac12\sum\br} \right)$, while the other cases are negligible.
Using the induction hypothesis \eqref{eq: claim singularity fellr}, the conclusion follows.
\end{proof}

\subsection{Identifying the CRT}
\label{sec: Identifying the CRT}

To complete the proof of Proposition~\ref{prop: joint moments cv to CRT}, we need to investigate the dependence of the constants $C_{\ell, \br}$ on $q$.
For this, we bring back the dependence on $q$ in notation, and write $C_{\ell,\br}(q)$ for the constant in \eqref{eq: claim singularity fellr}.
Let $\gamma(q) := \frac{\sqrt q \log q}{q-1}$ for $q\in(0,1)$, and $\gamma(q) := 1$ for $q=1$.
We wish to prove that, for $q\in(0,1]$ and any $\ell\in\bN$ and any family $\br \in \bN_0^{\pairs_\ell}$,
\begin{equation}\label{eq: claim dependency Cellr on q}
    C_{\ell,\br}(q) = \frac{1}{\sqrt q} \gamma(q)^{\sum\br} C_{\ell,\br}(1) \,.
\end{equation}
As in the previous section, we proceed by induction, using the lexicographic order on $(\ell,\br)$ and the well-founded partial order $\prec$ on $\bN_0^{\pairs_\ell}$.

\paragraph{Initialization.}
For $\ell=1$ and $\br = r \ge0$, \eqref{eq: claim dependency Cellr on q} has been proved in Proposition~\ref{prop: Rayleigh limit of typical depth when q fixed}, see \eqref{eq: claim singularity f1r}.
For $\ell\ge2$ and $\br = \bm0$, it has been proved in the initialization step in the proof of \eqref{eq: claim singularity fellr}.

\paragraph{Induction step.}
Now, fix $(\ell,\br)$ with $\ell\ge2$ and $\sum\br > 0$.
Assume that $\eqref{eq: claim dependency Cellr on q}$ holds for all the indices $(m,\bm\alpha)$ with either $m<\ell$, or $m=\ell$ and $\bm{\alpha}\prec\br$.
Let us write $\cC_{\ell, \br}(q)$ for the family of constants $C_{m,\bm\alpha}(q)$ with either $m<\ell$, or $m=\ell$ and $\bm{\alpha}\prec\br$.

By \eqref{eq: recursive singularity for fellr}, we can write
\begin{equation}\label{eq: recursive formula Cellr}
    C_{\ell,\br}(q) = \frac{\sqrt{2q}}{\ell-2+\sum\br} \xi\big( \cC_{\ell, \br}(q) \big)
\end{equation}
for some degree $2$ polynomial function $\xi$ that is independent of $q$.
Using Lemmas~\ref{lem: asymptotics first term Rellr tilde} and~\ref{lem: asymptotics second term Rellr tilde} and the induction hypothesis \eqref{eq: claim dependency Cellr on q}, we see that
\begin{equation*}
    \xi\big( \cC_{\ell, \br}(q) \big) = \frac1q \gamma(q)^{\sum\br} \xi\big( \cC_{\ell, \br}(1) \big) \,.
\end{equation*}
Finally, we get that
\begin{equation*}
    C_{\ell,\br}(q) 
    = \frac{1}{\sqrt q} \gamma(q)^{\sum\br} \frac{\sqrt{2}}{\ell-2+\sum\br} \xi\big( \cC_{\ell, \br}(1) \big)
    = \frac{1}{\sqrt q} \gamma(q)^{\sum\br} C_{\ell,\br}(1) \,,
\end{equation*}
where the second identity was obtained by applying \eqref{eq: recursive formula Cellr} with $q=1$.
This concludes the proof by induction of \eqref{eq: claim dependency Cellr on q}.

\medskip

Finally, let us conclude the proof of Proposition~\ref{prop: joint moments cv to CRT}.
Recall that for fixed $q\in(0,1]$, the functions $f_{\ell,\br}(x,q)$ are amenable to singularity analysis, with a single dominant singularity at $x_0(q)$.
Therefore, by \eqref{eq: claim singularity fellr} and \eqref{eq: claim dependency Cellr on q},
\begin{equation*}
    [x^n] f_{\ell,\br}(x,q) \sim \frac{1}{\sqrt q} \gamma(q)^{\sum\br} \frac{C_{\ell,\br}(1)}{\Gamma(\ell-\frac12+\frac12\sum\br)} n^{\ell-\frac32+\frac12\sum\br} x_0(q)^{-n}
\end{equation*}
as $n\to\infty$.
Thus, by \eqref{eq: asymptotic coefficient A for fixed q} and \eqref{eq: joint moment of distances from extracted coefficients},
\begin{equation*}
    \expec{ \prod_{0\le i< j\le \ell} \left(\big. \dist_n(W_{n,i},W_{n,j}) \right)_{r_{i,j}} }
    = \frac{1}{n^\ell} \frac{[x^n]f_{\ell,r}(x,q)}{[x^n]A(x,q)}
    \sim \gamma(q)^{\sum \br} \frac{ C_{\ell,\br}(1) \sqrt{2\pi} }{ \Gamma\left(\ell -\frac12 + \frac12\sum \br\right) } n^{\frac12\sum \br}
\end{equation*}
as $n\to\infty$.
Hence, the same holds when replacing the falling factorials by the $r_{i,j}$-th powers.
Moreover, we know that Proposition~\ref{prop: joint moments cv to CRT} holds for $q=1$ by \cite{Aldous_1993}, and we deduce that
\begin{equation*}
    \expec{ \prod_{0\le i< j\le \ell} \left(\big. \dist(W_{i},W_{j}) \right)^{r_{i,j}} }
    = \frac{ C_{\ell,\br}(1) \sqrt{2\pi} }{ \Gamma\left(\ell -\frac12 + \frac12\sum \br\right) } \,.
\end{equation*}
The proof of Proposition~\ref{prop: joint moments cv to CRT} easily follows.

\section{The dendron scaling limit in the critical window}
\label{sec:dendron limit}

This section is devoted to the proof of Theorem \ref{th: dendron limit if a/n}, which states the convergence of the descent-biased trees $\cT_n := \cT_n^{(q_n)}$ when $q_n = a/n$ for some fixed $a \in (0,+\infty)$. 
The structure of the proof is as follows. 
In Section \ref{ssec:reduced tree}, we define and investigate a ``reduced'' tree $\cT_n^{k,\red}$, which is a finite decorated tree obtained from $\cT_n$ by considering the relative positions of the first $k$ labels. 
Our claim is that, for large enough $k$, as $n \rightarrow \infty$, the structure of this reduced tree is close to the structure of the whole tree $\cT_n$. 
Then, in Section \ref{ssec: limit dendron convergence}, we use properties of the reduced tree along with the so-called Foata--Fuchs bijection (see Definition \ref{def:finite foata fuchs}), and a coupling between $\cT_n$ and a limit dendron $\cD_a$, to complete the proof of Theorem \ref{th: dendron limit if a/n}.
Finally, in \Cref{sec: interpolation}, we rely on known results on $p$-trees to prove \Cref{th: our dendron cv to CRT}.

\subsection{The $k$-reduced tree}
\label{ssec:reduced tree}

We start by properly defining the $k$-\textit{reduced tree} $T^{k,\red}$ of a finite tree $T \in \bT_n$ for $n \geq k \geq 1$. 
See \Cref{fig: basic notation for reduced tree and stuff} for an example.
Then, we conduct a precise asymptotic study of the $k$-reduced tree $\cT_n^{k,\red}$ associated with our descent-biased trees $\cT_n^{(a/n)}$.

\subsubsection{Definition and main result}

First, consider the subgraph of $T$ induced by the vertices with labels $1$ to $k$, which is denoted by $T\cap[k]$.
Inheriting from the tree structure of $T$, this subgraph is a forest of a certain number $\ell^{k}(T)$ of subtrees that are not connected to each other, which we call $k$-components, and whose sizes sum to $k$.

\paragraph{The reduced tree $T^{k,\red}$.}

The reduced tree $T^{k,\red}$ is, roughly speaking, the genealogy tree of these $\ell^{k}(T)$ components. 
It is a decorated non-plane unlabeled tree with either $\ell^k(T)$ or $\ell^k(T)+1$ vertices, constructed as follows. 
\begin{itemize}
    \item First, if the root label of $T$ is $\le k$, then the root of $T^{k,\red}$ is decorated by the $k$-component of the root in $T\cap[k]$; otherwise\footnote{
    As we will see in Lemma~\ref{lem: asymptotic probability of V1}, this case may be ignored with ``high'' probability.
    }, it is decorated by $\emptyset$.
    \item Then, the other vertices of $T^{k,\red}$ are decorated by the other $k$-components of $T \cap [k]$. 
    This allows us to use the notation $(t_u)_{u \in T^{k,\red}}$ for the family of $k$-components, where $t_u$ is the $k$-component indexed by the vertex $u$ of the reduced tree.
    Note that depending on the root label, $t_{\rho(T^{k,\red})}$ is either a $k$-component or the empty tree.
    \item The vertices of $T^{k,\red}$ are connected according to their natural tree structure. 
    That is, for $u,v \in T^{k,\red}$, $(u,v)$ is a parent-child pair in $T^{k,\red}$ if and only if there exist labels $(i,i')$ in $T$ such that $i$ is in $t_u$, $i'$ is in $t_v$, $i$ is an ancestor of $i'$, and the unique simple path from $i$ to $i'$ (excluded) in $T$ contains only labels $>k$.
\end{itemize}

The $k$-reduced tree $T^{k,\red}$ thus contains either $\ell^{k}(T)$ or $\ell^{k}(T)+1$ vertices (depending on whether the root of $T$ has label $\leq k$ or not), and always exactly $\ell^{k}(T)$ non-empty vertices.

\paragraph{The subtrees $(\tau_i)_{i \in [k]}$.}

Now, deleting the edges between each label $\le k$ and its parent in $T$ (which may also have label $\le k$), we obtain a family of disconnected subtrees $\tau_1, \dots, \tau_k$, rooted at labels $1, \dots, k$, and potentially a top-tree $\tau_\emptyset$ if the root label of $T$ is $> k$.
For all $i \in [k]$, let $\nu_i$ be the size of $\tau_i$.
Each subtree $\tau_i$ initially comes with an \enquote{improper labeling}, that is, with the same labels as in $T$.
We relabel it with a \enquote{proper labeling}, that is, we label its vertices with the relative order on $[\nu_i]$ induced by the improper labeling.
In particular, the root of $\tau_i$ always has improper label $i$ and proper label $1$.

\begin{figure}
    \centering
    \includegraphics[width=0.8\linewidth]{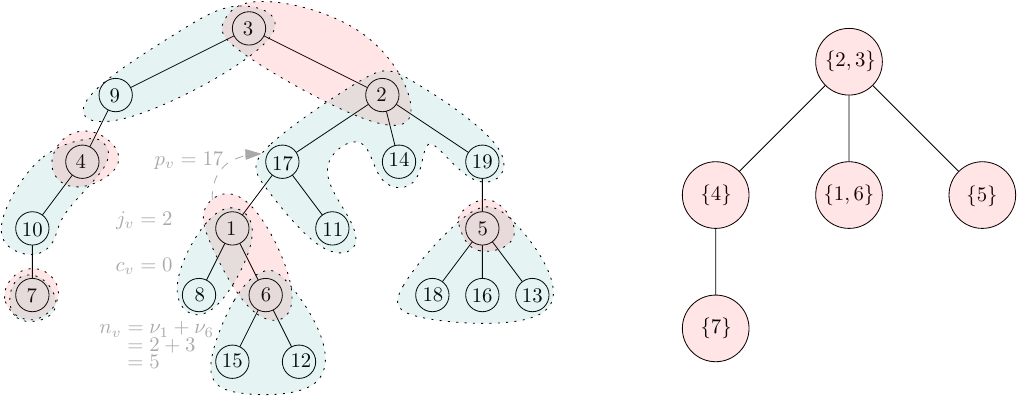}
    \caption{Left: a labeled tree $T \in \bT_{19}$.
    Right: its $k$-reduced tree $T^{k,\red}$ for $k=7$.
    On $T$, the subtrees $t_u$ for $u\in T^{k,\red}$ are highlighted in red, and the subtrees $\tau_i$ for $i\in[7]$ are highlighted in green (with their improper labeling).
    The \enquote{colored components}, obtained by merging the red and green ones, yield the subtrees $T_u$ for $u\in T^{k,\red}$.
    Some of the quantities associated with the reduced vertex $v\in T^\red$ corresponding to the labels $F_v=\{1,6\}$ are indicated in gray.
    }
    \label{fig: basic notation for reduced tree and stuff}
\end{figure}

\paragraph{Notation.}

Below is a list of notation concerning our objects of interest. 
It is partly illustrated in \Cref{fig: basic notation for reduced tree and stuff}.

\begin{itemize}
    \item $T^{k,\red}$ is the non-plane, rooted, unlabeled, decorated $k$-reduced tree.
    \item $\ell^{k}{\left( T \right)}$ is the number of non-empty vertices in $T^{k,\red}$, that is, the number of connected components in the subgraph of $T$ induced by the vertices with labels in $[k]$.
    \item For each vertex $u \in T^{k,\red}$,
    \begin{itemize}
        \item  $t_u$ is the $k$-component corresponding to $u$,
        \item $F_u \subseteq [k]$ is the family of labels (in $T$) of vertices in $t_u$,
        \item $j_u$ is the cardinality of $F_u$ (that is, the number of vertices in $t_u$),
        \item $c_u$ is the number of children of $u$ in $T^{k,\red}$.
    \end{itemize} 
    Note that the vertices of the tree $t_u$ initially come with an \enquote{improper labeling} $F_u$.
    In what follows, unless stated otherwise, we shall see the vertices of $t_u$ as \enquote{properly labeled}, using the relative order on $[j_u]$ induced by $F_u$.
    
    The unordered collections of the above objects are respectively denoted by $\mathbf{t}(T) := ( t_u )_{u\in T^{k,\red}}$, $\mathbf{F}(T) := ( F_u )_{u\in T^{k,\red}}$, $\bj(T) := ( j_u )_{u\in T^{k,\red}}$ and $\bc(T) := ( c_u )_{u\in T^{k,\red}}$.
    We allow $t_u = F_u = \emptyset$ and $j_u=0$ for the root vertex $u=\rho(T^{k,\red})$ of $T^{k,\red}$ if the root label of $T$ is $>k$.
    \item For each $i\in[k]$, the tree grafted at the vertex labeled $i$ in $T$ with labels $>k$ is denoted by $\tau_i$, and its size is denoted by $\nu_i$.
    Its improper labeling inherited from $T$ is $\{i\} \sqcup \Phi_i \subseteq [n]$ for some $\Phi_i \subseteq [n]\setminus[k]$, and its proper labeling is the one induced on $[\nu_i]$ (in particular, the root has proper label $1$).
    Their ordered collections are denoted by $\bm{\tau}(T) := ( \tau_i )_{1\le i\le k}$, $\mathbf{\Phi}(T) := ( \Phi_i )_{1\le i\le k}$, and $\bm{\nu}(T) := ( \nu_i )_{1\le i\le k}$.
    If the root label of $T$ is $>k$, we also add to the collection a tree $\tau_\emptyset$, improperly labeled by a set $\Phi_\emptyset \subseteq [n]$.
    \item For each $u\in T^{k,\red}$, let $[T]^{k}_u := \bigsqcup_{i\in F_u} \tau_i$ be the subtree of $T$ \enquote{rooted} at the root of $t_u$, and containing the subtrees $\tau_i$ for each label $i$ corresponding to a vertex of $t_u$.
    Write $[n]^{k}_u := \sum_{i\in F_u} \nu_i$ for its size.
    Equivalently, $[T]_u^k$ is the component of $T$ containing $t_u$ after deleting each edge between a parent with label $>k$ and a child with label $\le k$.
    Their unordered collections are $\mathbf{T}^{k}(T) := {\left( [T]^{k}_u \right)}_{u\in T^{k,\red}}$ and $\bn^{k}(T) := {\left( [n]^{k}_u \right)}_{u\in T^{k,\red}}$.
    For ease of notation, and when there is no ambiguity, we may sometimes write $T_u$ and $n_u$ instead of $[T]_u^k$ and $[n]_u^k$.
  \item  In $T$, for any non-root $v \in T^{k,\red}$, the root of the $k$-component $t_v$ is the child of some vertex with label $p_v > k$.
    Write $\mathbf{p}(T) = (p_v)_{v\in T^{k,\red} \text{ non root}}$ for their unordered collection. 
    Note that $p_v$ is the improper label of a vertex in $[T]_u^k$, where $u$ is the parent of $v$ in $T^{k,\red}$.
\end{itemize}

The dependence on $k$ is often skipped for ease of notation.
Note that quantities indexed by the reduced tree use roman letters, whereas quantities indexed by $[k]$ use greek letters.
Also, unless specified otherwise, the subtrees $t_u$ and $\tau_i$ will be regarded as {properly} labeled (that is, with labels in $[j_u]$ and $[\nu_i]$, repectively) by default.
We claim that for any given $n\ge k\ge1$, the mapping 
\[
    T \in \bT_n \mapsto \left( T^{k,\red}, \mathbf{t}(T), \mathbf{F}(T), \bm{\tau}(T), \mathbf{\Phi}(T), \mathbf{p}(T) \right)
\]
is one-to-one.
Indeed, one can reconstruct $T$ from $\left( T^{k,\red}, \mathbf{t}(T), \mathbf{F}(T), \bm{\tau}(T), \mathbf{\Phi}(T), \mathbf{p}(T) \right)$ as follows:
\begin{itemize}
\item for each $u\in T^{k,\red}$, place a tree $t_u$ whose proper labeling is consistently replaced by the set of labels $F_u$;
\item for each $i\in[k]$, graft to the vertex labeled $i$ the tree $\tau_i$ relabeled by $\Phi_i$;
\item potentially place the tree $\tau_\emptyset$ with labeling $\Phi_\emptyset$;
\item  finally, for each pair $(u,v)$ of parent-child in $T^{k,\red}$, attach the root of $t_v$ to the vertex labeled $p_v$ (which is in $[T]_u^k := \bigsqcup_{i\in F_u} \tau_i$).
\end{itemize}
One can check that the resulting tree is indeed the original tree $T$.

See \Cref{fig: simulation reduced components} for a simulation of a descent-biased tree $\cT_n := \cT_n^{(a/n)}$ together with its $k$-components and reduced tree.
In our study, we shall be mostly interested in the joint behavior of the statistics $\bj(\cT_n), \bc(\cT_n), \bn(\cT_n)$.
We will show in Proposition~\ref{prop: joint convergence of tilde j tilde c tilde n towards PoissonDirichlet} that these are asymptotically proportional, and distributed like a Poisson--Dirichlet random variable.
Both the statement of this result, and its proof, will require introducing an ordering for the unordered families $\bj(\cT_n), \bc(\cT_n), \bn(\cT_n)$.
We actually introduce two\footnote{
    Actually, we will also introduce a third, \emph{random}, reordering in the next section.
} slightly different orderings (which turn out to be asymptotically the same), for distinct purposes.

On the one hand, we use the symbol $\searrow$ to denote a reordering of our families dictated by $\bn(T)$.
Specifically, let $\bn^\searrow(T)$ denote the non-increasing reordering of $\bn(T)$, with arbitrary choice in case of equality. 
We also use $\bj^\searrow(T)$, $\bc^\searrow(T)$ to denote the respective reorderings of $\bj(T), \bc(T)$ according to the ordering of $\bn^\searrow(T)$. 
In other words, letting $(u_1,\ldots,u_\ell)$ be the vertices of $T^{k,\red}$ enumerated by non-increasing order of $(n_u)_{u \in T^{k,\red}}$, we let $\big( \bj^\searrow(T), \bc^\searrow(T), \bn^\searrow(T) \big) = \big( (j_{u_1},\ldots,j_{u_\ell}),(c_{u_1},\ldots,c_{u_\ell}),(n_{u_1},\ldots,n_{u_\ell}) \big)$.
This reordering will be useful in the application of Proposition~\ref{prop: joint convergence of tilde j tilde c tilde n towards PoissonDirichlet} in later sections.

On the other hand, we use the symbol $'$ to denote a reordering of our families dictated by $\bj(T)$.
That is, letting $(u_1',\ldots,u_\ell')$ be the vertices of $T^{k,\red}$ enumerated by non-increasing order of $(j_u)_{u \in T^{k,\red}}$ (with arbitrary choice in case of equality), we let $\big( \bj'(T), \bc'(T), \bn'(T) \big) = \big( (j_{u_1'},\ldots,j_{u_\ell'}),(c_{u_1'},\ldots,c_{u_\ell'}),(n_{u_1'},\ldots,n_{u_\ell'}) \big)$.
This reordering will be useful for the proof of Proposition~\ref{prop: joint convergence of tilde j tilde c tilde n towards PoissonDirichlet}.

Our main result states that, for a $q_n$-descent biased tree $\cT_n^{(q_n)}$ with $q_n := a/n$, as $n \rightarrow \infty$, the sequences of renormalized sizes of the $k$-components, degrees in the reduced tree and sizes of the components on top of the reduced vertices, all jointly converge towards the same Poisson--Dirichlet random variable, in the following sense:

\begin{figure}
    \centering
    \begin{minipage}{.49\linewidth}
    \includegraphics[width=.9\linewidth]{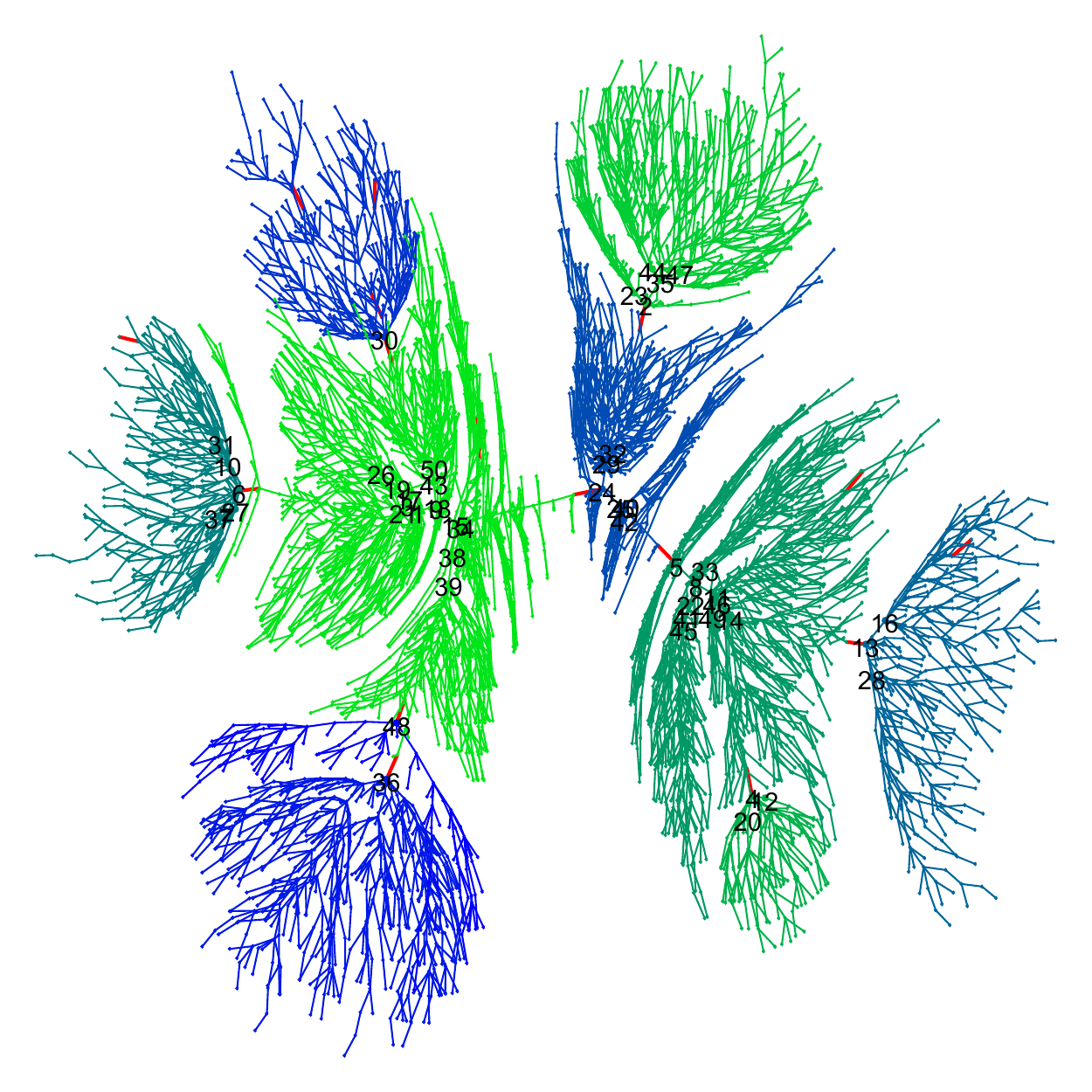}
    \end{minipage}
    \qquad
    \begin{minipage}{.4\linewidth}
    \includegraphics[width=\linewidth]{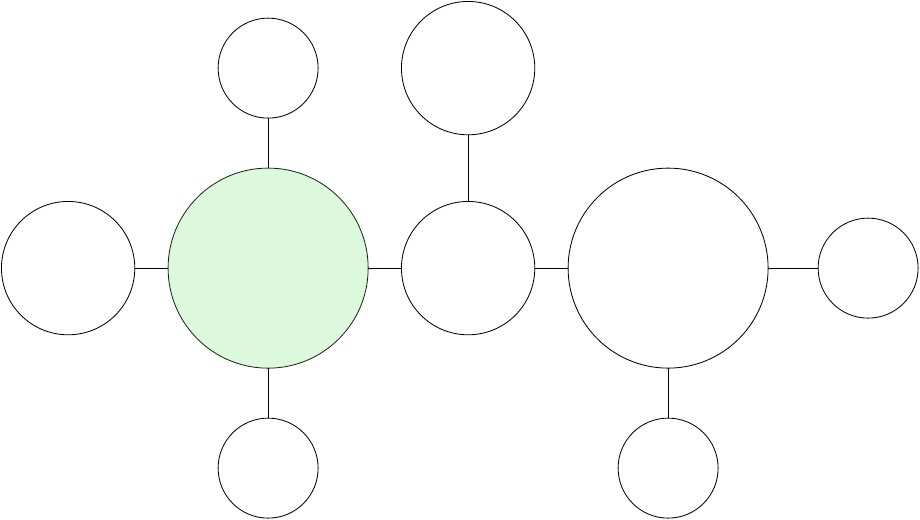}
    \end{minipage}
    \caption{Left:
    Simulation of a descent-biased tree $\cT_n^{(q)}$ with size $n=6000$ and bias parameter $q=3/n$.
    Its descent edges are displayed in red.
    The subtrees $[\cT_n]_{u}^k$ for $k=50$ and $u\in \cT_n^{k,\red}$ are highlighted, with colors ranging from light green to dark blue.
    The labels from $1$ to $50$ are also shown, with $1$ being the root.
    Right:
    The reduced subtree $\cT_n^{k,\red}$, rooted at the light green vertex.
    The vertex sizes have been chosen proportionally to the quantities $j_u$.
    }
    \label{fig: simulation reduced components}
\end{figure}

\begin{proposition}\label{prop: joint convergence of tilde j tilde c tilde n towards PoissonDirichlet}
Let $q_n=a/n$ and $\cT_n = \cT_n^{(q_n)}$.
The $\ell$-tuples $\bj^\searrow(\cT_n)$, $\bc^\searrow(\cT_n)$, $\bn^\searrow(\cT_n)$ asymptotically jointly approximate a single Poisson--Dirichlet random variable.
More precisely, let $\Delta$ denote the infinite-dimensional simplex endowed with the product topology, and let $d_{\Delta^3}$ be any metric for the weak topology on $\cM_1(\Delta^3)$. 
Embed 
%$\{(x_1,\ldots,x_\ell) \in \bR_+^\ell, x_1 \geq \ldots \geq x_\ell \geq 0, \sum_{i=1}^\ell x_i = 1\}$
$\{(x_1,\ldots,x_\ell) \in \bR_+^\ell : \sum_{i=1}^\ell x_i = 1\}$
naturally in $\Delta$ by completing an element with infinitely many $0$'s. Then,
    \begin{equation*}
        \lim_{k\to\infty} \limsup_{n\to\infty} 
        d_{\Delta^3}{\left( \Law{ \frac1k\bj^\searrow(\cT_n) , \frac{1}{a\log k} \bc^\searrow(\cT_n) , \frac1n \bn^\searrow(\cT_n) } \;,\; \Law{ X^{(a)} , X^{(a)} , X^{(a)} } \right)} 
        = 0
    \end{equation*}
    where $X^{(a)} \sim \PoissonDirichlet{0}{a}$.
\end{proposition}

It is clearly equivalent to prove Proposition~\ref{prop: joint convergence of tilde j tilde c tilde n towards PoissonDirichlet} for either $(\bj^\searrow(\cT_n)$, $\bc^\searrow(\cT_n)$, $\bn^\searrow(\cT_n))$ or $(\bj'(\cT_n)$, $\bc'(\cT_n)$, $\bn'(\cT_n))$.
Therefore, we will prove it for the latter but apply it with the former.

\subsubsection{Combinatorial formulas for joint statistics in descent-biased trees}

The first step is to write explicit formulas for some statistics of the descent-biased tree; these formulas will be extensively used later to prove Proposition \ref{prop: joint convergence of tilde j tilde c tilde n towards PoissonDirichlet}.

Let $V_\rho(k)$ denote the event that the root of $\cT_n^{k,\red}$ is non-empty, that is, the root label of $\cT_n$ is $\le k$. 
In our computations, we will only investigate probabilities of events that are contained in $V_\rho(k)$. The reason is the following lemma.

\begin{lemma}\label{lem: asymptotic probability of V1}
For any fixed $k\ge1$,
\begin{align*}
    \prob{V_\rho(k)} \cv{n\to\infty} \frac{k}{a+k} \,.
\end{align*}
\end{lemma}
\noindent
In particular, we have $\lim_{n\to\infty} \prob{V_\rho(k)} \cv{k\to\infty} 1$.

In order to prove it, it will be useful to add a layer of randomness to $T^{k,\red}$, providing an arbitrary ordering on the vertices of the reduced tree.
The construction below is defined whenever the event $V_\rho(k)$ holds.
First, choose a uniformly random permutation $\sigma \in \fS_\ell$ of $\ell$ elements, where $\ell := \ell^k(T)$.
Then, use it to enumerate $\mathbf{F}(T) = (F_u)_{u\in T^{k,\red}}$, e.g., by ranking the parts according to their smallest element and choosing the $i$-th part to be the one with rank $\sigma_i$.
This yields a labeling of the non-plane tree $T^{k,\red}$ by $[\ell]$; we write $\widetilde T^{k,\red}$ for that labeled tree.
Through this labeling, all unordered families indexed by $T^{k,\red}$ are now indexed by $[\ell]$, and these newly ordered objects are written with a tilde.
That is, the unordered families $\mathbf{t}(T)$, $\mathbf{F}(T)$, $\bj(T)$ and $\bc(T)$, indexed by the vertices of $T^{k,\red}$, give rise to randomly ordered families $\widetilde{\mathbf{t}}(T) = (t_\upsilon)_{1\le \upsilon\le \ell}$, $\widetilde{\mathbf{F}}(T) = (F_\upsilon)_{1\le \upsilon\le \ell}$, $\widetilde{\bj}(T) = (j_\upsilon)_{1\le \upsilon\le \ell}$ and $\widetilde{\bc}(T) = (c_\upsilon)_{1\le \upsilon\le \ell}$.

All probabilities in the lemma below are with respect to the randomness of both the tree $\cT_n := \cT_n^{(a/n)}$ and, conditionally given $\cT_n$, the permutation $\sigma \sim \Unif{\fS_{\ell^k(\cT_n)}}$.

\begin{lemma}
    \label{lem:combinatorial formulas}
    Fix a non-plane labeled tree $\widetilde t \in \bT_\ell$ with size $\ell$ and out-degrees $\wbc = (c_\upsilon)_{1\le \upsilon\le \ell}$.
    Moreover, fix $\wbj = (j_\upsilon)_{1\le \upsilon\le \ell}$, a list of positive integers summing to $k$, and $\wbn = (n_\upsilon)_{1\le \upsilon\le \ell}$, a list of positive integers summing to $n$, such that $n_\upsilon\ge j_\upsilon$ for all $\upsilon\in[\ell]$.
\begin{enumerate}[label=(\roman*)]
\item \begin{multline*}
    \prob{V_\rho(k), \widetilde \cT_n^\red = \widetilde t , \wbj(\cT_n) = \wbj, \wbn(\cT_n) = \wbn}
    \\= \frac{k!(n-k)!}{\ell!Z_{n,q}} q^{\ell-1}
    \prod_{\upsilon\in [\ell]} \left( 
    (n_\upsilon - j_\upsilon) ^{c_\upsilon}
    \cdot [x^{j_\upsilon}] A(x,q)
    \cdot [x^{n_\upsilon-j_\upsilon}] e^{j_\upsilon A(x,q)}
    \right) \,.
    \end{multline*}

\item  Conditionally given $\wbc(\cT_n)$, the reduced tree $\widetilde \cT_n^{k,\red}$ is a uniform labeled tree with out-degrees $\wbc(\cT_n)$.

\item \begin{multline*}
    \prob{V_\rho(k), \wbc(\cT_n) = \wbc, \wbj(\cT_n) = \wbj, \wbn(\cT_n) = \wbn}
    \\= \frac{k!(n-k)!}{\ell Z_{n,q}} q^{\ell-1}
    \prod_{\upsilon\in [\ell]} \left( 
    \frac{(n_\upsilon - j_\upsilon)^{c_\upsilon}}{c_\upsilon!}
    \cdot [x^{j_\upsilon}] A(x,q)
    \cdot [x^{n_\upsilon-j_\upsilon}] e^{j_\upsilon A(x,q)}
    \right) \,.
    \end{multline*}

\item \begin{align*}
    \prob{V_\rho(k), \wbj(\cT_n) = \wbj, \wbn(\cT_n) = \wbn}
    = \frac{k!(n-k)!}{\ell! Z_{n,q}} \left(q(n-k)\right)^{\ell-1}
    \prod_{\upsilon\in [\ell]} \left( 
    [x^{j_\upsilon}] A(x,q)
    \cdot [x^{n_\upsilon}] x^{j_\upsilon} e^{j_\upsilon A(x,q)}
    \right) \,.
    \end{align*}

\item \begin{align}\label{item: proba reduced tree has small root and tildej}
    \prob{V_\rho(k), \wbj(\cT_n) = \wbj}
    = \frac{k!(n-k)!}{\ell! Z_{n,q}} \left(q(n-k)\right)^{\ell-1}
    [x^{n-k}] e^{k A(x,q)}
    \prod_{\upsilon\in [\ell]}
    [x^{j_\upsilon}] A(x,q) \,.
    \end{align}
\end{enumerate}
\end{lemma}

Let us immediately see how we get \Cref{lem: asymptotic probability of V1} from that.

\begin{proof}[Proof of \Cref{lem: asymptotic probability of V1}]
Fix $\ell\ge1$ and sum equation~\eqref{item: proba reduced tree has small root and tildej} of Lemma~\ref{lem:combinatorial formulas} over all $\ell$-tuples $\wbj$ of positive integers summing to $k$:
\begin{align*}
    \prob{V_\rho(k), \ell(\cT_n)=\ell}
    = \frac{k!(n-k)!}{\ell! Z_{n,q}} \left(q(n-k)\right)^{\ell-1}
    [x^{n-k}] e^{k A(x,q)}
    \sum_{\substack{j_1, \dots, j_\ell \ge1\\ j_1+\dots+j_\ell =k}} \prod_{\upsilon\in [\ell]}
    [x^{j_\upsilon}] A(x,q).
\end{align*}
Thus, using the fact that $[x^0] A(x,q) = 0$,
\begin{align}\label{eq: reduced tree distribution ell}
    \prob{V_\rho(k), \ell(\cT_n)=\ell}
    = \frac{k!(n-k)!}{\ell! Z_{n,q}} \left(q(n-k)\right)^{\ell-1}
    [x^{n-k}] e^{k A(x,q)}
    [x^{k}] A(x,q)^\ell \,.
\end{align}
Finally, by summing \eqref{eq: reduced tree distribution ell} over $\ell\ge1$, we obtain
\begin{align}\label{eq: reduced tree distribution V1}
    \prob{V_\rho(k)}
    = \frac{k!(n-k)!}{Z_{n,q}} \left(q(n-k)\right)^{-1}
    [x^{n-k}] e^{k A(x,q)}
    [x^{k}] e^{q(n-k)A(x,q)} \,.
\end{align}
    Let us take $n\to\infty$ in \eqref{eq: reduced tree distribution V1} (recall that $k\ge1$ is fixed, and that $q=a/n$).
    By Lemma \ref{lem: asymptotic coefficient partition function in transition regime}, we know that 
    \begin{equation*}
        Z_{n,q} = n! [x^n] A(x,q) \sim \frac{n! n^{a-1}}{e^a \Gamma(a+1)} \,,
    \end{equation*}
    by Lemma \ref{lem: asymptotic coefficient exponentials in transition regime} that 
    \begin{equation*}
        [x^{n-k}] e^{k A(x,q)} \sim \frac{k n^{a-1+k}}{e^a\Gamma(a+k+1)} \,,
    \end{equation*}
    and by the standard enumeration of recursive trees that 
    \begin{equation*}
        [x^{k}] e^{q(n-k)A(x,q)} 
        \cv{} [x^k] e^{a A(x,0)}
        = [x^k] (1-x)^{-a}
        = \frac{\Gamma(a+k)}{\Gamma(a)k!} \,.
    \end{equation*}
    Plugging all of these into \eqref{eq: reduced tree distribution V1}, we find that $\prob{V_\rho(k)} \cv{} \frac{k}{a+k}$ as $n \rightarrow \infty$, as desired.
\end{proof}

We now prove Lemma \ref{lem:combinatorial formulas}.

\begin{proof}[Proof of Lemma \ref{lem:combinatorial formulas}]
%Fix a non-plane, labeled tree $\widetilde t$ with size $\ell$ and out-degrees $\wbc = (c_u)_{1\le u\le \ell}$.
%Also fix $\wbj = (j_u)_{1\le u\le \ell}$ a list of positive integers summing to $k$, and $\wbn = (n_u)_{1\le u\le \ell}$ positive integers summing to $n$ such that $n_u\ge j_u$ for all $u\in[\ell]$.
We can compute the probability $\prob{V_\rho(k), \widetilde \cT_n^\red = \widetilde t , \wbj(\cT_n) = \wbj, \wbn(\cT_n) = \wbn}$ by describing the permutation/tree pairs $(\sigma,T)$ satisfying these conditions.
To construct such a tree, one needs to choose $t_u$'s and $F_u$'s with sizes given by the $j_u$'s, then $\tau_i$'s and $\Phi_i$'s with positive sizes $\nu_i$ such that $\sum_{i\in F_u} \nu_i = n_u$, and consistent $p_v$'s.
Note also that once the partition $\wbF$ has been chosen, the permutation $\sigma$ is uniquely determined.
We get, by decomposing the tree $\cT_n$ according to the structure of its reduced tree,
\begin{align*}
    &\ell! \, Z_{n,q} \cdot \prob{V_\rho(k), \widetilde \cT_n^\red = \widetilde t , \wbj(\cT_n) = \wbj, \wbn(\cT_n) = \wbn}
    \\&= q^{\ell-1} \sum_{(t_\upsilon)_{\upsilon\in[\ell]} \in \prod\bT_{j_\upsilon}} 
    \! \left( \prod_{\upsilon\in [\ell]} q^\des{t_\upsilon} \right) \!\!
    \sum_{\substack{\wbF \vdash [k]\\ \abs{F_\upsilon}=j_\upsilon \\ \text{for all }\upsilon}}
    \sum_{\substack{(\nu_i)_{i\in[k]} \in \bN^k\\ \sum_{i\in F_\upsilon}\nu_i = n_\upsilon \\ \\ \text{for all }\upsilon}}
    \sum_{\substack{(\tau_i)_{i\in[k]} \in \prod\bT_{\nu_i}\\ \text{root labeled }1}} 
    \binom{n-k}{\nu_1-1, \dots, \nu_k-1}
    \prod_{1\le i\le k} q^\des{\tau_i}
    \\&\hspace{1.1cm} 
    \cdot \prod_{\substack{v \in \tilde t\setminus\rho\\\text{parent }=:\, u(v)}} \left( n_{u(v)} - j_{u(v)} \right)
    \\&= q^{\ell-1} (n-k)!
    \left( \prod_{\upsilon\in [\ell]} \sum_{t_\upsilon \in \bT_{j_\upsilon}} q^\des{t_\upsilon} \right) \!\!\!
    \sum_{\substack{\wbF \vdash [k]\\ \abs{F_\upsilon}=j_\upsilon \\ \text{for all }\upsilon}}
    \sum_{\substack{(\tau_i)_{i\in[k]} \in \bT^k\\ \text{root labeled }1}}
    \prod_{\upsilon\in [\ell]} \left( 
    \One{\sum_{i \in F_\upsilon}\abs{\tau_i} = n_\upsilon}
    \left( n_\upsilon - j_\upsilon \right)^{c_\upsilon}
    \prod_{i\in F_\upsilon} \frac{q^\des{\tau_i}}{(\abs{\tau_i}-1)!}
    \right) \,.
\end{align*}
Now, observe that the term inside the sum $\sum_{\wbF}$  does not depend on the partition $\wbF$.
Therefore, we can fix an arbitrary partition $\wbF$ and write
\begin{align*}
    &\ell! \, Z_{n,q} \cdot \prob{V_\rho(k), \widetilde \cT_n^\red = \widetilde t , \wbj(\cT_n) = \wbj, \wbn(\cT_n) = \wbn}
    \\&= q^{\ell-1} (n-k)!
    \left( \prod_{\upsilon\in [\ell]} \sum_{t_\upsilon \in \bT_{j_\upsilon}} q^\des{t_\upsilon} \right) \!
    \binom{k}{\wbj}
    \sum_{\substack{(\tau_i)_{i\in[k]} \in \bT^k\\ \text{root labeled }1}}
    \prod_{\upsilon\in [\ell]} \left( 
    \One{\sum_{i \in F_\upsilon}\abs{\tau_i} = n_\upsilon}
    \left( n_\upsilon - j_\upsilon \right)^{c_\upsilon}
    \prod_{i\in F_\upsilon} \frac{q^\des{\tau_i}}{(\abs{\tau_i}-1)!}
    \right) 
    \\&= q^{\ell-1} k!(n-k)!
    \prod_{\upsilon\in [\ell]} \left( 
    (n_\upsilon - j_\upsilon)^{c_\upsilon}
    \sum_{t_\upsilon \in \bT_{j_\upsilon}} \frac{q^\des{t_\upsilon}}{j_\upsilon!}
    \sum_{\substack{(\tau_i)_{i\in F_\upsilon} \in \bT^{F_\upsilon}\\ \text{root labeled }1}}
    \One{\sum_{i \in F_\upsilon}\abs{\tau_i} = n_\upsilon}
    \prod_{i\in F_\upsilon} \frac{q^\des{\tau_i}}{(\abs{\tau_i}-1)!}
    \right) 
    \\&= q^{\ell-1} k!(n-k)!
    \prod_{\upsilon\in [\ell]} \left( 
    (n_\upsilon - j_\upsilon)^{c_\upsilon}
    \cdot [x^{j_\upsilon}] A(x,q)
    \cdot [x^{n_\upsilon-j_\upsilon}] \left( \partial_x A^{(1)}(x,q) \right)^{j_\upsilon}
    \right)
\end{align*}
where $A^{(1)}(x,q)$  is the generating function for descent-biased trees with root labeled $1$.
By \Cref{lem:trees with smallest root}, we deduce Lemma \ref{lem:combinatorial formulas}~(i). 

Hence, the probability that $\widetilde \cT_n^{k,\red} = \widetilde t$ only depends on $\widetilde t$ through its sequence of out-degrees $\wbc$. This immediately implies Lemma \ref{lem:combinatorial formulas}~(ii).

Then, recall that the number of labeled trees $\widetilde t \in \bT_\ell$ with out-degrees $\wbc$ is $\frac{(\ell-1)!}{\prod\limits_{\upsilon\in[\ell]} c_\upsilon!}$. Therefore, using Lemma \ref{lem:combinatorial formulas}~(i), we get Lemma \ref{lem:combinatorial formulas}~(iii).

Now, sum over all $\ell$-tuples $\wbc$ of non-negative integers summing to $\ell-1$:
\begin{align*}
    &\prob{V_\rho(k), \wbj(\cT_n) = \wbj, \wbn(\cT_n) = \wbn}
    \\&= \frac{k!(n-k)!}{\ell Z_{n,q}} q^{\ell-1} \!\!\sum_{\substack{(c_\upsilon)_{\upsilon\in[\ell]} \in \bN_0^\ell\\\sum c_\upsilon = \ell-1}}
    \prod_{\upsilon\in [\ell]} \left( 
    \frac{(n_\upsilon - j_\upsilon)^{c_\upsilon}}{c_\upsilon!}
    \cdot [x^{j_\upsilon}] A(x,q)
    \cdot [x^{n_\upsilon-j_\upsilon}] e^{j_\upsilon A(x,q)}
    \right) 
    \\&= \frac{k!(n-k)!}{\ell Z_{n,q}} q^{\ell-1}
    [x^{\ell-1}] 
    \prod_{\upsilon\in[\ell]} \left( 
    \sum_{c\ge0} \frac{(n_\upsilon-j_\upsilon)^c x^c}{c!} 
    [x^{j_\upsilon}] A(x,q)
    \cdot [x^{n_\upsilon-j_\upsilon}] e^{j_\upsilon A(x,q)}
    \right) 
    \\&= \frac{k!(n-k)!}{\ell Z_{n,q}} q^{\ell-1}
    [x^{\ell-1}] e^{x(n-k)}
    \prod_{\upsilon\in [\ell]} \left( 
    [x^{j_\upsilon}] A(x,q)
    \cdot [x^{n_\upsilon-j_\upsilon}] e^{j_\upsilon A(x,q)}
    \right) \,,
\end{align*}
which implies Lemma \ref{lem:combinatorial formulas} (iv).

Finally, sum over all $\ell$-tuples $\wbn$ of non-negative integers summing to $n$ (the condition $n_\upsilon\ge j_\upsilon$ appears naturally in the formula):
\begin{align*}
    \prob{V_\rho(k), \wbj(\cT_n) = \wbj}
    &= \frac{k!(n-k)!}{\ell! Z_{n,q}} \left(q(n-k)\right)^{\ell-1}
    \left( \prod_{\upsilon\in [\ell]}
    [x^{j_\upsilon}] A(x,q) \right)
    \sum_{n_1+ \dots+ n_\ell = n} \prod_{\upsilon\in [\ell]} [x^{n_\upsilon}] x^{j_\upsilon} e^{j_\upsilon A(x,q)}
    \\&= \frac{k!(n-k)!}{\ell! Z_{n,q}} \left(q(n-k)\right)^{\ell-1}
    \left( \prod_{\upsilon\in [\ell]}
    [x^{j_\upsilon}] A(x,q) \right)
    \cdot 
    [x^n] \prod_{\upsilon\in[\ell]} x^{j_\upsilon} e^{j_\upsilon A(x,q)} \,,
\end{align*}
which is Lemma \ref{lem:combinatorial formulas} (v).
\end{proof}

\subsubsection{Proof of Proposition \ref{prop: joint convergence of tilde j tilde c tilde n towards PoissonDirichlet}}

We now have all the tools to investigate the behavior of our sequences, starting with the asymptotics of $\bj'(\cT_n)$.

\begin{lemma}\label{lem: asymptotic distribution tilde j Ewens}
    Conditionally on $V_\rho(k)$, the sequence $\bj'(\cT_n) := (j_{(1)}\ge \dots\ge j_{(\ell)})$ 
    %(that is, the non-increasing ordering of $\bj(\cT_n)$) 
    converges in distribution as $n\to\infty$ to the non-increasing ordering of 
    %the sequence of 
    cycle lengths in a Ewens random permutation with size $k$ and parameter $a$.
\end{lemma}

% \begin{remark}
% Notice that $\bj'(\cT_n)$ is not necessarily equal to $\bj^\searrow(\cT_n)$.  
% \end{remark}

\begin{proof}[Proof of \Cref{lem: asymptotic distribution tilde j Ewens}]
    Take $n\to\infty$ in Lemma \ref{lem:combinatorial formulas}~(v), using $A(x,q) \cv{} A(x,0) = -\log(1-x)$ along with the same asymptotics as before.
    This yields, for any $\ell$-tuple $\wbj$ of positive integers summing to $k$,
    \begin{align*}
        \prob{V_\rho(k), \wbj(\cT_n) = \wbj}
        \cv{n\to\infty}
        k!
        \frac{ k\Gamma(a+1)}{\Gamma(a+k+1)}
        \frac{a^{\ell-1}}{\ell!}
        \prod_{\upsilon\in [\ell]} \frac{1}{j_\upsilon} \,.
    \end{align*}
    Thus, by \Cref{lem: asymptotic probability of V1},
    \begin{align}
    \label{eq:ewens1}
        \probcond{ \wbj(\cT_n) = \wbj }{V_\rho(k)}
        \cv{n\to\infty}
        k!
        \frac{\Gamma(a)}{\Gamma(a+k)}
        \frac{a^{\ell}}{\ell!}
        \prod_{\upsilon\in [\ell]} \frac{1}{j_\upsilon} \,.
    \end{align}
    Now let $\pi\in\fS_k$ be Ewens-distributed with parameter $a$, and let $\wbj(\pi_k)$ be its sequence of cycle lengths after a uniformly random ordering.
    Then,
    \begin{align}
    \label{eq:ewens2}
        \prob{ \wbj(\pi_k) = \wbj }
        = \frac{a^\ell}{Z_{k,a}}
        \cdot \frac{1}{\ell!} 
        \cdot\binom{k}{\bj} \prod_{\upsilon\in[\ell]} (j_\upsilon-1)!
        = \frac{k!}{Z_{k,a}}
        \frac{a^{\ell}}{\ell!}
        \prod_{\upsilon\in [\ell]} \frac{1}{j_\upsilon} \,,
    \end{align}
    where $Z_{k,a} = \frac{\Gamma(a+k)}{\Gamma(a)}$ is the partition function for the Ewens model.
    Comparing Equations \eqref{eq:ewens1} and \eqref{eq:ewens2} concludes the proof.
\end{proof}

Recall that the non-increasing ordering of cycle lengths in an Ewens random permutation with size $k$ and parameter $a$ (also known as the $k$-th step of the Chinese restaurant process with parameter $a$) converges, as $k\to\infty$ and after $\frac1k$-rescaling, to the Poisson--Dirichlet distribution $\PoissonDirichlet{0}{a}$.
%This explains the first marginal convergence in \Cref{prop: joint convergence of tilde j tilde c tilde n towards PoissonDirichlet}.
Additionally, the number of cycles follows a central limit theorem with asymptotic mean $a\log k$ (see, e.g., \cite{Arratia_Barbour_Tavare_2003}).
Therefore, from \Cref{lem: asymptotic probability of V1,lem: asymptotic distribution tilde j Ewens}, we immediately deduce the following corollaries.

\begin{corollary}\label{cor: asymptotic concentration ell}
    We have
    \begin{equation*}
        \lim_{k\to\infty} \limsup_{n\to\infty} \abs{ \frac{\ell^k(\cT_n)}{a\log k} - 1 } = 0
    \end{equation*}
    in probability.
\end{corollary}

\begin{corollary}\label{cor: asymptotic distribution tilde j PoissonDirichlet}
    Let $\bj'(\cT_n)$ be the non-increasing reordering of $\wbj(\cT_n)$.
    Then,
    \begin{equation*}
        \lim_{k\to\infty} \lim_{n\to\infty} \Law{ \frac1k \bj'(\cT_n) } = \PoissonDirichlet{0}{a}
    \end{equation*}
    weakly on the infinite-dimensional simplex.
\end{corollary}

The final step in the proof of \Cref{prop: joint convergence of tilde j tilde c tilde n towards PoissonDirichlet} is to establish asymptotic proportionality between the rescaled sequences $\bj(\cT_n)$, $\bc(\cT_n)$, $\bn(\cT_n)$.
In what follows, we let $d_\infty$ denote the supremum distance between sequences of real numbers.

\begin{lemma}\label{lem: proportionality tilde j tilde n}
    The $\ell$-tuples $\frac1n \bn'(\cT_n)$ and $\frac1k\bj'(\cT_n)$ are close in the following sense:
    for any $\delta>0$,
    \begin{align*}
    \limsup_{n \rightarrow \infty} \prob{d_\infty{\left( \frac1k \bj'(\cT_n), \frac1n \bn'(\cT_n) \right)}
    \ge \delta} 
    \cv{k\to\infty} 0 \,.
    \end{align*}
\end{lemma}

\begin{lemma}\label{lem: proportionality tilde j tilde c}
    The $\ell$-tuples $\frac{1}{\ell-1} \bc'(\cT_n)$ and $\frac1k\bj'(\cT_n)$ are close in the following sense:
    for any $\delta>0$,
    \begin{align*}
    \limsup_{n \rightarrow \infty} \prob{d_\infty{\left( \frac1k \bj'(\cT_n), \frac{1}{\ell^k(\cT_n)-1} \bc'(\cT_n) \right)}
    \ge \delta} 
    \cv{k\to\infty} 0 \,.
    \end{align*}
\end{lemma}

\begin{proof}[Proof of \Cref{prop: joint convergence of tilde j tilde c tilde n towards PoissonDirichlet}]
\Cref{prop: joint convergence of tilde j tilde c tilde n towards PoissonDirichlet} follows directly from \Cref{cor: asymptotic distribution tilde j PoissonDirichlet,cor: asymptotic concentration ell} together with \Cref{lem: proportionality tilde j tilde n,lem: proportionality tilde j tilde c}.    
\end{proof}

It remains to prove \Cref{lem: proportionality tilde j tilde n,lem: proportionality tilde j tilde c}. 
To this end, we start by computing the asymptotic density of $\frac1n \wbn(\cT_n)$.

\begin{lemma}\label{lem: asymptotic density of tilde n conditional tilde j}
    Fix $\ell\ge2$ and a sequence of positive integers $\wbj=(j_1, \dots, j_\ell)$ summing to $k$.
    Conditionally on the events $V_\rho(k)$ and $\wbj(\cT_n) = \wbj$, the $\ell$-tuple $\frac1n \wbn(\cT_n)$ converges in distribution, as $n\to\infty$, to the density 
    \begin{equation*}
        \alpha \mapsto \frac{\Gamma(a+k+1)}{k} \prod_{\upsilon\in [\ell]} \frac{j_\upsilon \cdot \alpha_\upsilon^{\alpha_\upsilon a+j_\upsilon-1}}{\Gamma(\alpha_\upsilon a + j_\upsilon+1)}
    \end{equation*}
    with respect to the Lebesgue measure on the ($\ell\!-\!1$)--dimensional simplex 
    \[
        \Delta_{\ell} := \left\{ (\alpha_1,\dots,\alpha_{\ell-1}) \;:\; \alpha_1,\dots,\alpha_{\ell-1} \ge 0 ,\, \alpha_1+\dots+\alpha_{\ell-1} \le 1 \right\} \,,
    \]
    seen as a compact subset of $\bR^{\ell-1}$, where we write $\alpha_\ell := 1 - (\alpha_1+\dots+\alpha_\ell)$.
\end{lemma}

In particular, it has the following marginal distributions.

\begin{corollary}\label{cor: asymptotic density of n1 conditional tilde j}
    Fix $\ell\ge2$ and a non-increasing sequence of positive integers $\bj'=(j_1 \ge \dots \ge j_\ell)$ summing to $k$.
    Conditionally on the events $V_\rho(k)$ and $\bj'(\cT_n) = \bj'$, for each $i\in[\ell]$, the random variable~$\frac1n n_i'(\cT_n)$ converges in distribution to the density
    \begin{equation*}
        f_{k,j_i}: \alpha \mapsto \frac{j_i(k-j_i)}{k} \Gamma(a+k+1) \frac{\alpha^{\alpha a+j_i-1}}{\Gamma(\alpha a + j_i +1)} \frac{(1-\alpha)^{(1-\alpha)a+k-j_i-1}}{\Gamma((1-\alpha)a+k-j_i+1)}
    \end{equation*}
    on $[0,1]$.
\end{corollary}

\begin{proof}[Proof of \Cref{lem: asymptotic density of tilde n conditional tilde j}]
    By Lemma \ref{lem:combinatorial formulas} (iv) and (v), we have, for any sequence $\wbn$ of positive integers summing to $n$,
    \begin{align}\label{eq: reduced tree distribution tilde n conditional tilde j}
        \probcond{\Big.\wbn(\cT_n)=\wbn}{V_\rho(k), \wbj(\cT_n)=\wbj} 
        = \frac{1}{ [x^{n-k}] e^{k A(x,q)} } \prod_{\upsilon\in[\ell]} [x^{n_\upsilon-j_\upsilon}] e^{j_\upsilon A(x,q)} \,.
    \end{align}
    To work out the asymptotics related to this distribution, we rely on Lemma \ref{lem: asymptotic coefficient exponentials in transition regime} implying that for all $\alpha\in[0,1]$, all $n \in \bN$ such that $\alpha n \in \bN$ and all $j\ge1$,
    \begin{align}\label{eq: estimate coefficient xalphan of expA transition regime uniform in alpha}
        [x^{\alpha n}] e^{jA(x,q)} 
        = [x^{\alpha n}] e^{jA(x, \frac{\alpha a}{\alpha n})} 
        = \frac{j (\alpha n)^{\alpha a+j-1} e^{-\alpha a}}{\Gamma(\alpha a+j+1)} \left(1+o(1)\right),
    \end{align}
    where the $o$ is uniform in $\alpha$. Now, let $h$ be a continuous map on $\Delta_\ell$.
    Taking $n\to\infty$ in \eqref{eq: reduced tree distribution tilde n conditional tilde j} with the asymptotics \eqref{eq: estimate coefficient xalphan of expA transition regime uniform in alpha}, we get
    \begin{align*}
        &\expecond{ h{\left(\frac1n\wbn(\cT_n)\right)} }{V_\rho(k), \wbj(\cT_n)=\wbj} 
        \\&= \sum_{\substack{j_\upsilon\le n_\upsilon\le n \text{ for all }\upsilon\in[\ell-1]\\ n_{\ell}:= n- \sum_{\upsilon<\ell} n_\upsilon}}
        h{\left(\frac1n\wbn\right)}
        \probcond{\Big. \wbn(\cT_n)=\wbn}{V_\rho(k), \wbj(\cT_n)=\wbj} 
        \\&= \sum_{\substack{j_\upsilon\le n_\upsilon\le n \text{ for all }\upsilon\in[\ell-1]\\ n_{\ell}:= n- \sum_{\upsilon<\ell} n_\upsilon}}
        \One{n_\ell\ge j_\ell}
        h{\left(\frac1n\wbn\right)}
        \frac{\Gamma(a+k+1)}{k n^{a+k-1} e^{-a}} \prod_{\upsilon\in[\ell]} \frac{j_\upsilon (n_\upsilon-j_\upsilon)^{(n_\upsilon-j_\upsilon) a/n +j_\upsilon -1} e^{-(n_\upsilon-j_\upsilon) a/n}}{\Gamma((n_\upsilon-j_\upsilon) a/n+j_\upsilon+1)} \left(1+o(1)\right)
        \\&= \sum_{\substack{j_\upsilon\le n_\upsilon\le n \text{ for all }\upsilon\in[\ell-1]\\ n_{\ell}:= n- \sum_{\upsilon<\ell} n_\upsilon}}
        \One{n_\ell\ge j_\ell}
        h{\left(\frac1n\wbn\right)}
        \frac{\Gamma(a+k+1)}{k n^{a+k-1} e^{-a}} \prod_{\upsilon\in[\ell]} \frac{j_\upsilon n_\upsilon^{n_\upsilon a/n +j_\upsilon -1} e^{-n_\upsilon a/n}}{\Gamma(n_\upsilon a/n+ j_\upsilon +1)} \left(1+o(1)\right)
        \\&= \sum_{\substack{j_\upsilon\le n_\upsilon\le n \text{ for all }\upsilon\in[\ell-1]\\ n_{\ell}:= n- \sum_{\upsilon<\ell} n_\upsilon}}
        \One{n_\ell\ge j_\ell}
        h{\left(\frac1n\wbn\right)}
        \frac{\Gamma(a+k+1)}{k n^{\ell-1}} \prod_{\upsilon\in[\ell]} \frac{j_\upsilon (n_\upsilon/n)^{n_\upsilon a/n +j_\upsilon -1}}{\Gamma(n_\upsilon a/n+ j_\upsilon +1)} \left(1+o(1)\right)\,,
    \end{align*}
    where the $o$ is uniform in the $n_\upsilon$'s.
    One recognizes a (truncated) $(\ell-1)$--dimensional Riemann sum for the continuous map 
    \begin{equation*}
        (\alpha_1, \dots, \alpha_{\ell-1}) \mapsto 
        h{\left(\alpha\right)}
        \prod_{\upsilon\in[\ell]} \frac{j_\upsilon \alpha_\upsilon^{\alpha_\upsilon a +j_\upsilon-1}}{\Gamma(\alpha_\upsilon a+j_\upsilon +1)}
    \end{equation*}
    on the ($\ell\!-\!1$)--dimensional simplex 
    (where we write $\alpha_\ell := 1 - \sum_{\upsilon<\ell} \alpha_\upsilon$), yielding
    \begin{align}
    \label{eq:strange computation}
        \expecond{ h\left(\frac1n\wbn(\cT_n)\right) }{V_\rho(k), \wbj(\cT_n)=\wbj} 
        \cv{} \frac{\Gamma(a+k+1)}{k} 
        \int_{0}^{1} \mathrm{d}\alpha_1 \ldots \int_{0}^{1} \mathrm{d}\alpha_{\ell-1} \One{\alpha_\ell\ge0}
        h(\alpha) \prod_{\upsilon\in[\ell]} \frac{j_\upsilon \alpha_\upsilon^{\alpha_\upsilon a +j_\upsilon -1}}{\Gamma(\alpha_\upsilon a+j_\upsilon +1)}
    \end{align}
    as $n\to\infty$.
    This concludes the proof.
\end{proof}

\begin{remark}
In particular, \Cref{eq:strange computation} provides a computation of the nontrivial integral
    \begin{equation*}
        \frac{\Gamma(a+k+1)}{k} \int_{0}^{1} \mathrm{d}\alpha_1 \ldots \int_{0}^{1} \mathrm{d}\alpha_{\ell-1} \One{\alpha_\ell\ge0}
        \prod_{\upsilon\in[\ell]} \frac{j_\upsilon \alpha_\upsilon^{\alpha_\upsilon a +j_\upsilon -1}}{\Gamma(\alpha_\upsilon a+j_\upsilon +1)}
        = 1 \,,
    \end{equation*}
    where $\alpha_\ell := 1 - \sum_{\upsilon<\ell} \alpha_\upsilon$.
\end{remark}

From that we deduce the marginal distribution of $n_1$.

\begin{proof}[Proof of \Cref{cor: asymptotic density of n1 conditional tilde j}]
We need to compute the marginal distribution of $\alpha :=\alpha_{i}$, where $(\alpha_1, \dots, \alpha_\ell)$ is the weak limit of $\frac1n\bn'(\cT_n)$ conditionally given $\bj'(\cT_n)$.
Since \Cref{lem: asymptotic density of tilde n conditional tilde j} holds for $\wbn$ when conditioning on $\wbj$, it also holds for $\bn'$ when conditioning on $\bj'$, after reordering the vectors to make $\wbj$ non-increasing (with ties broken arbitrarily).
Therefore, we get that $\alpha$ has density
\begin{align*}
    \alpha \mapsto \frac{\Gamma(a+k+1)}{k} \frac{j \alpha^{\alpha a +j-1}}{\Gamma(\alpha a+j+1)} 
    %\int\limits_{\sum_{\upsilon \in [\ell]\backslash\{1\}}\alpha_\upsilon=1-\alpha}
    \int_{[0,1]^{\ell-2}} \One{\sum_{\upsilon\in[\ell-1]\setminus\{i\}} \alpha_\upsilon \le 1-\alpha}
    \prod_{\upsilon \in [\ell] \backslash \{i\}} \frac{j_\upsilon \alpha_\upsilon^{\alpha_\upsilon a +j_\upsilon-1}}{\Gamma(\alpha_\upsilon a+j_\upsilon+1)} 
    \prod_{\upsilon\in[\ell-1]\setminus\{i\}} \mathrm{d}\alpha_{\upsilon}
\end{align*}
on $[0,1]$, where we write $j:=j_i(\cT_n)$ and $\alpha_\ell := 1-\alpha - \sum_{\upsilon\in[\ell-1]\setminus\{i\}} \alpha_\upsilon$.
In order to compute that integral, define $b := (1-\alpha) a$ and, for any $v \in [\ell]\setminus\{i\}$, $\beta_\upsilon = \frac{\alpha_\upsilon}{1-\alpha}$. 
Then, we have $\beta_\upsilon b=\alpha_\upsilon a$, and we can rewrite
\begin{align*}
    \prod_{\upsilon \in [\ell] \backslash \{i\}} \frac{j_\upsilon \alpha_\upsilon^{\alpha_\upsilon a +j_\upsilon-1}}{\Gamma(\alpha_\upsilon a+j_\upsilon+1)} &= \prod_{\upsilon \in [\ell] \backslash \{i\}} \frac{j_\upsilon (\beta_\upsilon(1-\alpha))^{\beta_\upsilon b +j_\upsilon-1}}{\Gamma(\beta_\upsilon b+j_\upsilon+1)}\\
    &= (1-\alpha)^{\sum (\beta_\upsilon b+j_\upsilon-1)} \prod_{\upsilon \in [\ell] \backslash \{i\}} \frac{j_\upsilon \beta_\upsilon^{\beta_\upsilon b +j_\upsilon-1}}{\Gamma(\beta_\upsilon b+j_\upsilon+1)} \\
    &= (1-\alpha)^{(1-\alpha) a-(\ell-1)+k-j} \prod_{\upsilon \in [\ell] \backslash \{i\}} \frac{j_\upsilon \beta_\upsilon^{\beta_\upsilon b +j_\upsilon-1}}{\Gamma(\beta_\upsilon b+j_\upsilon+1)},
\end{align*}
where we have used that $\sum\limits_{\upsilon \in [\ell] \backslash \{i\}} \beta_\upsilon = 1$, $b=(1-\alpha)a$, and $\sum\limits_{\upsilon \in [\ell] \backslash \{i\}} j_\upsilon = k-j$.
We also know from Lemma \ref{lem: asymptotic density of tilde n conditional tilde j} that 
\begin{equation*}
    (\beta_\upsilon, \upsilon\in[\ell-1]\setminus\{i\}) \mapsto \frac{\Gamma(b+k-j+1)}{k-j} \prod_{\upsilon\in[\ell]\setminus\{i\}} \frac{j_\upsilon \beta_\upsilon^{\beta_\upsilon b +j_\upsilon-1}}{\Gamma(\beta_\upsilon b+j_\upsilon +1)}
\end{equation*}
is a probability density on $\Delta_{\ell-1}$, where we write $\beta_\ell := 1 - \sum_{\upsilon\in[\ell-1]\setminus\{i\}} \beta_\upsilon$. 
Hence, we have
\begin{align*}
    &%\int\limits_{\sum_{u=2}^\ell \alpha_u=1-\alpha} 
    \int_{[0,1]^{\ell-2}} \One{\sum_{\upsilon\in[\ell-1]\setminus\{i\}} \alpha_\upsilon \le 1-\alpha}
    \prod_{\upsilon \in [\ell] \backslash \{i\}} \frac{j_\upsilon \alpha_\upsilon^{\alpha_\upsilon a +j_\upsilon-1}}{\Gamma(\alpha_\upsilon a+j_\upsilon +1)} \prod_{\upsilon\in[\ell-1]\setminus\{i\}}\mathrm{d}\alpha_\upsilon 
    \\&= (1-\alpha)^{\ell-2} 
    %\int\limits_{\sum_{\upsilon=2}^\ell \beta_\upsilon=1}  
    \int_{[0,1]^{\ell-2}} \One{\sum_{\upsilon\in[\ell-1]\setminus\{i\}} \beta_\upsilon \le 1}
    (1-\alpha)^{(1-\alpha) a-(\ell-1)+k-j} \prod_{\upsilon \in [\ell] \backslash \{i\}} \frac{j_\upsilon \beta_\upsilon^{\beta_\upsilon b +j_\upsilon-1}}{\Gamma(\beta_\upsilon b+j_\upsilon +1)} \prod_{\upsilon\in[\ell-1]\setminus\{i\}}\mathrm{d}\beta_\upsilon \\
    &= (1-\alpha)^{(1-\alpha) a-1+k-j} \frac{k-j}{\Gamma((1-\alpha)a+k-j+1)} \,.
\end{align*}
Altogether, we find that $\alpha$ has the desired density on $[0,1]$.
\end{proof}

We can finally prove \Cref{lem: proportionality tilde j tilde n,lem: proportionality tilde j tilde c}.

\begin{proof}[Proof of \Cref{lem: proportionality tilde j tilde n}]
We will show that, conditionally on $\bj'(\cT_n)$ being ``typical'', the vector $\frac1n \bn'(\cT_n)$ is concentrated around $\frac{1}{k} \bj'(\cT_n)$.
Let $i\in\bN$ be fixed.
Thanks to \Cref{cor: asymptotic concentration ell}, we can assume with high probability that $i \le \ell$.
Write $j := j_i(\cT_n)$ and let $f_{k,j}(\alpha)$ be the asymptotic density of $\frac1n n_i(\cT_n)$ found in \Cref{cor: asymptotic density of n1 conditional tilde j}.
Fix $\varepsilon \in (0,1/2)$, and let $\beta = \frac{j}{k}$. 
We have the following approximation, uniformly in $\beta \in [\varepsilon, 1-\varepsilon]$ and in $\alpha \in (0,1)$, as $k \rightarrow \infty$:
\begin{align*}
    f_{k,j}(\alpha) &=\Gamma(a+k+1) \frac{j(k-j)}{2k\pi \sqrt{(\alpha a +j)((1-\alpha)a+k-j)}} \frac{\alpha^{\alpha a-1+j}}{\left(\frac{\alpha a +j}{e}\right)^{\alpha a +j}} \frac{(1-\alpha)^{(1-\alpha) a-1+k-j}}{\left(\frac{(1-\alpha) a +k-j}{e}\right)^{(1-\alpha) a +k-j}} (1+o(1))\\
    &=\Gamma(a+k+1)\frac{\sqrt{\beta(1-\beta)}}{2\pi} e^{a+k} \frac{1}{\alpha(1-\alpha)} \left( \frac{\alpha}{\alpha a+j} \right)^{\alpha a + j} \left( \frac{1-\alpha}{(1-\alpha)a+k-j} \right)^{(1-\alpha)a+k-j} (1+o(1)) \,.
\end{align*}
We now expand
\begin{align*}
\log\left(\left( \frac{\alpha}{\alpha a+j} \right)^{\alpha a +j}\right) 
= (\alpha a + j) \big( \log(\alpha/j)+\cO(1/k) \big)
= j \log(\alpha/j) + \cO(\log k) \,,
\end{align*}
again uniformly for $\beta \in [\varepsilon, 1-\varepsilon]$ and $\alpha \in (0,1)$. 
Hence,
\begin{align*}
    \log &\left( \frac{1}{\alpha(1-\alpha)} \left( \frac{\alpha}{\alpha a+j} \right)^{\alpha a + j} \left( \frac{1-\alpha}{(1-\alpha)a+k-j} \right)^{(1-\alpha)a+k-j} \right)
    \\ &\qquad \qquad= j \log(\alpha/j) + (k-j)\log((1-\alpha)/(k-j)) - \log(\alpha(1-\alpha)) + \cO(\log k)
    \\&\qquad \qquad= \left(1+\cO\left(\frac1k\right)\right) \left[\Big. j \log(\alpha/j) + (k-j) \log((1-\alpha)/(k-j)) \right]
    \\&\qquad \qquad= \left(1+\cO\left(\frac1k\right)\right) k \left[\Big. \beta \log(\alpha/\beta)+ (1-\beta) \log((1-\alpha)/(1-\beta)) - \log k \right] \,,
\end{align*}
% indeed, $-\log\alpha = \cO -\beta\log\alpha \le \cO -\beta\log\alpha + \beta\log j = \cO \frac1k j\log(j/\alpha)
again uniformly for $\beta \in [\varepsilon, 1-\varepsilon]$ and $\alpha \in (0,1)$. 
The function $g: \alpha\mapsto  \beta \log(\alpha/\beta)+ (1-\beta) \log((1-\alpha)/(1-\beta))$ is strictly concave and reaches its only maximum at $\alpha_*=\beta$, where $g(\alpha_*)=0$. 
Furthermore, $g''(\alpha)<-1$ for any $\alpha \in (0,1)$ and any $\beta \in [\varepsilon, 1-\varepsilon]$. 
Thus, Taylor's formula provides the following:
\begin{itemize}
    \item for all $\alpha$ such that $|\alpha-\alpha_*|<k^{-1/2}$, there exists $\gamma$ between $\alpha$ and $\alpha_*$ such that
    \begin{align*}
    g(\alpha)
    =g(\alpha_*)+\frac{g''(\gamma)}{2}(\alpha-\alpha_*)^2
    \ge \frac{1}{2k} g''(\alpha_*)(1+o(1))
    \end{align*}
    as $k \rightarrow \infty$, uniformly in $\beta \in [\varepsilon, 1-\varepsilon]$.
    \item for all $\alpha$ such that $|\alpha-\alpha_*|>k^{-1/4}$, there exists $\gamma$ between $\alpha$ and $\alpha_*$ such that
    \begin{align*}
    g(\alpha)
    =g(\alpha_*)+\frac{g''(\gamma)}{2}(\alpha-\alpha_*)^2
    \le -\frac{1}{2\sqrt{k}} \,.
    \end{align*}
\end{itemize}
Thus, we get
\begin{multline*}
    \left(\Gamma(a+k+1) e^{a+k} \frac{\sqrt{\beta(1-\beta)}}{2\pi}\right)^{-1}\int_{|\alpha-\alpha_*|<k^{-1/2}} f_{k,j}(\alpha) \mathrm d\alpha 
    \\= (1+o(1)) \int_{|\alpha-\alpha_*|<k^{-1/2}} e^{-k \log k}e^{kg(\alpha)} e^{\cO(g(\alpha)-\log k)} \mathrm d\alpha
    \ge 2k^{-1/2}e^{-k \log k} e^{-\cO(\log k)} \,,
\end{multline*}
while
\begin{align*}
    \left(\Gamma(a+k+1) e^{a+k} \frac{\sqrt{\beta(1-\beta)}}{2\pi}\right)^{-1} \int_{|\alpha-\alpha_*| 
    > k^{-1/4}} f_{k,j}(\alpha) \mathrm d\alpha 
    \le e^{- C\sqrt{k}}e^{-k \log k}e^{-\cO(\log k)},
\end{align*}
for some $C>0$ independent of $\beta \in [\varepsilon, 1-\varepsilon]$. 
Hence, we get that
\begin{equation}
\label{eq: cv of alpha}
    \probcond{ \abs{\alpha_i-\frac{j_i(\cT_n)}{k}} \le k^{-1/4} \quad}{\quad j_i(\cT_n) \in [\varepsilon k, (1-\varepsilon)k] } \cv{k\to\infty} 1 \,,
\end{equation}
where $i\in\bN$ is fixed and $\alpha_i$ is distributed under the $f_{k,j_i}$ density.
We can now conclude the proof. 
Fix $\eta>0$. 
Using the Poisson--Dirichlet asymptotics of $\frac1k \bj'(\cT_n)$ from \Cref{cor: asymptotic distribution tilde j PoissonDirichlet}, we get the following. 
There exist $\varepsilon>0$ and $p\ge1$ such that, for all $1\le i\le p$ and all large enough $n$, $\prob{j_i'(\cT_n) \in [\varepsilon k, (1-\varepsilon)k]} \geq 1-\eta$, and in addition
\begin{align*}
    \prob{\sum_{i=1}^p j_i' \ge (1-\eta)k} 
    \ge 1-\eta \,.
\end{align*}
Using \eqref{eq: cv of alpha} for each of the $j_i'$'s, $1 \leq i \leq p$, we get the result.
\end{proof}

\begin{proof}[Proof of \Cref{lem: proportionality tilde j tilde c}]
    From Lemma \ref{lem:combinatorial formulas} (iii) and (iv), we deduce that
    \begin{equation}\label{eq: reduced tree distribution tilde c conditional tilde j tilde n}
        \probcond{\big.\wbc(\cT_n)=\wbc}{V_\rho(k), \wbj(\cT_n)=\wbj, \wbn(\cT_n)=\wbn}
        = \frac{(\ell-1)!}{(n-k)^{\ell-1}} \prod_{\upsilon\in[\ell]} \frac{(n_\upsilon-j_\upsilon)^{c_\upsilon}}{c_\upsilon!} \,.
    \end{equation}
    Therefore, the same formula holds when reordering all vectors to make $\wbj$ non-increasing.
    For fixed $i$ ($1\le i\le \ell$), let us derive the conditional distribution of $c_i'(\cT_n)$. We have
    \begin{align*}
        \probcond{\big. c_i'(\cT_n)=c_i'}{V_\rho(k), \bj'(\cT_n)=\bj', \bn'(\cT_n)=\bn'} \nonumber
        &= \sum_{\substack{c_\upsilon' \ge 0, \upsilon\in[\ell]\setminus\{i\}\\\sum_{\upsilon\in[\ell]\setminus\{i\}} c_\upsilon'=\ell-1-c_i'}} \frac{(\ell-1)!}{(n-k)^{\ell-1}} \prod_{\upsilon\in[\ell]} \frac{(n_\upsilon'-j_\upsilon')^{c_\upsilon'}}{c_\upsilon'!} \nonumber
        \\&= \frac{(\ell-1)! (n_i'-j_i')^{c_i'}}{(n-k)^{\ell-1} c_i'!} [x^{\ell-1-c_i'}] \prod_{\upsilon\in[\ell]\setminus\{i\}} e^{x(n_\upsilon'-j_\upsilon')} \nonumber
        \\&= \frac{(\ell-1)! (n_i'-j_i')^{c_i'}}{(n-k)^{\ell-1} c_i'!} \frac{(n-n_i'-k+j_i')^{\ell-1-c_i'}}{(\ell-1-c_i')!} \,.
    \end{align*}
    That is, the conditional law of $c_i'(\cT_n)$, under the event $V_\rho(k)$ and given $\bj'(\cT_n), \bn'(\cT_n)$, is $\BinomialDistribution{\ell-1}{\frac{n_i'-j_i'}{n-k}}$.   
    It remains to show the desired concentration result for this distribution.

    Fix an arbitrary $p\in(0,1)$.
    For a given $\varepsilon>0$, let $E_\varepsilon$ be the event that $\frac1k j_i'(\cT_n) \in [\varepsilon,1-\epsilon]$.
    We know from \Cref{cor: asymptotic distribution tilde j PoissonDirichlet} that there exists a small enough $\varepsilon = \varepsilon(p)$, a $k_0$ and for each $k$ an $m_0(k)$, such that
    \begin{equation*}
        \text{for all }k\ge k_0 \text{ and for all }n\ge m_0(k),\quad
        \prob{E_\varepsilon} \ge 1-p \,.
    \end{equation*}
    Then, fix $0< \delta\le \frac13\varepsilon$ and use \Cref{lem: proportionality tilde j tilde n}.
    There exists $k_1$, and for each $k$ there exists $m_1(k)$, such that
    \begin{equation*}
        \text{for all }k\ge k_1 \text{ and for all }n\ge m_1(k),\quad
        \prob{ \abs{ \frac1k j_i'(\cT_n) - \frac1n n_i'(\cT_n) } \le \delta } \ge 1-p \,.
    \end{equation*}
    Let $E_\delta'$ be the above event.
    Under $E_\delta'$, we have $\abs{ \frac{n_i'-j_i'}{n-k} - \frac1k j_i' } \le \delta + \frac{2k}{n-k}$.
    For each $k$, let $m_2(k)$ be such that $\frac{2k}{n-k} \le \delta$ for all $n\ge m_2(k)$.
    Under the events $E_\varepsilon$ and $E_\delta'$ and for any $n\ge m_2(k)$, we have $\abs{ \frac{n_i'-j_i'}{n-k} - \frac1k j_i' } \le 2\delta$, and further $\frac{n_i'-j_i'}{n-k} \in [\frac13\varepsilon, 1-\frac13\varepsilon]$.
    
    For each $\ell\ge2$ and $q\in[0,1]$, write $S_{\ell,q}$ for a $\BinomialDistribution{\ell-1}{q}$ random variable.
    Then, there exists $\ell_0$ such that
    \begin{equation*}
        \text{for all }q\in[\tfrac13\varepsilon, 1-\tfrac13\varepsilon] \text{ and for all }\ell\ge\ell_0,\quad
        \prob{\abs{\frac{S_{\ell,q}}{\ell-1}-q} \le \delta} \ge 1-p \,.
    \end{equation*}
    Furthermore, again by \Cref{cor: asymptotic distribution tilde j PoissonDirichlet}, there exists $k_3$, and for each $k$ there exists $m_3(k)$, such that:
    \begin{equation*}
        \text{for all }k\ge k_3 \text{ and for all }n\ge m_3(k),\quad
        \prob{ \ell(\cT_n) \ge \ell_0 } \ge 1-p \,.
    \end{equation*}
    Combining all of the above, we find that
    \begin{equation*}
        \text{for all } k\ge \max(k_0,k_1,k_3),\text{ and for all } n\ge \max(m_0,m_1,m_2,m_3),\quad
        \prob{ \abs{ \frac{c_i'(\cT_n)}{\ell(\cT_n)-1} - \frac1k j_i'(\cT_n) } \le 3\delta } \ge 1-4p \,,
    \end{equation*}
    for any fixed $i\in\bN$.
    The conclusion follows in the same way as in the proof of \Cref{lem: proportionality tilde j tilde n}.
\end{proof}

\subsection{Convergence towards the limit dendron}
\label{ssec: limit dendron convergence}

Here we prove \Cref{th: dendron limit if a/n}, that is, the convergence of the descent-biased tree $\cT_n := \cT_n^{(a/n)}$ towards a dendron with respect to the Gromov-weak topology recalled in \Cref{sec: Gromov Prokhorov}. 
We start by defining the limit dendron, through a generalization of the so-called Foata--Fuchs bijection, in Section \ref{sssec:foata}. 
Then, after gathering some preliminary results in \Cref{ssec: descents in ancestral line,sssec:prelim}, we describe a coupling between discrete trees and the limit dendron in \Cref{sssec:coupling}. 
We then use this coupling in \Cref{sssec:final proof} to prove the convergence of the discrete descent-biased trees to the limit dendron.

\subsubsection{The Foata--Fuchs bijection and the Poisson--Dirichlet dendron}
\label{sssec:foata}

\paragraph{The Foata--Fuchs bijection.}

The construction of the limit dendron $\cD_a$ is based on a generalization to infinite trees of the Foata--Fuchs bijection, which is a bijection between finite trees with a given sequence of out-degrees and a certain set of lists of integers. 
We start by recalling this bijection; see, e.g., \cite{Addario-Berry_Blanc-Renaudie_Donderwinkel_Maazoun_Martin_2023} for more details, and \Cref{fig: finite Foata} for an example. 

\begin{definition}[The finite Foata--Fuchs construction]
\label{def:finite foata fuchs}
Fix $m \in \bN$ and consider a sequence ${\bx} := (x_1,\ldots,x_{m-1}) \in [m]^{m-1}$. We build a tree $T_\bx \in \bT_m$ from $\bx$ as follows. 
Let $D(\bx)$ be the set of elements of $[m]$ which do not appear in $\bx$.
Let $(u_1, \ldots, u_m)$ be the vertices of $T_\bx$, where $u_i$ has label $i$. Then,
\begin{itemize}
    \item the root of $T_\bx$ is the vertex $u_{x_1}$;
    \item if $x_2 \neq x_1$, then $u_{x_2}$ is a child of $u_{x_1}$;
    \item if $x_2=x_1$, then $u_{d_1}$ is a child of $u_{x_1}$, where $d_1=\min D(\bx)$.
    \item Iteratively, we have, for $3\le i \leq m-1$:
    \begin{itemize}
        \item if $x_i \notin \{x_1,\ldots,x_{i-1}\}$, then $u_{x_i}$ is a child of $u_{x_{i-1}}$;
        \item if $x_i \in \{x_1,\ldots,x_{i-1}\}$, then $u_d$ is a child of $u_{x_{i-1}}$, where $d$ is the smallest element of $D(\bx)$ not used yet.
    \end{itemize}
    \item Finally, $u_d$ is a child of $u_{x_{m-1}}$, where $d$ is the only remaining element in $D(\bx)$.
\end{itemize}
One can check that the tree $T_\bx$ constructed that way is indeed an element of $\bT_m$, and that the out-degree of a vertex $u_i$ is exactly the number of occurrences of $i$ in $\bx$.
\end{definition}

\begin{figure}
    \centering
    \includegraphics[width=.8\linewidth]{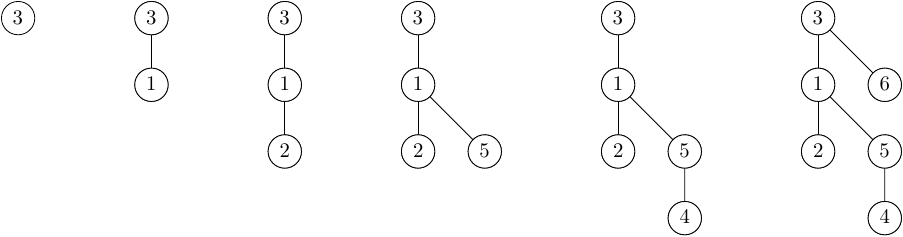}
    \caption{Construction of the tree $T_\bx \in \bT_6$ associated with the sequence $\bx = (3,1,1,5,3)$ through the Foata--Fuchs construction.}
    \label{fig: finite Foata}
\end{figure}

This construction turns out to be a bijection between $[m]^{m-1}$ and $\bT_m$. 
Its restriction to sequences with a given number of occurrences of each integer is also in bijection with the subset of trees with a given sequence of out-degrees. 
This is the content of the following theorem, whose proof can be found e.g. in \cite{Addario-Berry_Blanc-Renaudie_Donderwinkel_Maazoun_Martin_2023}.

\begin{theorem}[The Foata--Fuchs bijection]
\label{thm:foata fuchs}
Fix $m \geq 1$, and define the set 
\[
    C_m:= \Big\{ (c_1,\ldots,c_m) \in \bN_0^m : \sum_{i=1}^m c_i=m-1 \Big\} \,.
\] 
For any $\bc := (c_1,\ldots,c_m) \in C_m$, let $\bT_{m,\bc}$ be the subset of $\bT_m$ consisting of all trees whose sequence of out-degrees is $\bc$ (that is, such that the vertex labeled $i$ has $c_i$ children, for all $i \in [m]$). 
Then,
\begin{itemize}
    \item[(i)] the Foata--Fuchs construction is a bijection from $[m]^{m-1}$ to $\bT_m$;
    \item[(ii)] for any $\bc \in C_m$, the Foata--Fuchs construction is a bijection from $[m]^{m-1}_{\bc}$ to $\bT_{m,\bc}$.
\end{itemize}
Here, $[m]^{m-1}_{\bc}$ denotes the set of $(m\!-\!1)$-tuples of elements of $\{1,\ldots,m\}$ with $c_i$ occurrences of $i$, for all $i \in [m]$.
\end{theorem}

\begin{remark}
    In \Cref{def:finite foata fuchs}, we explain how to construct a tree given its degree sequence.
    This is the main use of the Foata--Fuchs bijection that we will need.
    For completeness, let us briefly describe the inverse map as well:
    consider a tree $T\in\bT_m$ with vertices $(u_1, \dots, u_m)$, where $u_i$ has label $i$.
    Let $\bc$ be the degree sequence of $T$; we wish to construct $\bx\in[m]^{m-1}_\bc$ such that $T = T_\bx$.
    
    First, let $u_d$ be the leaf of $T$ with lowest label and let $P_1 := (u_{x_1}, \dots, u_{x_j}, u_d)$ denote the shortest path from the root of $T$ to this leaf.
    We set the first $j$ elements of $\bx$ to $(x_1, \dots, x_j)$.
    Then, let $u_{d'}$ be the leaf of $T$ with the second lowest label.
    Let $P_2 := (u_{y_1}, \dots, u_{y_{j'}}, u_{d'})$ denote the shortest path from $P_1$ to $u_{d'}$.
    In particular, $y_1 \in \{x_1, \dots, x_j\}$.
    The next elements of $\bx$ are then defined as $(y_1, \dots, y_{j'})$ (notice that, then, $y_1$ appears for the second time in $\bx$).
    
    This procedure goes on iteratively: at each step, we follow the shortest path from the previously visited vertices to the leaf with the lowest label not yet visited.
    In this procedure, we follow each edge $e$ precisely once, and this corresponds to an occurrence (in $\bx$) of the label of the endpoint of $e$ that is closer to the root. 
    Hence, one can check that $\bx\in[m]^{m-1}_\bc$. See, e.g., \cite{Addario-Berry_Blanc-Renaudie_Donderwinkel_Maazoun_Martin_2023} for details.
\end{remark}

We now provide a countably infinite version of this Foata--Fuchs bijection, which builds an infinite tree from an infinite sequence of positive integers.

\begin{definition}[The countable Foata--Fuchs construction]
\label{def:countable foata}
Consider a sequence $\bz := (z_1,z_2,\ldots) \in \bN^\bN$ of positive integers. 
We construct a real tree $(T_\bz,\rho_\bz,d_\bz)$ from it as follows.
\begin{itemize}
    \item We first build a discrete tree $T^{\discrete}(\bz)$ with countably many vertices denoted by $\{v_1,v_2, \ldots\}$:
    \begin{itemize}
        \item the root $\rho_\bz$ of the tree is $v_{z_1}$;
        \item for all $j \geq 2$ such that $z_j \notin \{z_1,\ldots,z_{j-1}\}$, $v_{z_j}$ is a child of $v_{z_{j-1}}$, and we draw an edge between them.    
        \end{itemize}
    \item Then, $(T_\bz,\rho_\bz,d_\bz)$ is the real tree obtained from the discrete tree $T^{\discrete}(\bz)$ by regarding each edge as an interval of length $1$.
\end{itemize}
\end{definition}

From this, we can define the limit dendron $\cD_a$.
See \Cref{fig: PD dendron} for an illustration.

\begin{definition}[The dendron $\cD_a$]\label{def: size-biased dendron}
Fix $a>0$ and let $(X_1,X_2,\ldots)$ be distributed as a Poisson--Dirichlet variable $\PoissonDirichlet{0}{a}$. 
Then, define a sequence $\boldZ := (Z_j)_{j \geq 1}$ of i.i.d.\ random variables such that $\prob{Z_1=i}=X_i$, for all $i \in \mathbb{N}$. 
We construct the random dendron $\cD_a:=(T_a,\rho_a,d_a,\nu_a)$ as follows. 
\begin{itemize}
    \item The base tree $(T_a,\rho_a,d_a)$ is the tree $(T_{\boldZ},\rho_\boldZ,d_\boldZ)$ built from $\boldZ$ by the countably infinite Foata--Fuchs construction of Definition \ref{def:countable foata};
    \item the mass measure $\nu_a$ on $T_a \times [0,\infty)$ is defined as
    \begin{align*}
        \nu_a = \sum_{i \geq 1} X_i \delta_{(v_i,1)} \,,
    \end{align*}
    where $(v_i)_{i\ge1}$ are the vertices of the discrete tree $T^{\discrete}(\boldZ)$.
\end{itemize}
\end{definition}

\begin{figure}
    \centering
    \includegraphics[scale=.9]{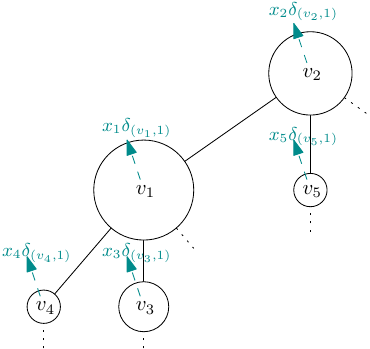}
    \caption{First steps in the construction of the dendron $\cD_a$ after sampling $X \approx (.32, .23, .14, .11, .09, \ldots)$ and $Z = (2,1,4,1,3,2,5,1,3,\ldots)$.
    The green arrows represent the mass measure $\nu_a$.
    The vertices of the discrete tree $T_a^{\discrete}$ are represented with sizes roughly proportional to their associated masses.}
    \label{fig: PD dendron}
\end{figure}

Our aim is to prove \Cref{th: dendron limit if a/n}, which states the convergence in distribution of $\frac{1}{\log n}\cT_n$ towards the random dendron $\cD_a$.
To this end, we fix $p \geq 1$ and consider $p$ i.i.d.\ uniform vertices $w_1^{(n)}, \ldots, w_p^{(n)}$ of the tree $\cT_n$. 
We will show the following result, which is equivalent to \Cref{th: dendron limit if a/n} by Proposition~\ref{prop: convergence to a random dendron}:
\begin{theorem}
\label{thm:convergence of distances critical window}
For any $p \geq 1$, the (random) matrix 
\begin{align*}
    \left( \frac{1}{\log n} d{\left( w_r^{(n)}, w_s^{(n)} \right)} \right)_{1 \leq r,s \leq p}
\end{align*}
converges in distribution to the random matrix $(d_a(w_{r}, w_{s}))_{1 \leq r,s \leq p}$, where $w_{1}, \ldots, w_{p}$ are i.i.d.\ points on $\cD_a$ drawn according to the measure $\nu_a$.
\end{theorem}

\begin{remark}
With the exact same ingredients, one can prove that, for all $p \geq 1$, the random matrix 
\begin{align*}
    \left( \frac{1}{\log n} d{\left( w_r^{(n)}, w_s^{(n)} \right)} \right)_{0 \leq r,s \leq p} \,,
\end{align*}
with the convention that $w_0^{(n)}$ is the root of $\cT_n$,
converges in distribution to the random matrix $(d_a(w_{r}, w_{s}))_{0 \leq r,s \leq p}$, where $w_{1}, \ldots, w_{p}$ are i.i.d. points on $\cD_a$ drawn according to the measure $\nu_a$ and $w_{0} := (\rho_a,0)$.
\end{remark}

\subsubsection{Descents in the ancestral line of a uniform vertex}
\label{ssec: descents in ancestral line}

One of the steps towards the convergence to the dendron $\cD_a$ consists in proving that, with high probability, the ancestral line of a typical vertex does not contain any descent to a vertex with a ``large'' label. 
More precisely, a consequence of \Cref{lem: asymptotic probability of V1} is the following result. 
For any vertex $u$ in a tree $T$, let $D_u(T)$ be the set of labels of ancestors of $u$ in $T$ (\textbf{including} $u$) that have a smaller label than their parent.
In other words, $D_u(T)$ contains the endpoints of descents on the path from the root to $u$.
Then, we have the following.

\begin{proposition}
\label{prop:not many descents on the ancestral line}
    Conditionally given the tree $\cT_n$, let $u_n$ be a uniformly random vertex in $\cT_n$. 
    Then,
    \begin{align*}
      \liminf_{n\to\infty} \prob{ D_{u_n}(\cT_n) \subseteq \{1, \ldots, K\} } \cv{K\to\infty} 1 \,.
    \end{align*}
\end{proposition}

Equivalently, $\max D_{u_n}(\cT_n) = \cO_\bP(1)$ as $n\to\infty$. 
In other words, the endpoints of the descents in the ancestral line of a typical vertex have bounded labels. 
In order to prove this property, we first prove a structural result on descent-biased permutations, which is reminiscent of \cite[Lemma 15]{Thevenin_Wagner_2023}.

\begin{lemma}
\label{lem:structure_of_permutations}
Let $\des{\sigma}$ be the number of descents of a permutation $\sigma$.
Fix $a > 0$ and let $q_n := a/n$.
Fix also $0<c<C$. 
For $c\log n\le h\le C\log n$, let $\sigma^{h}$ be a descent-biased permutation with size $h$ and bias $a/n$.
Following \cite{Thevenin_Wagner_2023}, we write $S^{(a/n)}_{h}$ for the law of $\sigma^h$.
Then, we have the following.
\begin{enumerate}[label=(\roman*)]
    \item For all $1\le r\le h$, let $\sigma^{h,r}$ be distributed as $S^{(a/n)}_{h}$, conditionally on its last element $\sigma_h=r$.
    The number of descents of $\sigma^{h,r}$ is uniformly stochastically bounded: 
    \begin{align*}
    \liminf_{n\to\infty} \inf_{c\log n\le h\le C\log n} \inf_{1\le r\le h}
    \prob{\des{\sigma^{h,r}} \le Q} \cv{Q\to\infty} 1 \,.
    \end{align*}
    \item For all $d\ge 1$, let $\sigma^{h,r,d}$ be distributed as $S^{(a/n)}_{h}$, conditionally on $\sigma_h=r$ and $\des{\sigma}=d$. 
    For a permutation~$\sigma$, let $D(\sigma)=\{\sigma_i : 2\le i\le h, \sigma_{i-1} > \sigma_i\}$ be the set of endpoints of its descents.
    Then, for all fixed $d\ge1$,
    \begin{align*}
    \liminf_{n\to\infty} \inf_{c\log n\le h\le C\log n} \inf_{1\le r\le h}
    \prob{D(\sigma^{h,r,d}) \subseteq \{1,\ldots,K\} } \cv{K\to\infty} 1 \,.
    \end{align*}
    \item For $d=1$, the first element of $\sigma^{h,r,d}$ can take any finite value less than $r$ with probability uniformly bounded from below.
    That is, for all fixed $K\ge1$, we have
    \begin{align*}
    \liminf_{n\to\infty} \inf_{c\log n\le h\le C\log n} \inf_{1\le r\le h} \inf_{1\le j\le \min(K,r-1)}
    \prob{\sigma_1^{r,h,d} = j} > 0 \,.
    \end{align*}
    \item For $d\ge2$, the first element of $\sigma^{h,r,d}$ can take any finite value, except $r$, with probability uniformly bounded from below.
    That is, for all fixed $d\ge2$ and $K\ge1$, we have
    \begin{align*}
    \liminf_{n\to\infty} \inf_{c\log n\le h\le C\log n} \inf_{1\le r\le h} \inf_{j\in [1,K] \setminus \{r\}}
    \prob{\sigma_1^{r,h,d} = j} > 0 \,.
    \end{align*}
\end{enumerate}
\end{lemma}

Let us immediately see how this implies Proposition \ref{prop:not many descents on the ancestral line}.

\begin{proof}[Proof of Proposition \ref{prop:not many descents on the ancestral line}]
Fix $\varepsilon>0$. First, by Proposition \ref{prop: asymptotic coefficient typical depth in transition regime}, there exist $0<c<C$ such that, for all sufficiently large $n$,
\[
    \prob{\dist(\rho,u_n)\in [c \log n, C \log n]} \geq 1-\varepsilon \,. 
\]
Write $h := \dist(\rho,u_n)$ for ease of notation.
Let $\rho =: v_0, \ldots, v_{h} := u_n$ be the ancestors of $u_n$, $v_j$ having height $j$, and let $L(v_0), \ldots, L(v_{h})$ be their labels.
Write $L_i$ for the $i$-th smallest of these labels, and let $r$ be such that $L_r = L(u_n)$.
We call $r = \rk(u_n)$ the rank of $u_n$; by definition, $u_n$ has precisely $r-1$ ancestors with labels lower that its own.
Now, perform the following shuffling operation on the tree $\cT_n$:
\begin{itemize}
    \item Let $\sigma \in \fS_{h+1}$ be distributed as $S_{h+1}^{(a/n)}$, conditionally on $\{ \sigma_{h+1} = r \}$. 
    \item Let $(v'_j)_{0 \leq j \leq h-1}$ be the vertices $(v_j)_{0 \leq j \leq h-1}$ reordered so that $v'_{j-1}$ has label $L_{\sigma_j}$.
    \item For $0\le j\le h-1$, write $C_j$ for the connected component of $\cT_n \backslash \bigcup_{i=1}^{h}\{(v_{i-1}, v_{i})\}$ containing $v_j$.
    Then reconnect the $C_j$'s according to this new order dictated by $\sigma$.
    In particular, $v'_0$ is now the root of the tree that we obtain.
\end{itemize}
Let $\tilde{\cT}_n$ be the tree obtained this way. 
Then, clearly, $(\tilde{\cT}_n, u_n) \overset{(d)}{=} (\cT_n,u_n)$.
See \Cref{fig: shuffling} for an illustration.

\begin{figure}
    \centering
    \includegraphics[width=.45\linewidth]{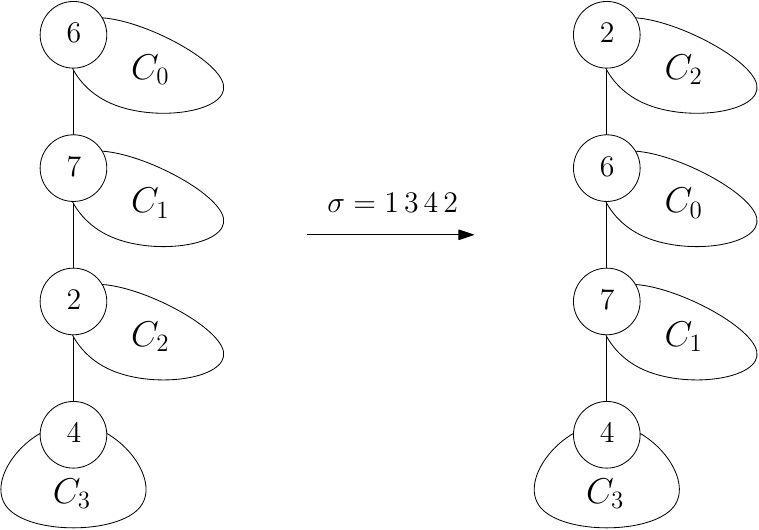}
    \caption{Example of an ancestral reshuffling.
    Left: 
    the vertex $u_n$ has height $h=3$ and label $4$.
    The ranks along its ancestral line are $3,4,1,2$.
    In particular, $r=2$.
    Right: 
    the resulting tree after shuffling the ancestral line with $\sigma = 1\,3\,4\,2$.}
    \label{fig: shuffling}
\end{figure}

Observe that $\big|D_{u_n}\big(\tilde\cT_n\big)\big|= \des{\sigma}$, and furthermore that $D_{u_n}\big(\tilde\cT_n\big)$ can be identified with $D(\sigma)$ via the ranks of the labels.
By Lemma \ref{lem:structure_of_permutations}~(i), there exists $Q>0$ such that
\begin{align*}
\liminf_{n\to\infty} 
\inf_{c\log n\le h\le C\log n} \inf_{1\le r\le h}
\probcond{ \des{\sigma} \le Q \;}{\, h(u_n) = h, \rk(u_n)=r} \ge 1-\varepsilon \,.
\end{align*}
From now on, we restrict ourselves to the case $\des{\sigma}\le Q$. 
%If $\des{\sigma}=0$ then the conclusion follows trivially, so we can assume $\des{\sigma}\ge1$.
From \Cref{lem:structure_of_permutations}~(ii), we get that for all $1\le d\le Q$,
\begin{align*}
\liminf_{n\to\infty} 
\inf_{c\log n \le h \le C\log n}\inf_{1\le r\le h} 
\probcond{ D_{u_n}\big(\tilde\cT_n\big) \subseteq \{L_1,\ldots,L_K\} \, }{ \, \des{\sigma}=d, h(u_n)=h, \rk(u_n)=r} \cv{K\to\infty} 1 \,.
\end{align*}
Hence, we can choose $K$ so large that, with probability at least $1-\varepsilon$, the labels of all descent endpoints are among the smallest $K$ labels in the ancestral line of $u_n$.
It remains to deduce that these labels are stochastically bounded.

We know by \Cref{lem: asymptotic probability of V1} that the root label of the tree is asymptotically stochastically bounded.
In addition, we have by \Cref{lem:structure_of_permutations}~(iv) that for all $2\le d\le Q$,
\begin{equation*}
    \liminf_{n\to\infty} 
    \inf_{c\log n\le h\le C\log n} \inf_{1\le r\le h} \inf_{\substack{1\le j\le K \\ j\ne r}} 
    \probcond{ L(v_0') = L_j \,}{\, \des{\sigma}=d, h(u_n)=h, \rk(u_n)=r} > 0 \,.
\end{equation*}
The latter extends the stochastic boundedness of the root label $L_{\sigma_1}$ to a stochastic boundedness for the smallest $K$ labels $L_1, \ldots, L_K$, conditionally on $\des{\sigma}\geq 2$. This concludes the proof when $\des{\sigma}\geq 2$.

When $\des{\sigma}=1$, we just need to use the fact that $D_{u_n}\big(\cT_n\big)$ contains at most one label which is among the $r-1$ smallest ones.
We can therefore conclude as before, using \Cref{lem:structure_of_permutations}~(iii).
\end{proof}

We now prove Lemma \ref{lem:structure_of_permutations}.

\begin{proof}[Proof of Lemma \ref{lem:structure_of_permutations}]
Let us start by proving (i). 
Let $r\in[h]$.
Write $Z_{h,r,a/n}$ for the partition function of descent-biased permutations conditioned on $\sigma_h=r$, that is,
\begin{equation*}
    Z_{h,r,a/n} := \sum_{\sigma\in\fS_h} \One{\sigma_h=r} (a/n)^{\des{\sigma}} \,.
\end{equation*}
Then, let us bound the probability of $\des{\sigma^{h,r}}\ge Q$.
To this aim, we can use \cite[Corollary~16~(ii)]{Thevenin_Wagner_2023} which bounds the number of permutations with a given number of descents and a given first value.
Thanks to the bijection $\sigma\in\fS_h \mapsto \left( n+1-\sigma_{n+1-i} \right)_{1\le i\le n}$, we know that the number of permutations with $d$ descents and ending with $r$ equals the number of permutations with $d$ descents and starting with $n-r+1$.
Therefore,
\begin{align*}
    \prob{\des{\sigma^{h,r}}\ge Q}
    \le Z_{h,r,a/n}^{-1} \sum_{d\ge Q} d^{r-1} (d+1)^{h-r} (a/n)^d
    \le Z_{h,r,a/n}^{-1} \sum_{d\ge Q} (d+1)^{h} (a/n)^d \,.
\end{align*}
In order to bound the partition function, define $\tau$ as the cycle $(r,r+1, \dots, h)$ (if $r=h$, then $\tau=\id_h$).
In particular $\tau_{h}=r$, and $\tau$ has at most one descent so that $\prob{\sigma^{h,r} = \tau} \ge Z_{h,r,a/n}^{-1} \cdot (a/n)$. 
This further entails that $Z_{h,r,a/n}^{-1} \le (a/n)^{-1}$, and thus
\begin{align*}
    \prob{\des{\sigma^{h,r}}\ge Q}
    \le \sum_{d\ge Q} (d+1)^{h} (a/n)^{d-1} \,.
\end{align*}
Under the assumption $h\le C\log n$, we get
\begin{align*}
    \prob{\des{\sigma^{h,r}}\ge Q} 
    \le a^{-1} \sum_{Q\le d\le C\log n} n^{C\log(d+1) + (C\log a)^+ - d + 1}
\end{align*}
where $(C\log a)^+ = \max(0, C\log a)$.
If $Q$ is large enough, this goes to $0$ as $n\to\infty$, proving (i).

In order to prove (ii), (iii) and (iv), we rely on a construction of $\sigma^{h,r,d}$ that is similar to the one presented in \cite[Lemma~15]{Thevenin_Wagner_2023}. 
Fix $d\ge 1$. 
Let $A_1, \ldots, A_{r-1}$ be i.i.d.~random variables, uniform on $\{1,\ldots,d+1\}$, and $A_{r+1}, \ldots, A_h$ be i.i.d.~random variables, uniform on $\{1,\ldots,d\}$. 
For $1\le i\le d+1$, let $L_i := \{j: A_j = i\}$. 
Let finally $\tau^{h,r,d}$ be the permutation obtained by concatenating $L_1,\ldots,L_{d+1},\{r\}$, where the elements in each $L_i$ are placed in increasing order. 
Then, uniformly for $h \in [c\log n, C\log n]$ and $r \in [h]$, $\prob{\tau^{h,r,d} \text{ has } d \text{ descents}} \cv{n\to\infty} 1$; 
furthermore, under this event, $\tau^{h,r,d}$ is distributed as $\sigma^{h,r,d}$. 
We omit the proof, which is the exact same as for \cite[Lemma~15]{Thevenin_Wagner_2023}.
Now observe that, for any $K \geq 1$ and any $1 \leq i \leq d$, we have
\begin{equation*}
\prob{\min L_i \geq K} 
= \prod\limits_{\substack{j=1 \\ j\ne r}}^{K-1} \prob{A_j \neq i}
\le \left( \frac{d}{d+1} \right)^{K-2} \,,
\end{equation*}
which goes to $0$ as $K \rightarrow \infty$. 
Furthermore, for $i=d+1$, we have 
\begin{align*}
\prob{\min L_{d+1} \cup \{r\} \geq K} 
= \One{r \geq K} \prod\limits_{j=1}^{K-1} \prob{A_j \ne d+1} 
\le \left( \frac{d}{d+1} \right)^{K-1} \,,
\end{align*}
which also goes to $0$ as $K \rightarrow \infty$. 
This proves (ii).
For (iii) and (iv), it remains to find an adequate lower bound on the probability that $\tau^{h,r,d}_1 = j$.
If $j<r$, we have
\begin{equation*}
    \prob{ \tau^{h,r,d}_1 = j }
    = \prob{A_j=1} \prod_{1\le i< j} \prob{A_i \ne 1}
    = \left( \frac{d}{d+1} \right)^{j-1} \cdot \frac{1}{d+1} \,,
\end{equation*}
whereas for $j>r$ we have
\begin{equation*}
    \prob{ \tau^{h,r,d}_1 = j }
    = \prob{A_j=1} \prod_{\substack{1\le i< j\\i\ne r}} \prob{A_i \ne 1}
    = \left( \frac{d}{d+1} \right)^{r-1} \cdot \left( \frac{d-1}{d} \right)^{j-r-1} \cdot \frac{1}{d} \,,
\end{equation*}
which is nonzero whenever $d \geq 2$.
The conclusion follows.
\end{proof}

\subsubsection{Structure of the reduced tree}
\label{sssec:prelim}

Now, we gather some preliminary results on the structure of $\cT_n$. 
Recall the notation $\cT_n^{k,\red}$, which is the $k$-reduced tree associated with $\cT_n$. 
From now on, we consider a sequence $(k_n)_{n \geq 1}$ such that $k_n \cv{n\to\infty} +\infty$ slowly enough, so that the following conditions hold.
\begin{itemize}
\item[(H1)] $k_n!=o(n)$.
\item[(H2)] \begin{equation}
\label{eq:kn}
 \left( \frac{1}{k_n}\bj^\searrow(\cT_n) , \frac{1}{a\log k_n} \bc^\searrow(\cT_n) , \frac1n \bn^\searrow(\cT_n) \right)
        \overset{(d)}{\cv{n\to\infty}} \left( X^{(a)} , X^{(a)} , X^{(a)} \right) ,
\end{equation}
where $X^{(a)} \sim \PoissonDirichlet{0}{a}$ and the vectors $\bj^\searrow(\cT_n), \bc^\searrow(\cT_n)$, $\bn^\searrow(\cT_n)$ are constructed from $\cT_n^{k_n,\red}$.
\item[(H3)] 
\begin{equation}
\label{eq:concentration ell}
\frac{\ell^{k_n}(\cT_n)}{a \log k_n} \overset{(d)}{\cv{n\to\infty}} 1 .
\end{equation}
\end{itemize}
The existence of such a sequence is guaranteed by \Cref{prop: joint convergence of tilde j tilde c tilde n towards PoissonDirichlet} and \Cref{cor: asymptotic concentration ell}. 
One reason why we want $\rm (H1)$ to hold is the following result:

\begin{lemma}
\label{lem:no small descent}
Assume that $k_n ! = o(q_n^{-1})$.
Then, with high probability, there is no descent edge between two vertices with labels $\le k_n$ in $\cT_n$.
\end{lemma}

\begin{proof}
%    This can be seen by observing that the trees $t_u, u\in\cT_n^{k_n,\red}$ are, conditionally given their labels, descent-biased trees.
%    However, let us sketch a more direct, self-contained proof for the sake of clarity.
    Let us call a tree $t\in\bT_n$ \enquote{bad} if it contains a descent edge between two labels $\le k_n$, and \enquote{good} otherwise.
    Recall that $Z_{n,q} = \sum_{t\in\bT_n} q_n^{\des{t}}$ denotes the partition function of our model.
    Let $\check{Z}_{n,q}$ denote the restriction of this sum to bad trees $t$.
    Observe that any bad tree $t\in\bT_n$ can be mapped to a good tree $t'\in\bT_n$ through a reordering of the small labels.
    This map strictly decreases the number of descents, and any given good tree $t'$ has at most $k_n!$ preimages under this map.
    Therefore,
    \[
        \check Z_{n,q} \le k_n! \cdot q_n \cdot Z_{n,q} \,,
    \]
    and the claim follows readily.
\end{proof}

In what follows, let $\{ (\cT_n)_i^{k_n} ,\, 1\le i\le \ell^{k_n}(\cT_n) \}$ be the trees $\{ [\cT_n]_u^{k_n} ,\, u \in \cT_n^{k_n,\red}\}$ ranked by non-increasing order of sizes $\bn^\searrow(\cT_n) = \left( (n)_1^{k_n} \ge \ldots \ge (n)_{\ell^{k_n}(\cT_n)}^{k_n} \right)$ (with arbitrary choice in case of equality). 
From now on, we call them \textit{components}. 

The next two lemmas give structural properties of the components of the tree $\cT_n$: 
with high probability, the large components of $\cT_n$ contain a macroscopic number of vertices, and very few of those vertices have an ancestor with large label which is the endpoint of a descent.

\begin{lemma}
\label{lem:lemmerelou4bis}
For all fixed $L \geq 1$ and all $\varepsilon>0$, there exists $\eta>0$ such that
\begin{align*}
    \liminf_{n\to\infty} \prob{ (n)_i^{k_n} \geq \eta n \text{ for all } i\in[L]\;} \geq 1-\varepsilon.
\end{align*}
\end{lemma}

\begin{proof}[Proof of Lemma \ref{lem:lemmerelou4bis}]
    This follows from \Cref{eq:kn} and the fact that, for all $i \geq 1$, the $i$-th element of a $\PoissonDirichlet{0}{a}$ random variable has a density on $\bR_+^*$. 
\end{proof}

For any $k \geq 1$ and any subset $S \subseteq \cT_n$, let $V^{(k)}(S) := \{ u \in S \,:\, \max D_u(\cT_n) \geq k+1\}$ be the set of vertices in $S$ that have an ancestor with label $\geq k+1$ that is the endpoint of a descent, and $N^{(k)}(S) := |V^{(k)}(S)|$ its cardinality.

\begin{lemma}
\label{lem:lemmerelou5bis}
For any $L \geq 1$, we have
\begin{align*}
    \sup\limits_{1 \leq i \leq L} \frac{N^{(k_n)}\big( (\cT_n)_i^{k_n} \big)}{(n)_i^{k_n}} \overset{\bP}{\cv{n\to\infty}} 0
\end{align*}
\end{lemma}

Lemmas \ref{lem:lemmerelou4bis} and \ref{lem:lemmerelou5bis} can be rephrased as follows: for any $L \geq 1$, with high probability each of the largest $L$ components has size at least $\eta n$, and contains a proportion $o(1)$ of vertices that have a large descent endpoint on their ancestral line.
%an ancestor with label $\geq k_n+1$ which is the endpoint of a descent.

\begin{proof}[Proof of Lemma \ref{lem:lemmerelou5bis}]
    Fix $\varepsilon>0$. By \Cref{lem:lemmerelou4bis},  there exists $\eta>0$ such that
    \begin{align*}
        \limsup_{n\to\infty} \prob{ (n)_i^{k_n} < \eta n \text{ for some } i \in [L]} < \varepsilon.
    \end{align*}
    Now fix $\eta'>0$, and let $u_n$ denote a uniformly random vertex of $\cT_n$.
    By \Cref{prop:not many descents on the ancestral line}, there exists $k_0 = k_0(\eta,\eta',\varepsilon) \ge 1$ such that, for all $k \geq k_0$, $\limsup_{n\to\infty}\prob{ \max D_{u_n}(\cT_n) \ge k+1 } \le \eta \eta' \varepsilon$. 
    Since $k_n\to\infty$, we deduce by monotonicity that
    \begin{equation*}
        \limsup_{n\to\infty}\prob{ \max D_{u_n}(\cT_n) \ge k_n+1 } \le \eta \eta' \varepsilon \,.
    \end{equation*}
    This implies by Markov's inequality that
    \begin{align*}
        \limsup_{n\to\infty} \prob{|E_{n}| \geq \eta \eta' n} \le \varepsilon,
    \end{align*}
    where $E_{n} := \{ u\in \cT_n, \max D_u \geq k_n+1\}$.
    On the other hand, under the event $|E_{n}| \leq \eta \eta' n$, for all $i \in [L]$ such that $(n)_i^{k_n} \geq \eta n$, we have $N^{(k_n)}((\cT_n)_i^{k_n}) \leq \eta \eta' n \leq \eta' (n)_i^{k_n}$. 
    Therefore,
    \begin{equation*}
        \limsup_{n\to\infty} \prob{ N^{(k_n)}\big( (\cT_n)_i^{k_n} \big) > \eta' (n)_i^{k_n} \text{ for some } i\in[L]}
        \le 2\varepsilon \,.
    \end{equation*}
    Since this holds for any $\varepsilon>0$ and $\eta'>0$, we get the result.
\end{proof}

\subsubsection{A coupling argument}
\label{sssec:coupling}

We now present a coupling between the random trees $\{\cT_n, n \geq 1\}$ and the dendron $\cD_a$, based on the Foata--Fuchs constructions (recall Theorem \ref{thm:foata fuchs} and Definition \ref{def:countable foata}). 
This coupling keeps track of the location of a large proportion of the total mass of the tree.
More precisely, we prove that, under this coupling, vertices of the reduced tree $\cT_n^{k_n,\red}$ approximate vertices of the dendron $\cD_a$, and that the proportion of vertices in a large component of $\cT_n$ converges to the mass of the corresponding vertex in the dendron.

\paragraph{The $M$-structures of $\cT_n$ and of $\cD_a$.}
\

\smallskip

\noindent
For fixed $k$ and $M$ with $k \geq M \geq 1$, we define the $M$-structure $\St_M^{(k)}(\cT_n)$ of the tree $\cT_n$ as, roughly speaking, the subtree of $\cT_n^{k,\red}$ spanned by the vertices corresponding to the largest $M$ components of $\cT_n$. 

Let $\mathbb{T}_M^{\geq}$ denote the set of finite rooted trees whose vertices have distinct labels in $\mathbb{N}$  and whose leaves all have labels in $[M]$.

\begin{definition}
\label{def: M structure of discrete tree}
Fix $k$ and $M$ with $k \geq M \geq 1$, and fix a tree $T\in\bT_n$ with $n$ vertices labeled from $1$ to $n$. 
Assume that the root label of $T$ is $\le k$.
The $M$-structure $\St_M^{(k)}(T)$ is the element of $\mathbb{T}_M^{\geq}$ that is constructed as follows.
\begin{itemize}
    \item Relabel the vertices of $T^{k,\red}$ as $u_{(1)}, \ldots, u_{(\ell)}$, by non-increasing order of the sizes $[n]_{u_{(1)}}^k \ge \ldots \ge [n]_{u_{(\ell)}}^k$ of the associated components (with arbitrary choice in case of equality);
    \item give the label $i$ to $u_{(i)}$, for $1 \leq i \leq \ell$;
\end{itemize}
Then, define $\St_M^{(k)}(T)$ as the subtree of $T^{k,\red}$ spanned by $u_{(1)}, \ldots, u_{(\min(M,\ell))}$, forgetting the decorations and keeping only the labels defined above. 
By construction, $\St_M^{(k)}(T)$ is an element of $\mathbb{T}_M^{\geq}$.
\end{definition}

\begin{remark}
   By \Cref{eq:concentration ell}, we get that $\ell^{k_n}(\cT_n)/\log k_n \cv{} a$ in probability as $n \to \infty$. 
   In particular, we have 
    \[
    \prob{\min(M,\ell^{k_n}(\cT_n)) = M} \cv{n\to\infty} 1 \,.
    \]
\end{remark}

We also define the continuous analogue, namely, the $M$-structure $\St_M(\cD_a)$ of the dendron $\cD_a$ for $M \geq 1$. 
It is the (almost surely finite) subtree spanned by the vertices with the largest $M$ masses.

\begin{definition}
Recall the definition of the dendron $\cD_a:=(T_a,\rho_a,d_a,\nu_a)$ with measure $\nu_a := \sum_{i \geq 1} X_i \delta_{(v_i,1)}$, with the notation of Definition~\ref{def: size-biased dendron}. 
The $M$-structure $\St_M(\cD_a)$ is the subtree of $T_a$ spanned by $v_1, \ldots, v_M$. 
Giving the label $i$ to each vertex $v_i$ makes it an element of $\mathbb{T}_M^{\geq}$.
\end{definition}

\paragraph{The coupling.}

\ 

\smallskip

\noindent
The main result of this part is the existence of a coupling between the random trees $\{ \cT_n, n \geq 1 \}$ and the dendron $\cD_a$, under which the sequence of distance matrices between $p$ i.i.d.\ uniform points in $\cT_n$ converges to the distance matrix between $p$ i.i.d.\ points in $\cD_a$ distributed according to $\nu_a$. 
This coupling is based on the convergence in distribution of the $M$-structure of $\cT_n$ to the $M$-structure of $\cD_a$, for all fixed $M \geq 1$, see \Cref{lem:proba of a structure}.

By Skorokhod's theorem, we can fix a sequence $(k_n)_{n \geq 1}$ that satisfies $k_n!=o(n)$ and for which the convergences of \Cref{eq:kn,eq:concentration ell} hold almost surely on some probability space (sometimes refered to as the ``Skorokhod space'').
Letting $(X^{(a)}, X^{(a)}, X^{(a)})$ denote the limit triple, we let $\cD_a$ be the dendron constructed from $(X_1^{(a)}, X_2^{(a)}, \ldots) = X^{(a)}$ by Definition \ref{def: size-biased dendron}. 
From now on, we work under the almost sure convergence in this probability space. 
The next lemma shows the convergence of the $M$-structure of $\cT_n$ towards the $M$-structure of $\cD_a$.

\begin{lemma}
\label{lem:proba of a structure}
Fix $M \geq 1$. 
Then the following statements hold:
\begin{itemize}
    \item[(i)] On the Skorokhod space where \Cref{eq:kn,eq:concentration ell} hold a.s., we have, for any given $\mathfrak{T}\in \bT_M^{\geq}$,
    \[
        \probcond{ \St_M^{(k_n)}(\cT_n)=\mathfrak{T} }{\bc^\searrow(\cT_n)} 
        \overset{a.s.}{\cv{n\to\infty}} 
        \probcond{\Big. \St_M(\cD_a)=\mathfrak{T} }{X^{(a)}} \,.
    \]
    \item[(ii)] Let $\max \St_M^{(k_n)}(\cT_n)$ be the largest label of a vertex in the $M$-structure of $\cT_n$. 
    Then $\max \St_M^{(k_n)}(\cT_n) = \cO_\bP(1)$, i.e.:
    \[
        \limsup_{n \rightarrow \infty} \prob{\max \St_M^{(k_n)}(\cT_n) \geq A} \underset{A \rightarrow \infty}{\rightarrow} 0 \,.
    \]
\end{itemize}
\end{lemma}

\begin{proof}
Fix $\varepsilon>0$. 
Let $(Z_1,Z_2,\ldots)$ be i.i.d.\ random variables such that $\prob{Z_1=i}=X_i^{(a)}$ for all $i$, and let 
\[
    t_M := \inf\{ t \geq 1 : \text{for all } j \in [M], \text{ there exists } s \leq t \text{ such that } Z_s= j \}
\] 
be the first time $t \geq 1$ for which $[M] \subseteq \{Z_1,\ldots,Z_t\}$. 
It is clear that $t_M<\infty$ almost surely. 
Hence, one can choose $Q \geq 1$ such that $\prob{t_M \leq Q} \geq 1-\varepsilon$.

For $(z_1,\ldots,z_Q) \in \bN^Q$, let $T(z_1,\ldots,z_Q)$ be the discrete tree obtained from $(z_1,\ldots,z_Q)$ after the first $Q$ steps of the countable Foata--Fuchs construction (\Cref{def:countable foata}). 
That is, the vertices of $T(z_1,\ldots,z_Q)$ are all $v_i$ (where $v_i$ has label $i$) such that there exists $j \in [Q]$ with  $z_j=i$.
The root of $T(z_1,\ldots,z_Q)$ is $v_{z_1}$. For $2 \leq j \leq Q$ such that $z_j \neq \{z_1,\ldots,z_{j-1}\}$, $v_{z_j}$ is a child of $v_{z_{j-1}}$. 
Let $T_M(z_1,\ldots,z_Q) \in \bT_M^\le$ denote the subtree of $T(z_1,\ldots,z_Q)$ obtained by removing all vertices that have no descendant with a label in $[M]$.
For later use,
consider $\fT \in \mathbb{T}_M^{\geq}$ and an arbitrary $(z_1, \ldots, z_Q) \in \bN^Q$ such that $T_M(z_1,\ldots,z_Q)=\fT$, if it exists. 

As in the construction of \Cref{def: M structure of discrete tree}, we label the vertices of $\cT_n^{k_n,\red}$ according to the sizes of the associated components (with ties broken arbitrarily).
That is, we give the label $i$ to the vertex of $\cT_n^{k_n,\red}$ with component $(\cT_n)_i^{k_n}$.
Conditionally given $\ell^{k_n}(\cT_n) = \ell$, consider a vector $\bc = (c_1, \dots, c_\ell)$ of positive integers summing to $\ell-1$. 
Then, conditionally on $\bc^\searrow(\cT_n):= \bc$, we know by \Cref{lem:combinatorial formulas} (ii) that $\cT_n^{k_n,\red}$ is distributed as a uniformly random tree with the prescribed out-degree sequence $(c_1,\ldots,c_\ell)$, where $c_i$ is the out-degree of the vertex labeled $i$.

Let $(x_1,\ldots,x_{\ell-1}) \in [\ell]^{\ell-1}_\bc$ be the image of $\cT_n^{k_n,\red} \in \bT_{\ell,\bc}$ by the Foata--Fuchs bijection (\Cref{thm:foata fuchs}).
Conditionally given $\bc_n := \bc^\searrow(\cT_n)$, we know that $(x_1,\ldots,x_{\ell-1})$ is distributed uniformly on $[\ell]^{\ell-1}_{\bc_n}$.
Then, using \eqref{eq:concentration ell}, we get under our coupling that $\ell^{k_n}(\cT_n) \sim a\log k_n \overset{a.s.}{\cv{n\to\infty}} \infty$, and also
\begin{align*}
    \probcond{\big. (x_1,\ldots,x_Q)=(z_1,\ldots,z_Q)}{\bc_n} = \prod_{i=1}^Q \frac{c_{z_i}(\cT_n)}{a \log k_n} (1+o_{a.s.}(1))
    \overset{a.s.}{\cv{n\to\infty}} \prod_{i=1}^Q X^{(a)}_{z_i}.
\end{align*}
Now, recalling the construction of the dendron $\cD_a$ from the sequence $X^{(a)}$ (Definition \ref{def: size-biased dendron}), we have
\begin{align*}
    \probcond{\Big. (Z_1,\ldots,Z_Q)=(z_1,\ldots,z_Q) }{X^{(a)}} =\prod_{i=1}^Q \prob{Z_i=z_i}
    = \prod_{i=1}^{Q} X^{(a)}_{z_i}.
\end{align*}
Finally, observe that $\St_M(\cD_a)$ equals $T_M(Z_1, \dots, Z_Q)$ for any $Q\ge1$ such that $Q \ge t_M$,
and likewise that $\St_M^{(k_n)}(\cT_n)$ equals $T_M(x_1, \dots, x_Q)$ for any $Q\ge1$ such that $[M] \subseteq \{ x_1, \dots, x_Q \}$.
Hence,
\begin{align*}
\probcond{\St_M^{(k_n)}(\cT_n)=\fT}{\bc_n} 
&\geq \sum\limits_{\substack{(z_1,\ldots,z_Q) \in \bN^Q \\ T_M(z_1,\ldots,z_Q)=\fT}} \probcond{\big. (x_1,\ldots,x_Q)=(z_1,\ldots,z_Q)}{\bc_n} \\
%&=\sum\limits_{\substack{(z_1,\ldots,z_Q) \\ T_M(z_1,\ldots,z_Q)=\fT}}\prod_{i=1}^Q \frac{c_{z_i}}{a \log k_n} (1+o(1))\\
&= \sum\limits_{\substack{(z_1,\ldots,z_Q) \in\bN^Q \\ T_M(z_1,\ldots,z_Q)=\fT}}\prod_{i=1}^Q X_{z_i}^{(a)} (1+o_{\text{a.s.}}(1))\\
& \geq \probcond{\big. \St_M(\cD_a)=\fT }{X^{(a)}} - \prob{t_M > Q} + o(1)\\
& \geq \probcond{\big. \St_M(\cD_a)=\fT }{X^{(a)}} -\varepsilon + o(1) \,.
\end{align*}
Letting $\varepsilon \rightarrow 0$, we find that, for all $\fT \in \bT^{\geq}_M$,
\begin{align*}
\liminf_{n \rightarrow \infty} \probcond{ \St_M^{(k_n)}(\cT_n)=\fT }{\bc_n} \geq \probcond{\Big. \St_M(\cD_a)=\fT }{X^{(a)}}
\end{align*}
almost surely, which concludes the proof of (i).
The proof of (ii) follows from the fact that, for any fixed $Q \geq 1$, $\prob{\max\limits_{1 \leq i \leq Q} Z_i \geq A} \cv{A\to\infty} 0$.
\end{proof}

Therefore, for any fixed $M \geq 1$, we may find a sequence $(k_n)_{n \geq 1}$ and couple the trees $(\cT_n)_{n \geq 1}$ and the dendron $\cD_a$ in such a way that (H1), (H2) and (H3) hold almost surely, and in addition that 
\begin{align*}
\St_M^{(k_n)}(\cT_n) \overset{a.s.}{\cv{n\to\infty}} \St_M(\cD_a) \,.
\end{align*}

The next lemma states that, under the almost sure convergence \eqref{eq:kn}, we also have convergence of the locations of uniform points in the tree $\cT_n$.

\begin{lemma}
\label{lem:same component}
Let $u_n$ be a uniform vertex in $\cT_n$, and $u \in \cD_a$ with law $\nu_a$. 
Then, for all $i \in [M]$, under the almost sure convergence \eqref{eq:kn},
\[
    \probcond{u_n \in (\cT_n)_i^{k_n}}{\bn^\searrow(\cT_n)} 
    \overset{a.s.}{\cv{n\to\infty}} 
    \probcond{\big. u=(v_{i},1)}{X^{(a)}} \,.
\]
\end{lemma}

\begin{proof}
This is immediate by the convergence \eqref{eq:kn} of $\frac1n \bn^\searrow(\cT_n)$ towards $X^{(a)}$, the fact that the coordinates of $X^{(a)}$ are a.s.\ distinct, and the construction of the dendron $\cD_a$.
\end{proof}

\subsubsection{Conditioning on the structure of the reduced subtree spanned by uniform vertices}
\label{sssec:final proof}

We can finally prove the convergence of the distance matrix, that is, Theorem \ref{thm:convergence of distances critical window}. 
Fix $\varepsilon>0$, and consider $M,Q \geq 1$ such that $E_{M,Q}$ holds with probability $\geq 1-\varepsilon$, where $E_{M,Q}$ is the event that
\begin{itemize}
    \item $w_r \in \bigcup\limits_{i=1}^M \{(v_i,1)\}$ for all $r \in [p]$, and
    \item the largest label in $\St_M(\cD_a)$ is $\leq Q$.
    In particular, the size of the tree is $\abs{ \St_M(\cD_a) } \leq Q$.
\end{itemize}
Note that such an $M$ exists because, since $\sum_{i \geq 1} X^{(a)}_i =1$, we have $\nu_a\big( \{(v_1,1),\ldots,(v_M,1)\} \big) \cv{M\to\infty} 1$. 
On the other hand, the existence of $Q$ is immediate by Lemma \ref{lem:proba of a structure} and the fact that the largest label in $\St_M(\cD_a)$ is (deterministically) larger than the number of vertices in $\St_M(\cD_a)$.

In view of the previous results, we can work under (H1), (H2), (H3), and the almost sure convergence 
\begin{align*}
\St_M^{(k_n)}(\cT_n) \overset{a.s.}{\cv{n\to\infty}} \St_M(\cD_a) \,.
\end{align*}
Using also Lemma \ref{lem:same component}, it is enough to prove the convergence of the distance matrices under the following assumptions. Fix $i_1, \ldots, i_p \geq 1$ and assume without loss of generality that $M,Q \geq \max\limits_{1 \leq r \leq p} i_r$. Consider a tree $\mathfrak{T} \in \bT^\geq_M$ with $\leq Q$ vertices that have labels in $[Q]$.
As usual, we write $v_i$ for the vertex with label $i$ in $\mathfrak{T}$.
We will assume that the following holds for sufficiently large $n$:
\begin{itemize}
    \item[(A1)] for all $r \in [p]$, $w_r^{(n)} \in (\cT_n)_{i_r}^{k_n}$;
    \item[(A2)] the $M$-structure of $\cT_n$ is $\mathfrak{T}$;
\end{itemize}
and that the following holds in $\cD_a$:
\begin{itemize}
    \item[(B1)] for all $r \in [p]$, $w_{r} = (v_{i_r},1)$;
    \item[(B2)] the $M$-structure of $\cD_a$ is $\mathfrak{T}$.
\end{itemize}

\begin{proposition}
\label{prop:spanned subtree}
%Fix $i_1, \ldots, i_p \geq 1$ and $Q \geq M \geq \max\limits_{1 \leq r \leq p} i_r$. Consider a tree $\mathfrak{T} \in \bT^\geq_M$ with $\leq Q$ vertices, whose vertices have labels in $[Q]$, where we denote by $v_i$ the vertex with label $i$. 
Assume (A1) and (A2) for sufficiently large $n$. 
Then, conditionally on these events,
\begin{align*}
    \frac{1}{\log n} \left( \dist_{\cT_n}(w_r^{(n)},w_s^{(n)}) \right)_{1 \leq r,s \leq p}
    \overset{\bP}{\cv{n\to\infty}}
    \left( 2+\dist_{\mathfrak{T}}(v_{i_r},v_{i_s}) \right)_{1 \leq r,s \leq p} \,.
\end{align*}
\end{proposition}

\begin{lemma}
\label{lem: distances in a dendron}
Consider the dendron $\cD_a$. 
Let $w_1, \dots, w_p$ be as in (B1), 
%i.i.d. distributed according to $\nu_a$, and assume that (B1) and 
and assume that (B2) holds. 
Then,
\begin{align*}
    \left( d_a(w_{r},w_{s}) \right)_{1 \leq r,s \leq p} = \left(\big. 2+\dist_{\mathfrak{T}}(v_{i_r},v_{i_s})\right)_{1 \leq r,s \leq p} \,.
\end{align*}
\end{lemma}

The proof of Lemma \ref{lem: distances in a dendron} is immediate by definition of distances in a dendron. 
Furthermore, observe that Proposition \ref{prop:spanned subtree} and Lemma \ref{lem: distances in a dendron} are enough to prove Theorem \ref{thm:convergence of distances critical window}, and therefore \Cref{th: dendron limit if a/n}.

\begin{proof}[Proof of Theorem \ref{thm:convergence of distances critical window}]
Fix $\varepsilon>0$. Consider $M,Q \geq 1$ such that $E_{M,Q}$ holds with probability at least $1-\varepsilon$. 
Under the event $E_{M,Q}$, there are only finitely many possible trees $\mathfrak{T}$ and finitely many allocations $(i_1, \ldots, i_p)$ such that $\probcond{ \St_M(\cD_a)=\mathfrak{T} \, ; \, w_r = (v_{i_r},1) \text{ for all } r \in [p] }{ E_{M,Q}}>0$. 
Fix any such $(\mathfrak{T},i_1,\ldots,i_p)$. 
Using the coupling of \Cref{sssec:coupling} along with Lemmas \ref{lem:proba of a structure} and \ref{lem:same component}, our result follows from Proposition \ref{prop:spanned subtree} and Lemma \ref{lem: distances in a dendron}.
\end{proof}

The last step that remains is to prove Proposition \ref{prop:spanned subtree}.

\begin{proof}[Proof of Proposition \ref{prop:spanned subtree}]
First, we claim that we only need to prove our result under the additional assumptions that
\begin{itemize}
    \item[(C1)] $(n)_Q^{k_n} \geq n/\log n$; 
    \item[(C2)] $N^{(k_n)}\big( (\cT_n)_i^{k_n} \big) = o\big( (n)_i^{k_n} \big)$ for all $i \in [Q]$;
    \item[(C3)] There is no descent edge between two vertices of labels $\leq k_n$. 
\end{itemize}
Indeed, these three events hold with high probability as $n \rightarrow \infty$, using \Cref{lem:no small descent,lem:lemmerelou5bis,lem:lemmerelou4bis}.

Observe that, conditionally on assumptions (A1) and (A2), for all $r \in [p]$ independently, $w_r^{(n)}$ is a uniform vertex of $(\cT_n)_{i_r}^{k_n}$.
Furthermore, the following holds for all $j\in[Q]$ such that $(\cT_n)_j^{k_n}$ is not the root component of $\cT_n$.
Let $\rho_j$ denote the root of $(\cT_n)_j^{k_n}$, and $x_j$ its parent in $\cT_n$.
Then, $x_j$ is a uniformly random vertex of $(\cT_n)_{p(j)}^{k_n} \setminus t_{p(j)}^{k_n}$, where $p(j) \in \cT_n^{k_n,\red}$ is such that $(\cT_n)_{p(j)}^{k_n}$ is the component containing $x_j$. 
Recall that $t_{p(j)}^{k_n}$ is the subtree of $(\cT_n)_{p(j)}^{k_n}$ consisting of labels $\le k_n$.

Hence, by (C2), with probability $1-o(1)$, $w_r^{(n)} \in T^\circ_{i_r}$, where $T^\circ_{i_r} := \big( V^{(k_n)}\big( (\cT_n)_{i_r}^{k_n} \big) \big)^c$ is the set of vertices in $(\cT_n)_{i_r}^{k_n}$ such that all descent endpoints in their ancestral line have labels $\leq k_n$. 
In addition, conditionally on that event, $w_r^{(n)}$ is distributed as a uniform vertex of $T^\circ_{i_r}$. 
Furthermore, with probability $1-o(1)$ for all valid $j\in[Q]$, the vertex $x_j$ is an element of $T^\circ_{p(j)}$ and, conditionally on this event, is a uniform vertex of $T^\circ_{p(j)} \setminus t_{p(j)}^{k_n}$. 

By (C1) and (C3), for all $i \in [Q]$, $(n)_i^{k_n} \geq n/\log n$ and $T^\circ_i$, conditionally on its size, is distributed as a random recursive tree of that size. 
We finally use the following classical result (see, e.g., \cite{Panholzer_2004}): let $T_m$ be a random recursive tree with $m$ vertices, let $u,v$ be i.i.d.\ uniform vertices in $T_m$, $d(u,v)$ their graph distance in $T_m$, and $h(u)$ the height of the vertex $u$. Then, as $m \rightarrow \infty$, $d(u,v) = 2\log(m)(1+o_{\textbf{P}}(1))$ and $h(u)=\log(m)(1+o_{\textbf{P}}(1))$.

This immediately implies that, conditionally on (A1), (A2), (C1), (C2) and (C3),
\begin{itemize}
    \item for all $r \in [p]$, $\dist_n(w_r^{(n)},x_{i_r})=\log(n) (1+o_\bP(1))$;
    \item for all $j$ such that $j$ is the label of a non-root vertex of $\mathfrak{T}$, $\dist_n(x_j, x_{p(j)})=\log(n)(1+o_\bP(1))$;
    \item for all labels $j_1,j_2$ of non-root vertices of $\mathfrak{T}$ such that $p(j_1)=p(j_2)$, $\dist_n(x_{j_1},x_{j_2})=2\log(n) (1+o_\bP(1))$.
\end{itemize}
The result follows from the tree structure of $\mathfrak{T}$.
\end{proof}

\subsection{Interpolation on the space of dendrons}
\label{sec: interpolation}

Here we prove \Cref{th: our dendron cv to CRT}, stating that the family $(\cD_a)_{a > 0}$ interpolates, on the space of dendron distributions, between Aldous' CRT (when $a \rightarrow \infty$) and the trivial dendron $\Upsilon_{\delta_1}$ (when $a \rightarrow 0$).

\begin{lemma}\label{lem: asymptotic first term and variance for PD}
    For $a>0$, let $\big( X_{a,1} \ge X_{a,2} \ge \dots \big)$ be a $\PoissonDirichlet{0}{a}$ random variable and $\sigma_a^2 := \sum_{i\ge1} X_{a,i}^2$.
    Then,
    \begin{equation*}
        a X_{a,1}^2 \overset{(d)}{\cv{a\to\infty}} 0
        \quad\text{and}\quad
        a\sigma_a^2 \overset{(d)}{\cv{a\to\infty}} 1 \,.
    \end{equation*}
\end{lemma}

\begin{proof}
    The distribution of $X_{a,1}$ is given, e.g., in \cite{Arratia_Barbour_Tavare_2003}.
    By \cite[Eq.\ (4.101)]{Arratia_Barbour_Tavare_2003}, we have
    \begin{equation*}
        \expec{X_{a,1}^2}
        = \frac{1}{a+1} \int_0^\infty x e^{-x-aE_1(x)} \mathrm dx,
        \quad\text{where}\quad
        E_1(x) := \int_x^\infty y^{-1} e^{-y} \mathrm dy \,.
    \end{equation*}
    It is straightforward, by dominated convergence, to see that $a \cdot \expec{X_{a,1}^2} \to 0$ as $a\to\infty$.
    Therefore, $a X_{a,1}^2 \to 0$ in probability as $a\to\infty$.
    For the second statement, we rely on a simple distributional equation for $\sigma_a^2$.
    Indeed, let $B_a$ denote a $\BetaDistribution{1}{a}$ random variable, independent of $\sigma_a^2$.
    Then, we have the following equality in distribution:
    \begin{equation*}
        \sigma_a^2 \overset{(d)}{=} B_a^2 + (1-B_a)^2 \sigma_a^2 \,.
    \end{equation*}
    From this we deduce
    \[
        \expec{\sigma_a^2} = \frac{2}{(1+a)(2+a)} + \frac{a}{2+a} \expec{\sigma_a^2}\,,
    \]
    and so $\expec{\sigma_a^2} = \frac{1}{1+a}$.
    Squaring the distributional equality further gives
    \[
        \expec{\sigma_a^4} = \frac{24}{(1+a)(2+a)(3+a)(4+a)}
        + \frac{4a}{(2+a)(3+a)(4+a)} \expec{\sigma_a^2} + \frac{a}{4+a} \expec{\sigma_a^4}\,,
    \]
    and so $\expec{\sigma_a^4} = \frac{6+a}{(1+a)(2+a)(3+a)}$.
    It follows that $\var{\sigma_a^2} = \frac{2a}{(1+a)^2(2+a)(3+a)} = o(1/a^2)$ and that $a\sigma_a^2 \cv{} 1$ in probability as $a\to\infty$ by the Bienaymé--Chebyshev inequality.
\end{proof}

We can now prove \Cref{th: our dendron cv to CRT}.

\begin{proof}[Proof of \Cref{th: our dendron cv to CRT}]
The proof of (i) is now immediate, as $X^{(a)} \cv{a\to0} (1,0,\ldots)$ in distribution with respect to the $\ell_\infty$ norm on $\Delta$. 
So we turn to the proof of (ii).
    For $a>0$, let $\fp_a$ be the random probability measure on $\bN$ given by
    \[
        \fp_a(i) = X_{a,i} \text{ for all } i\ge1,
        \quad\text{where}\quad
        \big( X_{a,1} \ge X_{a,2} \ge \dots \big) \sim \PoissonDirichlet{0}{a} \,.
    \]
    Then, the base tree $T_a^{\discrete}$ of our dendron $\cD_a$ is a $\fp_a$-tree in the sense of \cite{Camarri_Pitman_2000,Blanc-Renaudie_2025}.
    We wish to use \cite[Theorem~6.2]{Blanc-Renaudie_2025}, which gives a criterium for the Gromov--Prokhorov convergence of p-trees toward inhomogeneous CRTs (in our case, the regular CRT).
    See also the discussion in \cite[Section~5]{Camarri_Pitman_2000}.
    In our case, it suffices to check that
    \begin{equation*}
        X_{a,1} / \sigma_a \overset{(d)}{\cv{a\to\infty}} 0
    \end{equation*}
    where $\sigma_a = \sqrt{\sum_{i\ge1} X_{a,i}^2}$.
    This is a direct consequence of \Cref{lem: asymptotic first term and variance for PD}.
    Therefore, we get that 
    \begin{equation*}
        \left( T_a^{\discrete} , {\sigma_a} \cdot d_a , \fp_a \right)
        \overset{(d)}{\cv{a\to\infty}}
        \left( \cT_\be , \dist, \mu \right)
    \end{equation*}
    with respect to the Gromov--Prokhorov topology.
    In this convergence statement, the right-hand side is the Brownian CRT.
    The left-hand side is the discrete base tree of the dendron $\cD_a$ with distances renormalized by $\sigma_a$ and measure $\fp_a = \sum_{i\ge1} X_{a,i} \delta_{v_i}$ on its vertices.
    Since $\sigma_a \sim 1/\sqrt{a}$ in probability as $a\to\infty$ by \Cref{lem: asymptotic first term and variance for PD}, it is straightforward to see that the Prokhorov distance between the measures $\fp_a$ on ${\sigma_a} \cdot T_a^{\discrete}$ and $\nu_a$ on $\sigma_a \cdot \cD_a$ goes to $0$ in probability, and thus
    \begin{equation*}
        \left( \cD_a , \frac{1}{\sqrt a} \cdot d_a , \nu_a \right)
        \overset{(d)}{\cv{a\to\infty}}
        \left( \cT_\be , \dist, \mu \right)
    \end{equation*}
    follows.
\end{proof}

\section{The Benjamini--Schramm local limit}
\label{sec:local limit}

This section is devoted to the proof of our local limit result, \Cref{th: BS local limit unified}.
We split the proof into two separate cases: first when $q_n \to q\in(0,1]$, second when $q_n \to 0$ as $n \to \infty$.

\subsection{Local limit for fixed $q \in (0,1]$}

In this subsection we fix $q_n \equiv q \in (0,1]$. 
One can easily check that, since we only consider events involving finite trees, the proof can be written mutatis mutandis to the case $q_n \rightarrow q \in (0,1]$.
In order to properly define the random rooted tree $(\cT_*^{(q)}, \rho_*^{(q)})$ of Definition~\ref{def:t etoile q}, we introduce some notation.

For $T\in\bT_n$ and $k \in [n]$, we let $T(k)$ denote the subtree of $T$ rooted at the vertex labeled $k$.
For $j\in[n]$ labeling a vertex in $T(k)$, we write $\rk_{T(k)}(j)$, or sometimes simply $\rk(j)$, for the rank of $j$ among labels in $T(k)$.
This allows us to see $T(k) \in \bT$ as a properly labeled tree, that is, with labels ranging from $1$ to $\abs{T(k)}$.

For all $n,r,m,i \geq 1$, and all $d \geq 0$, let $\bT_{n,d} \subseteq \bT_n$ be the set of trees with $n$ vertices and $d$ descents;
let $\bU_{n,d,i}$ be the set of trees with $n$ vertices, $d$ descents and root-label $i$;
let $\bV_{n,d,r,m,i} \subseteq \bT_{n,d}$ be the set of trees $T \in \bT_{n,d}$ such that $|T(r)|=m$ and $\rk_{T(r)}(r)=i$.
%Denote by $v_r$ the vertex labeled $r$.
Finally, for all $m \geq i \geq 1$, define 
\begin{equation}\label{eq: function fmi}
    f_{m,i}^{(q)}: \chi \in[0,1] \mapsto \left(\chi+q(1-\chi)\right) x_0(q)^m \frac{\chi^{i-1}}{(i-1)!}\frac{(1-\chi)^{m-i}}{(m-i)!} \sum\limits_{d \geq 0} \abs{\bU_{m,d,i}}q^d \,.
\end{equation}

Recall the definitions of the variables $Y_{q,\chi}$ in Definition \ref{def:grafted trees} and the Markov chain 
$(X_k)_{k \geq 0}$ in Definition \ref{def:t etoile q}: the former is given by
$\prob{Y_{q,\chi}=(m,i)}=f_{m,i}^{(q)}(\chi)$ for all $m \geq i \geq 1$. The latter is started at $X_0 \sim \Unif{[0,1]}$ and has the kernel
\[
        \kappa_q(\mathrm dy|x)=\frac{1}{x+q(1-x)} \mathbbm{1}_{y \in [0,x]} \mathrm dy + \frac{q}{x+q(1-x)} \mathbbm{1}_{y \in [x,1]} \mathrm dy \,.
    \]

We will see in \Cref{cor:sums at 1} that, for any given $\chi\in[0,1]$, the function $(m,i) \mapsto f_{m,i}^{(q)}(\chi)$ defines a probability mass function, justifying that \Cref{def:grafted trees} is valid.

Our aim here is to prove \Cref{th: BS local limit unified} for $q_n \equiv q \in (0,1]$.
We start by stochastically bounding the sizes of the subtrees in $\cT_n^{(q)}$.
In what follows, we often write $v_r$ for the vertex labeled $r$.

\begin{proposition}
\label{prop:stochastic bound for size subtree}
For any given $q\in(0,1]$, we have
\begin{align*}
    \limsup_{n \rightarrow \infty} \sup_{1 \leq k \leq n} \prob{|\cT_n^{(q)}(k)| \geq M} \cv{M\to\infty} 0. 
\end{align*}
\end{proposition}

This is an immediate consequence of the following lemma.

\begin{lemma}
    \begin{itemize}
    \item[(i)] For all $q \in (0,1]$,
        \begin{align*}
            \limsup_{n \rightarrow \infty} \prob{|\cT_n^{(q)}(1)| \geq M} \cv{M\to\infty} 0.
        \end{align*}
    \item[(ii)] For all $n \geq 1$, all $q \in (0,1]$, and all $k \leq n-1$, we have $\prob{|\cT_n^{(q)}(k)| \geq M} \geq \prob{|\cT_n^{(q)}(k+1)| \geq M} .$
    \end{itemize}
\end{lemma}

\begin{proof}
    Let us prove (i). Define, for $d \geq 0$ and $n > M \geq 1$, 
    \begin{align*}
        B_M(x,q) := \sum_{n \geq 1, d \geq 0}\frac{
        %|f(d,n,M)|
        \abs{\bV_{n,d,1,M,1}}
        }{n!} x^n q^d,
    \end{align*}
    where we recall that $\bV_{n,d,1,M,1}$ is the set of trees $T$ with $d$ descents, $n$ vertices and such that $|T(1)|=M$ (counting the vertex $v_1$ labeled $1$ itself). 
    When $n>M$, there is necessarily a descent leading to $v_1$ in such trees, and thus $v_1$ is not the root of the tree. 
    We immediately get that, for any fixed $n>M \geq 1$,
    \begin{align*}
    [x^n] B_M(x,q) = q \frac{n-M}{n} [x^{n-M}] A(x,q) [x^{M-1}] e^{A(x,q)},
    \end{align*}
    so that, for fixed $M\ge1$ and as $n\to\infty$,
    \begin{align*}
        \prob{|\cT_n^{(q)}(1)|=M} 
        = q \frac{n-M}{n} [x^{M-1}] e^{A(x,q)} \frac{[x^{n-M}]A(x,q)}{[x^n] A(x,q)}
        \sim q x_0^M [x^{M-1}] e^{A(x,q)} %\text{ by e.g. \cite[Eq.(15)]{Thevenin_Wagner_2023} }
    \end{align*}
    by \eqref{eq: asymptotic coefficient A for fixed q}, where $x_0:=q^{\frac{q}{1-q}}$ is the dominant singularity of $A(x,q)$. 
    Using the fact that 
    \begin{align*}
        \sum_{M \geq 1} q x_0^M [x^{M-1}]e^{A(x,q)}=qx_0e^{A(x_0,q)}=1
    \end{align*}
    by \eqref{eq: asymptotic A close to singularity for fixed q}, we get the result.
    
     We now prove (ii). Fix $k \leq n-1$. 
     Let $E_1 \subset \bT_n$ be the set of trees where $v_k$ and $v_{k+1}$ are not connected by an edge, $E_2$ the set of trees where $v_k$ is a child of $v_{k+1}$ and $E_3$ the set of trees where $v_{k+1}$ is a child of $v_k$.
     We consider a random transformation $T \mapsto \hat{T}$ on $\bT_n$. 
     For $T \in \bT_n$, define $T^{k \leftrightarrow k+1}$ as the tree in which the labels of the vertices labeled $k$ and $k+1$ have been swapped (leaving the rest of the tree and the tree structure unchanged). 
     Now, define $\hat{T}$ as follows:
     \begin{itemize}
         \item If $T \in E_1$, then $\hat{T}$ is set to $T$ with probability $\frac12$, and it is set to $T^{k \leftrightarrow k+1}$ with probability $\frac12$;
         \item If $T \in E_2$, then $\prob{\hat{T}=T}=\frac{q}{q+1}$ and $\prob{\hat{T}=T^{k \leftrightarrow k+1}}=\frac{1}{q+1}$;
         \item If $T \in E_3$, then $\prob{\hat{T}=T}=\frac{1}{q+1}$ and $\prob{\hat{T}=T^{k \leftrightarrow k+1}}=\frac{q}{q+1}$.
     \end{itemize}
     The map $T \mapsto \hat{T}$ preserves the distribution of $\cT_n := \cT_n^{(q)}$. 
     Indeed, we have the following cases:
     \begin{itemize}
         \item If $T \in E_1$, then $\prob{\hat{\cT_n}=T}=\frac{1}{2} \prob{\cT_n=T}+\frac{1}{2} \prob{\cT_n=T^{k \leftrightarrow k+1}} = \prob{\cT_n=T}$ since $T$ has as many descents as $T^{k \leftrightarrow k+1}$.
         \item If $T \in E_2$, then $\prob{\hat{\cT_n}=T}=\frac{q}{q+1} \prob{\cT_n=T}+\frac{q}{q+1} \prob{\cT_n=T^{k \leftrightarrow k+1}} = \prob{\cT_n=T}$ since $T$ has one more descent than $T^{k \leftrightarrow k+1}$.
         \item If $T \in E_3$, we get analogously that $\prob{\hat{\cT_n}=T}=\prob{\cT_n=T}$.
     \end{itemize}
Now fix $M \geq 1$. We have
\begin{align*}
&\prob{\abs{\cT_n(k+1)} \geq M > \abs{\cT_n(k)}, \cT_n \in E_1} 
\\&= \prob{\abs{\hat{\cT_n}(k+1)} \geq M > \abs{\hat{\cT_n}(k)}, \hat{\cT_n} \in E_1}\\
&= \frac12 \prob{\abs{\cT_n(k+1)} \geq M > \abs{\cT_n(k)}, \cT_n \in E_1} + \frac12 \prob{\abs{\cT_n(k)} \geq M > \abs{\cT_n(k+1)}, \cT_n \in E_1},
\end{align*}
so that $\prob{\abs{\cT_n(k+1)} \geq M > \abs{\cT_n(k)}, \cT_n \in E_1}=\prob{\abs{\cT_n(k)} \geq M > \abs{\cT_n(k+1)}, \cT_n \in E_1}$.

On the other hand,
\begin{align*}
&\prob{\abs{\cT_n(k+1)} \geq M > \abs{\cT_n(k)}, \cT_n \in E_2} 
\\&= \prob{\abs{\hat{\cT_n}(k+1)} \geq M > \abs{\hat{\cT_n}(k)}, \hat{\cT_n} \in E_2}\\
&= \frac{q}{q+1} \left[\Big. \prob{\abs{\cT_n(k+1)} \geq M > \abs{\cT_n(k)}, \cT_n \in E_2} + \prob{\abs{\cT_n(k)} \geq M > \abs{\cT_n(k+1)}, \cT_n \in E_3} \right] \,,
\end{align*}
so that $\prob{\abs{\cT_n(k+1)} \geq M > \abs{\cT_n(k)}, \cT_n \in E_2} = q \prob{\abs{\cT_n(k)} \geq M > \abs{\cT_n(k+1)}, \cT_n \in E_3}$.

Finally, we have $\prob{\abs{\cT_n(k+1)} \geq M > \abs{\cT_n(k)}, \cT_n \in E_3} = \prob{\abs{\cT_n(k)} \geq M > \abs{\cT_n(k+1)}, \cT_n \in E_2}=0$.
Using the fact that 
\[ 
    \prob{\abs{\cT_n(k+1)} \geq M} - \prob{\abs{\cT_n(k)} \geq M} = \prob{\abs{\cT_n(k+1)} \geq M > \abs{\cT_n(k)}} - \prob{\abs{\cT_n(k)} \geq M > \abs{\cT_n(k+1)}}
\]
together with the previous inequalities, we get the result.
\end{proof}

In what follows, we let ${\rm par}(r)$ denote the label of the parent of $v_r$.

\begin{lemma}
\label{lem:distribution of the small tree}
Fix $i,m$ with $m \geq i \geq 1$, and any function $g: \bN \rightarrow \bN$ such that $g(n) \rightarrow \infty$ as $n \rightarrow \infty$. As $n \rightarrow \infty$, uniformly for all $r \in \llbracket g(n), n-g(n) \rrbracket$, we have
\begin{align*}
    \prob{\abs{\cT_n(r)}=m, \rk(r)=i, {\rm par}(r) < r} \underset{n \rightarrow \infty}{\sim} \frac{r/n}{r/n+q(1-r/n)} f_{m,i}^{(q)}\left(\frac{r}{n}\right),
\end{align*}
and 
\begin{align*}
    \prob{\abs{\cT_n(r)}=m, \rk(r)=i, {\rm par}(r) > r} \underset{n \rightarrow \infty}{\sim} \frac{q(1-r/n)}{r/n+q(1-r/n)}f_{m,i}^{(q)}\left(\frac{r}{n}\right).
\end{align*}
%Here, recall that $x_0 := q^{\frac{q}{1-q}}$ is the radius of convergence of $q \mapsto A(x,q)$.
\end{lemma}

Hence, we get the following corollary:

\begin{corollary}
\label{cor:sums at 1}
Fix $i,m$ with $m \geq i \geq 1$, and any function $g: \bN \rightarrow \bN$ such that $g(n) \rightarrow \infty$ as $n \rightarrow \infty$. As $n \rightarrow \infty$, uniformly for all $r \in \llbracket g(n), n-g(n) \rrbracket$, we have
\begin{itemize}
    \item[(i)]  $\prob{\abs{\cT_n(r)}=m, \rk(r)=i} \underset{n \rightarrow \infty}{\sim} f_{m,i}^{(q)}\left(\frac{r}{n}\right)$;
    \item[(ii)] $\lim\limits_{M \rightarrow \infty}\sum\limits_{m=1}^M\sum\limits_{i=1}^m f_{m,i}^{(q)}\left( \frac{r}{n} \right) = 1 + o(1)$.
    \item[(iii)] For any given $\chi\in[0,1]$, $\sum_{m\ge i\ge 1} f_{m,i}^{(q)}(\chi) = 1$.
\end{itemize}
\end{corollary}

\begin{proof}[Proof of Corollary \ref{cor:sums at 1}]
This follows immediately from Proposition \ref{prop:stochastic bound for size subtree} and Lemma \ref{lem:distribution of the small tree}.    
\end{proof}

Let us now prove Lemma \ref{lem:distribution of the small tree}.

\begin{proof}[Proof of Lemma \ref{lem:distribution of the small tree}]
Let $\bV^>_{n,d,r,m,i}$ (resp., $\bV^<_{n,d,r,m,i}$) denote the subset of $\bV_{n,d,r,m,i}$ of trees in which ${\rm par}(r)>r$ (resp., ${\rm par}(r)<r$).
We have, for all $n,r,m$,
\begin{align*}
\sum_{d \geq 0} |\bV^<_{n,d,r,m,i}| q^d = (r-i)  \binom{r-1}{i-1} \binom{n-r}{m-i} \sum_{d' \geq 0} |\bT_{n-m,d'}| q^{d'} \sum_{d_0 \geq 0} |\bU_{m,d_0,i}|q^{d_0} \,.
\end{align*}
Indeed, the first factor comes from the choice of the parent of $v_r$, and the binomial terms account for the choice of the labels of the vertices in the subtree grafted at $v_r$. 
In particular, for fixed $i$ and $m$, and uniformly for $g(n) \leq r \leq n-g(n)$, we have
\begin{align*}
\sum_{d \geq 0} \frac{|\bV^<_{n,d,r,m,i}|}{n!} q^d &= (1+o(1))\frac{(n-m)!}{n!} r \frac{r^{i-1}}{(i-1)!} \frac{(n-r)^{m-i}}{(m-i)!} \sum_{d' \geq 0} \frac{|\bT_{n-m,d'}|}{(n-m)!} q^{d'} \sum_{d_0 \geq 0} |\bU_{m,d_0,i}|q^{d_0} \\
&= (1+o(1)) \frac{r}{n}  \frac{\left( \frac{r}{n} \right)^{i-1}}{(i-1)!} \frac{\left( 1 - \frac{r}{n} \right)^{m-i}}{(m-i)!} \sum_{d' \geq 0} \frac{|\bT_{n-m,d'}|}{(n-m)!} q^{d'} \sum_{d_0 \geq 0} |\bU_{m,d_0,i}|q^{d_0}.
\end{align*}
This entails, using \eqref{eq: asymptotic coefficient A for fixed q} for the last step,
\begin{align*}
\prob{\abs{\cT_n(r)}=m, \rk(r)=i, {\rm par}(r)<r} &=\prob{\cT_n \in \bigcup_{d \geq 0}\bV^<_{n,d,r,m,i}}\\
&= (1+o(1)) \frac{r}{n} \frac{\left( \frac{r}{n} \right)^{i-1}}{(i-1)!} \frac{\left( 1 - \frac{r}{n} \right)^{m-i}}{(m-i)!} \sum_{d_0 \geq 0} |\bU_{m,d_0,i}|q^{d_0} \frac{[x^{n-m}] A(x,q)}{[x^n]A(x,q)}\\
&=(1+o(1)) x_0^m\frac{r}{n} \frac{\left( \frac{r}{n} \right)^{i-1}}{(i-1)!} \frac{\left( 1 - \frac{r}{n} \right)^{m-i}}{(m-i)!} \sum_{d_0 \geq 0} |\bU_{m,d_0,i}|q^{d_0}.
\end{align*}
The second part of Lemma \ref{lem:distribution of the small tree} is obtained in an analogous way.
\end{proof}

Finally, consider a tree $T$ with $n$ vertices and $r_0 \in [n]$. 
We define the $p$-ancestral line of $v_{r_0}$, denoted by $\AncestralLine_p(r_0)$, as the sequence $((r_j,m_j,i_j))_{0 \leq j \leq p}$ that is given as follows.
For $1 \leq j \leq p$, $v_{r_j}$ is the parent of $v_{r_{j-1}}$, $T'(r_j)=T(r_j) \backslash T(r_{j-1})$, $m_j = \abs{T'(r_j)}$ and $i_j = \rk'(r_j)=\rk_{T'(r_j)}(r_j)$. Moreover, we define $T'(r_0)=T(r_0)$, $m_0 = \abs{T'(r_0)}$ and $i_0 = \rk'(r_0)=\rk_{T(r_0)}(r_0)$.
If $j$ is greater than the depth of $v_{r_0}$, we set $(r_j,m_j,i_j) := (0,0,0)$ by convention.
See \Cref{fig: ancestral line} for an illustration.

\begin{figure}
    \centering
    \includegraphics[width=.4\linewidth]{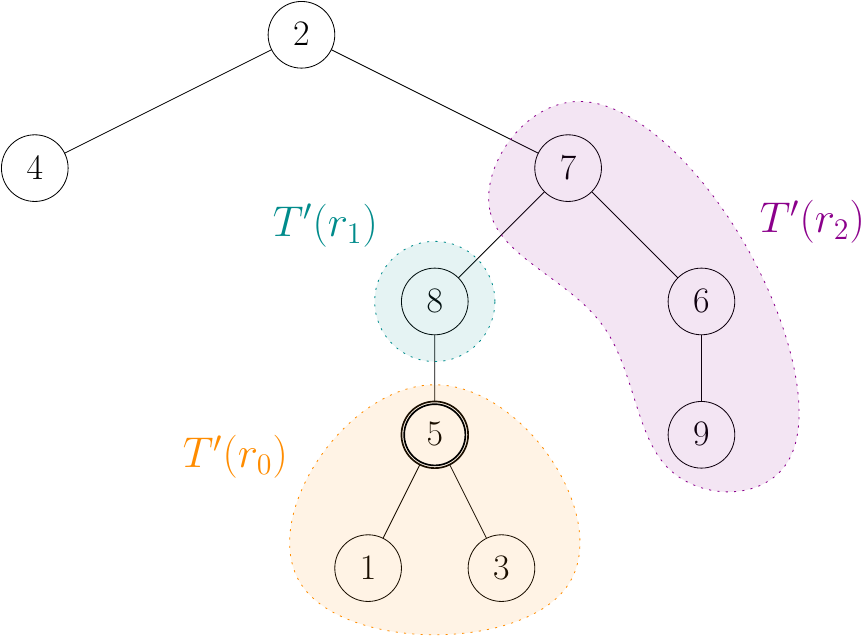}
    \caption{A tree $T\in\bT_9$, where the subtrees involved in the $2$-ancestral line of $v_5$ are highlighted.
    %$T'(r_0)$, $T'(r_1)$, $T'(r_2)$ are respectively in orange, cyan, magenta.
    We get $\AncestralLine_2(5) = ((5,3,3), (8,1,1), (7,3,2))$.
    For $p \ge3$, the next terms of the ancestral line would be $(2,2,1)$, and then all $(0,0,0)$'s.
    }
    \label{fig: ancestral line}
\end{figure}

\begin{proposition}
\label{prop:ancestral line local limit fixed q}
Fix $p \geq 0$, and fix $p+1$ pairs of integers $(m_j,i_j) \in \bN^2$ such that $m_j \geq i_j \geq 1$ for $0 \leq j \leq p$. 
Then, as $n \rightarrow \infty$, uniformly for all distinct $r_0,\ldots,r_p \in \llbracket 1,n \rrbracket$ such that $\sqrt{n} \leq r_j \leq n-\sqrt{n}$ for all $0 \leq j \leq p$, 
we have
\begin{align*}
\prob{\AncestralLine_p(r_0) = ((r_j,m_j,i_j))_{0 \leq j \leq p}} 
\:\underset{n \rightarrow\infty}{\sim}\:
\prod_{j=0}^p f_{m_j,i_j}^{(q)}\left( \frac{r_j}{n} \right) \prod_{j=0}^{p-1} \frac{\One{r_{j+1} < r_j} + q \One{r_{j+1} > r_j}}{r_j + q(n-r_j)}.
\end{align*}
\end{proposition}

In particular, summing over all possible values of $((r_j,m_j,i_j))_{0 \leq j \leq p}$ and using Corollary \ref{cor:sums at 1} (ii) shows that the size of the tree $T_p$ rooted at the $p$-th ancestor of a uniform vertex is stochastically bounded, in the sense that
\begin{align*}
\limsup_{n \rightarrow \infty}\prob{|T_p| \geq K} \cv{K\to\infty} 0 \,.
\end{align*}

\begin{proof}[Proof of Proposition \ref{prop:ancestral line local limit fixed q}]
Consider $\ba := ((r_j,m_j,i_j))_{0 \leq j \leq p}$ and let $M_p := \sum_{j=0}^p m_j$.
Let $\bV_{n,d,\ba}$ denote the set of trees with $n$ labeled vertices and $d$ descents such that $\AncestralLine_p(r_0)=\ba$. We have, in the same way as in the proof of Lemma~\ref{lem:distribution of the small tree},
\begin{multline*}
\sum_{d \geq 0} \abs{\bV_{n,d,\ba}} q^d 
\\= (1+o(1)) \left[ r_p+q(n-r_p) \right] 
\left( \prod_{j=0}^p \frac{r_j^{i_j-1}}{(i_j-1)!} \frac{(n-r_j)^{m_j-i_j}}{(m_j-i_j)!} \sum_{d_j \geq 0} \abs{\bU_{m_j,d_j,i_j}}q^{d_j} \right)
\sum_{d' \geq 0} |\bT_{n-M_p,d'}| q^{d'}  q^{d(\sigma_p)},
\end{multline*}
where $d(\sigma_p):= \abs{\{ j \in [p], r_j<r_{j-1}\}}$ is the number of descents in the permutation induced by $\{r_p,r_{p-1},\ldots,r_0\}$.
This can be rewritten as follows:
\begin{align*}
&\sum_{d \geq 0} \frac{\abs{\bV_{n,d,\ba}}}{n!} q^d 
\\&= (1+o(1)) n^{-p} \left[ \frac{r_p}{n} + q \left( 1-\frac{r_p}{n} \right) \right] 
\left( \prod_{j=0}^p \frac{\left(\frac{r_j}{n}\right)^{i_j-1}}{(i_j-1)!} \frac{\left(1-\frac{r_j}{n}\right)^{m_j-i_j}}{(m_j-i_j)!} \sum_{d_j \geq 0} \abs{\bU_{m_j,d_j,i_j}}q^{d_j} \right)
[x^{n-M_p}]A(x,q) q^{d(\sigma_p)}
\\&= (1+o(1)) n^{-p} 
\left( \prod_{j=0}^p x_0^{-m_j} f_{m_j,i_j}^{(q)}\left( \frac{r_j}{n} \right) \right)
\left( \prod_{j=0}^{p-1} \left( \frac{r_j}{n} + q\left( 1-\frac{r_j}{n}\right) \right)^{-1} \right)
[x^{n-M_p}]A(x,q) q^{d(\sigma_p)} \,.
\end{align*}
Therefore, we get
\begin{align*}
\prob{\big. \AncestralLine_p(r_0) = ((r_j,m_j,i_j))_{0 \leq j \leq p}} &= \frac{1}{[x^n]A(x,q)} \sum_{d \geq 0} \frac{\abs{\bV_{n,d,\ba}}}{n!} q^d\\
&= (1+o(1)) n^{-p} \prod_{j=0}^p f_{m_j,i_j}^{(q)}\left( \frac{r_j}{n} \right) \prod_{j=0}^{p-1} \left( \frac{r_j}{n} + q\left( 1-\frac{r_j}{n}\right) \right)^{-1} q^{d(\sigma_p)}\\
&= (1+o(1)) \prod_{j=0}^p f_{m_j,i_j}^{(q)}\left( \frac{r_j}{n} \right) \prod_{j=0}^{p-1} \frac{\One{r_{j+1} < r_j} + q \One{r_{j+1} > r_j}}{r_j + q(n-r_j)}
\end{align*}
as announced.
\end{proof}

We can finally prove \Cref{th: BS local limit unified} when $q_n \equiv q\in(0,1]$ is constant.
\medskip

\begin{proof}[Proof of \Cref{th: BS local limit unified}, case $q_n \equiv q > 0$]
Fix $p \geq 0$. 
Let $R_0 := R^{(n)}_0$ be a uniformly random integer in $[n]$, and let $\AncestralLine_p(R_0) := (R_j,M_j,I_j)_{0 \leq j \leq p}$ be its $p$-ancestral line (with $(R_j,M_j,I_j)=(0,0,0)$ if $h(v_{R_0})<j$). 
By our \Cref{prop: Rayleigh limit of typical depth when q fixed} (or, alternatively, by \cite[Proposition 24]{Thevenin_Wagner_2023} and Markov's inequality), we know that $\prob{h(v_{R_0}) \leq p} = o(1)$, where $h(v_{R_0})$ denotes the distance in $\cT_n$ between the root and the vertex labeled $R_0$.

Let $F: ([0,1] \times \bN \times \bN)^{p+1} \rightarrow \bR$ be bounded and continuous. 
By \Cref{prop:ancestral line local limit fixed q} and dominated convergence, we have
\begin{align*}
\bE&\left[ F\left(\left(\frac{R_j}{n},M_j,I_j\right)_{0 \leq j \leq p}\right) \right]\\ 
& = o(1)+\sum_{(r_j,m_j,i_j) \in ([n] \times \bN^2)^{p+1}} F\big( (r_j/n,m_j,i_j)_{0 \leq j \leq p} \big) \prob{R_0=r_0} \prob{\AncestralLine_p(r_0) := ((r_j,m_j,i_j))_{0 \leq j \leq p}}\\
& = o(1) + \frac{1}{n}\sum_{(r_j,m_j,i_j) \in ([n]\times \bN^2)^{p+1}} F\big( (r_j/n,m_j,i_j)_{0 \leq j \leq p} \big)\prod_{j=0}^p f_{m_j,i_j}^{(q)}\left( \frac{r_j}{n} \right) \prod_{j=0}^{p-1} \frac{\One{r_{j+1} < r_j} + q \One{r_{j+1} > r_j}}{r_j + q(n-r_j)}\\
& \cv{n\to\infty} \int_{[0,1]^{p+1}} \sum_{(m_j,i_j) \in \bN^{2p+2}} F(x_j,m_j,i_j) \prod_{j=0}^p f_{m_j,i_j}^{(q)}(x_j) \prod_{j=0}^{p-1} \frac{\One{x_{j+1}<x_j}+q\One{x_{j+1}>x_j}}{x_j+q(1-x_j)} \prod_{j=0}^p \mathrm dx_j\\
&=\bE\left[F\left( \left(X_j,\cM_j,\cI_j\right)_{0 \leq j \leq p} \right)\right]
\end{align*}
In the last step, $(X_j)_{0 \leq j \leq p}$ is distributed like the Markov chain of Definition \ref{def:t etoile q} and $(\cM_j,\cI_j)_{0 \leq j \leq p}$ are independent pairs conditionally on $(X_j)_{0 \leq j \leq p}$, with $(\cM_j,\cI_j)$ being distributed as $Y_{q,X_j}$ (see Definition \ref{def:grafted trees}). 
Finally, observe that clearly, for all $n$, for all $r \in [n]$, and all $(M,i) \in \bN^2$, conditionally on $|\cT_n(r)|=M$ and $\rk(r)=i$, the tree $\cT_n(r)$ is distributed as $\cT_{M,i}^{(q)}$. 
This completes our proof.
\end{proof}

\subsection{Local limit for $q_n \rightarrow 0$}

When $q_n \rightarrow 0$, it turns out that the Benjamini--Schramm limit of $\cT_n^{(q_n)}$ is the same as when $q_n=0$ for all $n$, that is, when $\cT_n^{(q_n)}$ is a random recursive tree of size $n$.

The local limit $\cT_*^\rec$ (in the sense of Definition \ref{def:bs local limit}) of random recursive trees $\cT_n^{(0)}$ was characterized by Aldous \cite[Section 4]{Aldous_1991} (see also \cite{Holmgren_Janson_2017}, \cite[Section 2.2]{Contat_Laulin_2025}). 
The description that we give now, which is more suited to our result, is due to Dadedzi and the third author \cite[Appendix A]{Dadedzi_Wagner_2024}. We will indeed show that the tree of Definition \ref{def:t etoile q} when $q=0$ has the same distribution as the one in \cite{Dadedzi_Wagner_2024}.

\begin{definition}[The local limit $\cT_*^\rec$, \cite{Dadedzi_Wagner_2024}]
\label{def:local limit dadedzi wagner}
The local limit $(\cT_*^\rec, \rho_*^\rec)$ of a random recursive tree can be constructed as follows:
\begin{itemize}
    \item Start with an infinite one-ended path of vertices $v_0,v_1,v_2,\ldots$.
    \item Define positive integer random variables $Z_0,Z_1,Z_2,\ldots$ such that $\prob{Z_0=k}=\frac{1}{k(k+1)}$ for $k \geq 1$ and, for all $r \geq 1$, given $Z_0,Z_1,\ldots,Z_{r-1}$, for all $k \geq 1$,
    \begin{align*}
        \prob{Z_r=k} =\frac{Y_{r-1}+1}{(k+Y_{r-1})(k+Y_{r-1}+1)},
    \end{align*}
    where $Y_{r-1} := Z_0+\ldots+Z_{r-1}$.
    \item For every $i \geq 0$, graft a random recursive tree onto $v_i$ that has $Z_i$ vertices ($v_i$ included). All these trees are taken to be independent.
    Finally, set $\rho_*^\rec = v_0$.
\end{itemize}
\end{definition}

In order to prove \Cref{th: BS local limit unified} when $q_n\to0$, we show the following.

\begin{proposition}
\label{prop:local limit when q_n goes to 0}
Let $(q_n)_{n \geq 1}$ be a sequence of elements of $[0,1]$ such that $q_n \rightarrow 0$ as $n \rightarrow \infty$. 
Then, in the sense of Definition \ref{def:bs local limit}, we have the following convergence in distribution:
\begin{align*}
    \cT_n^{(q_n)} \overset{(d)}{\cv{n\to\infty}} (\cT_*^\rec , \rho_*^\rec) \,.
\end{align*}
\end{proposition}

In other words, the local limit of $\cT_n^{(q_n)}$ is the same as the local limit of $\cT_n^{(0)}$, as soon as $q_n \rightarrow 0$.
Before proving Proposition~\ref{prop:local limit when q_n goes to 0}, we show that Definition~\ref{def:local limit dadedzi wagner} agrees with Definition~\ref{def:t etoile q} when $q_*=0$.

\begin{lemma}
\label{lem:definitions of local limits agree}
The tree $(\cT_*^{(0)},\rho_*^{(0)})$ obtained from Definition~\ref{def:t etoile q} with $q_*=0$ has the same distribution as the local limit $(\cT_*^\red, \rho_*^\rec)$ of random recursive trees from Definition~\ref{def:local limit dadedzi wagner}.
\end{lemma}

Clearly, \Cref{th: BS local limit unified} for $q_n \rightarrow 0$ is a consequence of \Cref{prop:local limit when q_n goes to 0} and \Cref{lem:definitions of local limits agree}.

\begin{proof}[Proof of Lemma \ref{lem:definitions of local limits agree}]
Recall from Remark~\ref{rk:kernels} that the tree in Definition~\ref{def:t etoile q} is obtained as in Definition~\ref{def:local limit dadedzi wagner}, by grafting independent size-conditioned random recursive trees onto $v_0,v_1,\ldots$. 
Their sizes are determined as follows:
\begin{itemize}
    \item let $X_0 \sim \Unif{[0,1]}$;
    \item For $r \geq 1$, conditionally on $X_{r-1}$, let $X_r \sim \Unif{[0,X_{r-1}]}$;
    \item For $i \geq 0$, conditionally on $(X_0,X_1,\ldots,X_i)$, the tree grafted at $v_i$ has size $G_i \sim \mathrm{Geom}(X_i)$.
\end{itemize}
Hence, we only need to prove that, for any $r \geq 0$,
\begin{equation}
(G_0,\ldots,G_r) \sim (Z_0,\ldots,Z_r),
\end{equation}
where the sequence $(Z_i)_{i \geq 0}$ is as in Definition~\ref{def:local limit dadedzi wagner}.
To this end, fix $k_0, \ldots, k_r \in \mathbb{N}_0$. It is proved in \cite[Proof of Theorem A.1]{Dadedzi_Wagner_2024} that
\begin{align*}
    \prob{(Z_0,\ldots,Z_r)=(k_0,\ldots,k_r)} = \frac{1}{K_0 K_1 \cdots K_r(K_r+1)},
\end{align*}
where $K_j := \sum_{i=0}^j k_i$ for $0 \leq j \leq r$.
Let us show that this agrees with $\prob{E_r}$, where $E_r$ is the event $\{(G_0,\ldots,G_r)=(k_0,\ldots,k_r)\}$. 
By definition, we can write
\begin{align*}
\prob{E_r}&= \int_0^1 \mathrm dx_0 \prob{\mathrm{Geom}(x_0)=k_0} 
\int_0^{x_0} \frac{\mathrm dx_1}{x_0} \prob{\mathrm{Geom}(x_1)=k_1}
\ldots
\int_0^{x_{r-1}} \frac{\mathrm dx_r}{x_{r-1}} \prob{\mathrm{Geom}(x_r)=k_r} \\
&= \int_0^1 \mathrm dx_0 (1-x_0)^{k_0-1} x_0 
\int_0^{x_0} \mathrm dx_1 \frac{(1-x_1)^{k_1-1} x_1}{x_0} 
\ldots
\int_0^{x_{r-1}} \mathrm dx_r \frac{(1-x_r)^{k_r-1} x_r}{x_{r-1}} \\
&= \int_0^1 \mathrm dx_0 (1-x_0)^{k_0-1} 
\int_0^{x_0} \mathrm dx_1 (1-x_1)^{k_1-1} 
\ldots 
\int_0^{x_{r-1}} \mathrm dx_r (1-x_r)^{k_r-1} x_r \\
&= \int_0^1 \mathrm dx_r\, x_r(1-x_r)^{k_r-1}  
\int_{x_r}^1 \mathrm dx_{r-1} (1-x_{r-1})^{k_{r-1}-1}  
\ldots 
\int_{x_1}^1 \mathrm dx_0 (1-x_0)^{k_0-1} \,.
\end{align*}
The last integral is equal to $\int_0^{1-x_1} y_0^{k_0-1} \mathrm dy_0 = \frac{(1-x_1)^{k_0}}{k_0}$,
so we have
\begin{align*}
\prob{E_r}
&=
\frac{1}{k_0} \int_0^1 \mathrm dx_r\, x_r(1-x_r)^{k_r-1} 
\ldots 
\int_{x_2}^1 \mathrm dx_1 (1-x_1)^{k_1+k_0-1} \,.
\end{align*}
By immediate induction, we get
\begin{align*}
\prob{G_0=k_0,\ldots,G_r=k_r}
&= \frac{1}{k_0(k_0+k_1) \cdots (k_0+\ldots+k_{r-1})} \int_0^1 x_r(1-x_r)^{\sum_{i=0}^r k_i-1} \mathrm dx_r \\
&= \frac{1}{K_0 K_1 \ldots K_r (K_r+1)}.
\end{align*}
This completes the proof, as the distribution of $(G_0,\ldots,G_r)$ characterizes the tree $\cT_*^{(0)}$.
\end{proof}

We can now prove Proposition \ref{prop:local limit when q_n goes to 0}. Recall that $Z_{n,q} := n! \cdot [x^n] A(x,q) = \sum_{t\in \bT_n} q^\des{t}$.

\begin{lemma}\label{lem: ratio of partition functions when q to 0}
    Let $(q_n)_{n\ge1}$ be a sequence of parameters in $[0,1]$ such that $q_n\to0$.
    Then
    \[
    \frac{Z_{n-1,q_n}}{Z_{n,q_n}} \sim \frac1n 
    \]
    as $n\to\infty$.
\end{lemma}

\begin{proof}
Recall the explicit formula $Z_{n,q} = \prod_{k=1}^{n-1} (n-k+kq)$ from \cite{Thevenin_Wagner_2023} (obtained by specializing the results of \cite{ER86}). 
Thus, for $n \geq 3$,
\begin{align*}
    \frac{Z_{n-1,q}}{Z_{n,q}} &= \frac{\prod_{k=1}^{n-2}(n-1-k+kq)}{\prod_{k=1}^{n-1}(n-k+kq)}\\
    &= \frac{\prod_{k=2}^{n-1}(n-k+(k-1)q)}{\prod_{k=1}^{n-1}(n-k+kq)}\\
    &= \frac{1}{n-1} \prod_{k=1}^{n-1} \frac{n-k+(k-1)q}{n-k+kq}\\
    &= \frac{1}{n-1} \prod_{k=1}^{n-1}\left(1- \frac{q}{n-k+kq}\right) \,.
\end{align*}
Hence we only need to prove that the product above converges to $1$ as $n\to\infty$ and $q_n \to 0$. 
To do so, observe that, for $1 \leq k \leq n-1$,
\begin{align*}
\log \left( 1- \frac{q_n}{n-k+kq_n} \right) =  -\frac{q_n}{n-k+kq_n}(1+o(1)),
\end{align*}
where the $o$ is uniform over all $k$. 
Therefore, we only need to prove that $S_n(q_n) \rightarrow 0$, where $q_n > 0$ and
\begin{align*}
    S_n(q_n) := q_n \sum_{k=1}^{n-1} \frac{1}{k+(n-k)q_n}.
\end{align*}
To do so, consider $K_n := \lceil n q_n \rceil$, and split the sum at $K_n$. 
Using the fact that $K_n \leq n/2$ if $n$ is large enough, we get
\begin{align*}
   q_n \sum_{k=1}^{K_n} \frac{1}{k+(n-k)q_n} 
   \leq q_n \sum_{k=1}^{K_n} \frac{1}{(n-k)q_n}
   \leq q_n \sum_{k=1}^{K_n} \frac{2}{nq_n}
   = \frac{2K_n}{n}.
\end{align*}
Since $K_n =o(n)$ by definition, this goes to $0$. 
On the other hand, we have
\begin{align*}
    q_n \sum_{k=K_n+1}^{n-1} \frac{1}{k+(n-k)q_n} 
    \leq q_n \sum_{k=K_n+1}^{n-1} \frac{1}{k}
    \leq q_n \log\left( \frac{n}{K_n} \right)
    \leq q_n \log \left( \frac{1}{q_n} \right).
\end{align*}
This also goes to $0$ as $n \rightarrow \infty$, and the result follows.
\end{proof}

If $v$ is a vertex in $\cT_n$, we let $v^r$ be its $(r-1)$st ancestor, i.e.~$v^1:=v$, and $v^{r+1}$ is the parent of $v^r$ for each $r\ge1$.
If $v^r$ is the root of $\cT_n$, then so is $v^{r+1}$ by convention.
Let $T_v^r$ be the subtree of $\cT_n$ rooted at $v^r$ that consists of $v^r$ and all its descendants.
Define $y_v^r := \abs{T_v^r}$, as well as $x_v^1 := y_v^1$ and $x_v^{r+1} := y_v^{r+1} - y_v^r$ for each $r\ge1$.
The ancestral line on $v$ can therefore be described by a spine with vertices $v^r$ for $r\ge1$, to which subtrees of size $x_v^r$ are grafted.

We say that a subtree $T$ of $\cT_n$ \enquote{creates no descent} if there is no descent in $T$ and the root of $T$ does not form a descent with its parent (if there is one).

\begin{lemma}\label{lem: local limit when no descent and q to 0}
    Let $(q_n)_{n\ge1}$ be a sequence in $[0,1]$ such that $q_n \to 0$.
    Fix $r\ge1$ and a sequence $x^1, \dots, x^r$ of nonnegative integers.
    Define $y^j := x^1+\dots+x^j$ for each $j\in [r]$.
    Let $u_n$ denote a uniformly random vertex of $\cT_n^{(q_n)}$.
    Then,
    \[
    \prob{ x_{u_n}^1 = x^1 , \dots , x_{u_n}^r = x^r ,\, T_{u_n}^r \text{ creates no descent } }
    \cv{n\to\infty} \frac{1}{(y^r+1) y^r \cdots y^1} \,.
    \]
\end{lemma}

\begin{proof}
    In order to compute this probability, we need to count how many pairs $(t,v)$, where $t$ is a tree of size $n$ and $v$ is a vertex of $t$, satisfy the desired condition.
    This count, weighted by $q^\des{t}$, is denoted by $R_{n,q}\left( x^1,\dots,x^r \right)$.

    Such a pair is uniquely characterized by the following:
    choose a decreasing sequence of labels $\ell^1 > \dots > \ell^r$ for $v^1, \dots, v^r$, then a recursive tree $t^j$ of size $x^j$ rooted at $v^j$ for each $j\in[r]$, along with sets of labels, then choose a labeled tree $\check t$ of size $n-y^r$ with the remaining labels, and finally a vertex of $\check t$ with label greater than $\ell^r$ to which we graft $v^r$.

    The choice of the trees $t^j$ for $j\in[r]$ along with their labels has been enumerated in \cite[Appendix~A]{Dadedzi_Wagner_2024}.
    If we fix the label $\ell^r$ of $v^r$, we have
    \[
    \binom{n-\ell^r}{y^r-1} \frac{(y^r-1)!}{y^{r-1} y^{r-2} \cdots y^1} 
    \]
    combinations.
    Then, note that the labels $1$ to $\ell^r-1$ are necessarily in the tree $\check t$.
    Therefore, we have $\ell^r-1$ choices for the parent of $v^r$, and we can use the identity
    \[
    \sum_{\ell=1}^n \binom{n-\ell}{y^r-1} (\ell-1) = \binom{n}{y^r+1}
    \]
    as in \cite{Dadedzi_Wagner_2024}.
    Since the weighted enumeration of $\check t$ is $Z_{n-y^r,q}$, we finally obtain
    \[
    R_{n,q}\left( x^1,\dots,x^r \right) 
    = \binom{n}{y^r+1} \frac{(y^r-1)!}{y^{r-1} y^{r-2} \cdots y^1} Z_{n-y^r,q}   
    = Z_{n-y^r,q} \frac{n!}{(n-y^r-1)!} \frac{1}{(y^r+1) y^r y^{r-1} y^{r-2} \cdots y^1} \,.
    \]
Therefore, the probability that we wish to compute is
\begin{align*}
    \frac{R_{n,q}\left( x^1,\dots,x^r \right)}{n Z_{n,q}}
    = \frac{Z_{n-y^r,q}}{Z_{n,q}} \frac{(n-1)!}{(n-y^r-1)!} \frac{1}{(y^r+1) y^r y^{r-1} y^{r-2} \cdots y^1} \,.
\end{align*}
Since $\frac{Z_{n-y,q}}{Z_{n,q}} \sim n^{-y}$ for any fixed $y\ge0$ by \Cref{lem: ratio of partition functions when q to 0}, this converges to the announced limit.
\end{proof}

We can finally prove Proposition \ref{prop:local limit when q_n goes to 0}.

\begin{proof}[Proof of Proposition \ref{prop:local limit when q_n goes to 0}]
    Fix $r\ge1$ and a sequence $x^1, \dots, x^r$ of nonnegative integers.
    Define $y^j := x^1+\dots+x^j$ for each $j\in [r]$.
    Let $u_n$ denote a uniformly random vertex of $\cT_n^{(q_n)}$.
    Then, by \Cref{lem: local limit when no descent and q to 0},
    \begin{equation}\label{eq: liminf local limit when q to 0}
    \liminf_{n\to\infty} \prob{ x_{u_n}^1 = x^1 , \dots , x_{u_n}^r = x^r } 
    \ge \frac{1}{(y^r+1) y^r \cdots y^1} \,.
    \end{equation}
    The right-hand side defines a probability distribution, so by Fatou's lemma,
    \begin{align*}
        1
        = \sum_{x^1,\dots,x^r\ge0} \frac{1}{(y^r+1) y^r \cdots y^1}
        \le \liminf_{n\to\infty} \sum_{x^1,\dots,x^r\ge0} \prob{ x_{u_n}^1 = x^1 , \dots , x_{u_n}^r = x^r } 
        \le 1 \,.
    \end{align*}
    Hence, the $\liminf$ in \eqref{eq: liminf local limit when q to 0} is actually a limit and the inequality is an equality.
    Furthermore, this means that the tree $T_{u_n}^r$ has no descent with high probability and has the same limit distribution as in the recursive case.
    This concludes the proof.
\end{proof}

\appendix

\section{Real trees, dendrons, and Gromov--Prokhorov convergence}
\label{sec: Gromov Prokhorov}

\subsection{The deterministic setting}

We only consider rooted structures here, i.e.~with a distinguished point, but will often omit the term ``rooted''.

\emph{Real trees} are loopless geodesic spaces.
Formally, a real tree is a metric space $(T,d)$ which is complete, separable, and such that for any $(u,v)\in T^2$, there is an isometric path from $u$ to $v$ and all simple paths from $u$ to $v$ have the same range.
See, e.g., \cite{LeGall_2006}.
If $(T,d)$ is a real tree and $x\in T$, then $T\setminus\{x\}$ consists of a disjoint union of open connected components, which we call \emph{branches}.
We will always consider \emph{rooted} real trees $(T,\rho,d)$, that is, real trees $(T,d)$ with a distinguished point $\rho\in T$ called the root.
A \emph{measured real tree} is a real tree $(T,\rho,d)$ together with a Borel probability measure $\mu$ on $T$.
This measure is commonly assumed to be non-atomic and supported on the leaves of $T$, but we do not need any such assumption here.
Following \cite{Janson21}, the following generalizations of measured real trees are called \emph{dendrons}.

\begin{definition}[Rooted dendron]
\label{def:dendron}
A dendron $D=(T,\rho,d,\nu)$ is a real tree $(T,\rho,d)$ together with a Borel probability measure $\nu$ on $A_D := T \times [0,\infty)$, which satisfies in addition that $\nu(B \times [0,\infty))>0$ for any branch $B$ of the tree $T$.
We call $T$ the \emph{base tree} of the dendron $D$.
We say that two dendrons $D := (T,\rho,d,\nu)$ and $D' := (T',\rho',d',\nu')$ are isomorphic if there exists an isometry $f: (T,d) \to (T',d')$ such that $f(\rho)=\rho'$ and the mapping $(x,t) \mapsto (f(x),t)$ is measure-preserving from $(A_D,\nu)$ to $(A_{D'},\nu')$.
A dendron $D$ being given, we define the function $d_D: D^2 \rightarrow \mathbb{R}_+$ by
\begin{align*}
    d_D((x,a),(y,b)) := d(x,y)+a+b \text{ for all } (x,a),(y,b) \in D.
\end{align*}
\end{definition}

We emphasize that these objects are called \textit{long dendrons} in \cite{ET22} and \cite{Janson21}. 
However, as mentioned in \cite[Remark 3.8]{Janson21}, it is more natural to call them simply dendrons.
Note that any measured real tree $(T,\rho,d,\mu)$ can be seen as a dendron $(T,\rho,d,\nu)$, by letting $\nu$ be the push-forward of $\mu$ via $x\in T \mapsto (x,0) \in T\times[0,\infty)$.
Conversely, any dendron $D=(T,\rho,d,\nu)$ such that $\nu$ is supported on $T\times\{0\}$ is \enquote{actually just} a measured real tree.
For this reason, we only deal with dendrons below.

Each rooted discrete tree $(T,\rho)$ can be seen as a measured real tree $(T,\rho,d,\mu)$ where $d$ is the graph distance and $\mu$ is the uniform measure on vertices.
In the former, we keep the notation $(T,\rho,d)$ to denote the pointed metric space containing all vertices of $T$, and a path of unit length between each pair of neighbors.
Then, each discrete tree can be seen as a dendron.
The convergence of discrete trees towards a dendron (or towards a measured real tree) is therefore to be interpreted in the (pointed) \emph{Gromov--Prokhorov} (GP) topology introduced in \cite{Greven_Pfaffelhuber_Winter_2009} (called Gromov--weak therein).
In our case, it can be restated as follows:

\begin{definition}[Pointed convergence to a dendron]
\label{def: convergence to a fixed dendron}
    Let $(T_n,\rho_n)_{n\ge1}$ be a sequence of rooted finite trees, $(c_n)_{n\ge1}$ be positive numbers, and $D=(T,\rho,d,\nu)$ be a dendron.
    We say that the rescaled trees $c_n \cdot T_n := (T_n, \rho_n, c_n d_n, \mu_n)$ converge to the dendron $D$ if, for each $r\ge1$, we have the convergence in distribution of the following matrices:
    \begin{equation}\label{eq: cv of distance matrix for dendron convergence}
        \left( c_n d_n{\big( x_i^{(n)}, x_j^{(n)} \big)} \right)_{0\le i,j\le r}
        \cv{n\to\infty}
        \left( d_D{\left(z_i, z_j\right)} \One{i\ne j} \right)_{0\le i,j\le r}
    \end{equation}
    where $(z_i)_{i\ge1}$ are i.i.d.~random points in $A_D$ and $z_0 := \rho$, and for each $n\ge1$, $\big(x_i^{(n)}\big)_{i\ge1}$ are i.i.d.~random vertices in $T_n$ and $x_0^{(n)} := \rho_n$.
\end{definition}

%If the dendron $D$ is actually just a measured real tree, then this coincides with the standard GP convergence of trees.

\subsection{The random setting}

It is often the case that the sequence of discrete trees and the limit dendron (or measured real tree) are \emph{random}.
In that setting, one may wish to prove convergence \emph{in distribution} for the GP topology.
Technically, this would require the convergence \eqref{eq: cv of distance matrix for dendron convergence} to hold in distribution for the random laws of the random matrices (where the randomness of the law comes from the randomness of the space, and the randomness of the matrix comes from the randomness of the points therein).
Such a convergence, where randomness is dealt with in two steps, is usually called \emph{quenched}.
Proving a quenched convergence can be quite technical due to this two-steps procedure.

It is usually simpler to prove an \emph{annealed} convergence, that is, the convergence in distribution \eqref{eq: cv of distance matrix for dendron convergence} where the randomness of the spaces and of the points are considered \emph{simultaneously}.
Fortunately, it turns out that annealed convergence is sufficient to characterize convergence in distribution for the GP topology.

\begin{proposition}[Pointed convergence to a random dendron]
\label{prop: convergence to a random dendron}
    Let $(\cT_n, \rho_n)_{n\ge1}$ be a sequence of random rooted finite trees, $(c_n)_{n\ge1}$ be positive numbers, and $\cD=(\cT,\rho,d,\nu)$ be a random dendron.
    Then $c_n \cdot \cT_n$ converges in distribution to $\cD$ if and only if, for each $r\ge1$, we have the convergence in distribution of the following matrices:
    \begin{equation*}
        \left( c_n d_n{\big( x_i^{(n)}, x_j^{(n)} \big)} \right)_{0\le i,j\le r}
        \cv{n\to\infty}
        \left( d_\cD{\left(z_i, z_j\right)} \One{i\ne j} \right)_{0\le i,j\le r}
    \end{equation*}
    where $(z_i)_{i\ge1}$ are i.i.d.~random points in $A_\cD$ and $z_0 := \rho$, 
    %(conditionally given $\cD$), 
    and for each $n\ge1$, $\big(x_i^{(n)}\big)_{i\ge1}$ are i.i.d.~random vertices in $\cT_n$ and $x_0^{(n)} := \rho_n$.
    %(conditionally given $\cT_n$).
    That is, for each $r\ge1$ and each bounded continuous function $\phi$ on $\bR^{r\times r}$, we have
    \begin{equation*}
        \expec{ \expecond{ \phi\left(\left( c_n d_n{\big( x_i^{(n)}, x_j^{(n)} \big)} \right)_{0\le i,j\le r}\right) }{ \cT_n } }
        \cv{n\to\infty}
        \expec{ \expecond{ \phi\left(\left( d_\cD{\left(z_i, z_j\right)} \One{i\ne j} \right)_{0\le i,j\le r}\right) }{ \cD } } \,.
    \end{equation*}
\end{proposition}

This is a well-known result for measured real trees; see e.g.~\cite[Theorem~3(i)]{Aldous_1993} for a similar statement with a different formalism.
The ``leaf-tight property'' required there is automatically satisfied if the limit $\left( d_\cT{\left(z_i, z_j\right)} \right)_{1\le i,j\le r}$ is the random distance matrix of a random measured real tree, as can be seen with \cite[Theorem~3(ii)]{Aldous_1993}.

\begin{proof}
    If $D=(T,\rho,d,\nu)$ is a rooted dendron, write $\cM_D^r$ for the random matrix $\left( d_D{\big( z_i, z_j \big)} \One{i\ne j} \right)_{1\le i,j\le r}$, where $(z_i)_{i\ge2}$ are i.i.d.~random points in $A_D$ and $z_1 := \rho$. 
    Let $\cC_K(\bR^{r\times r})$ be the space of continuous functions with bounded support on $r\times r$ matrices, endowed with the topology of uniform convergence.
    A \emph{compact polynomial function} is any function $\Phi$ on the space of dendrons of the form
    \begin{equation*}
        \Phi(D) = \expec{ \phi{\left( \cM_D^r \right)} }\,,
    \end{equation*}
    where $r\ge1$ and $\phi \in \cC_K(\bR^{r\times r})$.
    By Definition~\ref{def: convergence to a fixed dendron}, a fixed sequence of finite trees $(T_n, \rho_n, c_n d_n, \mu_n)$ --- identified with a sequence of dendrons $D_n$ --- converges to a fixed dendron $D$ if and only if for each $r\ge1$, the random matrices $\cM_{D_n}^r$ converge weakly to the random matrix $\cM_D^r$.
    Since weak and vague convergence towards a given random variable are equivalent \cite[Lemma~6.21]{Kallenberg_2006}, this is equivalent to
    \[
        \Phi(D_n) \cv{n\to\infty} \Phi(D)
    \]
    for all compact polynomial functions $\Phi$.
    
    Now, recall that the spaces $\cC_K(\bR^{r\times r})$ are separable.
    Thus, a sequence of random finite trees $(\cT_n, \rho_n, c_n d_n, \mu_n)$ --- identified with a sequence of random dendrons $\cD_n$ --- converges in distribution to a random dendron $\cD$ if and only if we have the convergence in distribution
    \[
        \Phi(\cD_n) \cv{n\to\infty} \Phi(\cD)
    \]
    finitely jointly over compact polynomial functions $\Phi$.
    By boundedness, this is equivalent to the convergence of all joint moments, that is,
    \[
        \expec{ \prod_{i=1}^k \Phi_i(\cD_n) }
        \cv{n\to\infty}
        \expec{ \prod_{i=1}^k \Phi_i(\cD) }
    \]
    for each finite family $\Phi_1, \dots, \Phi_k$ of compact polynomial functions, with repetitions allowed.

    The last key ingredient is that the space of compact polynomial functions is multiplicatively stable.
    Indeed, if $D$ is a fixed dendron and $\phi_1 \in \cC_K(\bR^{r_1\times r_1})$, $\phi_2 \in \cC_K(\bR^{r_2\times r_2})$, then
    \begin{align*}
        \Phi_1(D) \Phi_2(D)
        &= \expec{ \phi_1{\left( \cM_D^{r_1} \right)} } \expec{ \phi_2{\left( \cM_D^{r_2} \right)} }
        \\&= \expec{ \phi_1{\left(\left( d_D{\big( z_i, z_j \big)} \One{i\ne j} \right)_{i,j\in \{1, 2, \dots, r_1\}}\right)} } \expec{ \phi_2{\left(\left( d_D{\big( z_i, z_j \big)} \One{i\ne j} \right)_{i,j\in \{1, r_1+1, \dots, r_1+r_2-1\}}\right)} }
        \\&= \expec{ \phi_1{\left(\left( d_D{\big( z_i, z_j \big)} \One{i\ne j} \right)_{i,j\in \{1, 2, \dots, r_1\}}\right)} \phi_2{\left(\left( d_D{\big( z_i, z_j \big)} \One{i\ne j} \right)_{i,j\in \{1, r_1+1, \dots, r_1+r_2-1\}}\right)} }
        \\&= \expec{ \phi_3{\left( \cM_D^{r_3} \right)} }\,,
    \end{align*}
    where $(z_i)_{i\ge2}$ are i.i.d.~random points on $A_D$, $z_1$ is the (deterministic) root of $D$, $r_3 := r_1 + r_2 - 1$, and $\phi_3\in \cC_K(\bR^{r_3\times r_3})$ is naturally defined to make the last equality hold.
    Therefore, for each finite family $\Phi_1, \dots, \Phi_k$ of compact polynomial functions (with repetitions allowed), the product $\Phi_1 \dots \Phi_k$ is a compact polynomial function $\Phi$.
    Hence, a random sequence of finite trees $(\cT_n, \rho_n, c_n d_n, \mu_n)$ converges to a random dendron $\cD$ if and only if 
    \[
        \expec{ \Phi(\cD_n) }
        \cv{n\to\infty} 
        \expec{ \Phi(\cD) }
    \]
    for each compact polynomial function $\Phi$.
    Observing that 
    \[
        \expec{ \Phi(\cD) }
        = \expec{ \expecond{ \phi{\left( \cM_\cD^r \right)} }{\cD} }
        = \expec{ \phi{\left( \cM_\cD^r \right)} }
    \]
    and likewise for $\cT_n$, this concludes the proof.
\end{proof}

This result was already observed in \cite[Corollary~2.8]{Loehr_2013} for metric-measure spaces; see also \cite[Section~2]{Bassino_Bouvel_Feray_Gerin_Pierrot_2022} for a similar discussion.
Although dendrons are not metric-measure spaces (because the function $d_D$ is not a metric), the proof is analogous.
This property is typical of topologies defined by the convergence of finite\footnote{
    What we mean here by ``finite'' is ``generated by a finite number of points''
}
substructures;
see, e.g., \cite[Theorem~3.1]{diaconis2008graph} for graphons and \cite[Corollary~4.6]{Bassino_Bouvel_Feray_Gerin_Pierrot_2018} for permutons.

\bibliographystyle{alphaurl}
\bibliography{bibli}

\end{document}